\documentclass[11pt]{article}
\usepackage{pstricks}
\usepackage[applemac]{inputenc}

\usepackage{amsmath,amsthm,amssymb}
\usepackage{mathrsfs}
\usepackage{graphicx,psfrag,epsfig}
\RequirePackage{pst-all}

\def\al{\alpha}
\def\be{\beta}
\def\ga{\gamma}
\def\Ga{{\Gamma}}
\def\de{\delta}

\def\De{\Delta}
\def\eps{{\varepsilon}}
\def\ka{\kappa}
\def\la{\lambda}

\def\om{\omega}
\def\Om{\Omega}

\def\sig{{\sigma}}
\def\Sig{{\Sigma}}

\def\th{{\theta}}
\def\Th{\Theta}

\def\ph{\varphi}
\def\ze{{\zeta}}

\newcommand{\un}[1]{{\underline{#1}}}
\def\ov{\overline}

\def\Id{\mathop{\hbox{{\rm Id}}}\nolimits}

\def\Im{\mathop{\Im m}\nolimits}
\def\Re{\mathop{\Re e}\nolimits}

\def\ov{\overline}
\def\ha{\widehat}

\def\pg{\leaders\hbox to 5mm{\hfil.\hfil}\hfill}

\def\text#1{\,\hbox{#1}\;}

\def\tst{\textstyle}

\def\Bbb{\mathbb}
\def\A{\mathbb A}
\def\R{\mathbb R}

\def\Z{\mathbb Z}

\def\T{\mathbb T}
\def\N{{\mathbb N}}

\def\abs#1{\left\vert#1\right\vert}
\def\norm#1{\Vert#1\Vert}
\def\Max{{\rm Max\,}}
\def\norm#1{\left\Vert#1\right\Vert}

\def\abs#1{\left\vert#1\right\vert}

\def\Im{{\rm Im\,}}
\def\Re{{\rm Re\,}}

\def\Sup{\mathop{\rm Sup\,}\limits}
\def\sup{{\rm Sup\,}}
\def\Inf{\mathop{\rm Inf\,}\limits}
\def\Max{\mathop{\rm Max\,}\limits}
\def\Min{\mathop{\rm Min\,}\limits}
\def\Log{{\rm Log\,}}

\def\sh{{\rm sh\,}}

\def\norm#1{\left\Vert#1\right\Vert}

\def\abs#1{\left\vert#1\right\vert}

\def\Dron#1#2{\Frac{\partial#1}{\partial#2}}

\def\un{\underline}

\def\til{\widetilde}
\def\ha{\widehat}
\def\setm{\setminus}
\def\d{\partial}
\def\ov{\overline}

\def\inv{^{-1}}

\newtheorem{thm}{Theorem}

\newtheorem{lemma}{Lemma}[section]

\newtheorem{prop}{Proposition}[section]

\newtheorem{cor}{Corollary}[section]

\newtheorem{Def}{Definition}[section]

\reversemarginpar

\newcounter{parag}[subsection]
\renewcommand{\theparag}{{\bf \thesubsection.\arabic{parag}}}
\newcommand{\parag}{\medskip \addtocounter{parag}{1} \noindent{\theparag\ } }

\newcounter{paraga}

\newcounter{pparag}

\def\Bbibitem#1#2{\bibitem[#1]{#2}}

\def\Frac#1#2{{{\displaystyle\strut#1}\over{\displaystyle\strut#2}}}
\def\frac#1#2{{\textstyle {#1\over#2}}}

\def\pdemi{{\tst{1\over2}}}

\def\be{\beta}

\def\rk{{\rm rank\,}}

\def\F{{\mathscr F}}
\def\F{{\mathscr F}}
\def\B{{\mathscr B}}
\def\P{{\mathscr P}}
\def\M{{\mathscr M}}

\def\J{{\mathscr J}}
\def\K{{\mathscr K}}
\def\I{{\mathscr I}}
\def\P{{\mathscr P}}

\def\jC{{\mathscr C}}
\def\jA{{\mathscr A}}
\def\jN{{\mathscr N}}

\def\jD{{\mathscr D}}
\def\jL{{\mathscr L}}

\def\jK{{\mathscr K}}
\def\O{{\mathscr O}}
\def\P{{\mathscr P}}
\def\jR{{\mathscr R}}
\def\U{{\mathscr U}}
\def\S{{\mathscr S}}
\def\jV{{\mathscr V}}

\def\S{{\mathscr S}}

\def\A{{\Bbb A}}
\def\Z{{\Bbb Z}}
\def\N{{\Bbb N}}
\def\R{{\Bbb R}}
\def\T{{\Bbb T}}
\def\BB{{\Bbb B}}

\def\bS{{\Bbb S}}

\def\htop{{\rm h_{top}}}

\def\htop{{\rm h_{\rm top}\,}}

\def\L{{\rm C\,}}
\def\LF{{\rm C^*\,}}
\def\LP{{\rm C^+}}

\def\un{\underline}

\def\diam{{\rm diam}}

\title{Dynamical complexity and symplectic integrability}

\author{Jean-Pierre Marco
\thanks{Universit\'e Paris 6.
Institut de Math\'ematiques de Jussieu, UMR 7586, {\em Analyse alg\'ebrique},
175 rue du Chevaleret, 75013 Paris.
E-mail: marco@math.jussieu.fr
}}
\date{}

\begin{document}

\maketitle

\begin{abstract}
We introduce two numerical conjugacy invariants for dynamical systems -- the complexity and weak complexity indices -- which are well-suited for the study of ``completely integrable'' Hamiltonian systems.  These invariants can be seen as ``slow entropies'', they describe the polynomial growth rate of the number of balls (for the usual ``dynamical'' distances) of  coverings of the ambient space.
We then define a new class of integrable systems, which we call decomposable systems, for which one can prove that the weak complexity index is smaller than the number of degrees of freedom. Hamiltonian systems integrable by means of non-degenerate integrals (in Eliasson-Williamson sense), subjected to natural additional assumptions, are the main examples of decomposable systems. We finally give explicit examples of  computation of the complexity index, for Morse Hamiltonian systems on surfaces and
for two-dimensional gradient systems.
\end{abstract}

\vskip3cm

\newpage

\tableofcontents

\newpage

\section{Introduction}
Symplectic geometry enables one to associate a Hamiltonian vector field with each regular function on a symplectic
manifold. It moreover provides a very efficient framework for proving the existence of first integrals for
such vector fields and perform specific reductions. In the so-called completely integrable cases this leads
to a geometric description of the orbits and their time parametrization in an open and dense subset of the phase space.

Such integrable systems seems to be dynamically simple, it actually turns out that for a large subclass of
integrable systems (but not all of them) the  classical  topological entropy vanishes. However, nobody doubts that the two fixed centers
problem has a more complex dynamical behaviour than the Kepler problem, or the geodesic flow on the triaxial ellipsoid is
more complex than the geodesic flow on the round sphere. A theory of {\em geometric} complexity has already been constructed
by Fomenko and his collaborators, our purpose here is to give a {\em dynamical} approach to integrable complexity and, to some extent,
analyze the relations between both approaches.

Our aim in the first part of this paper is first to introduce new tools which reveal additional structure for such systems.
While the topological entropy detects an exponential growth rate for the complexity of general dynamical systems,
for integrable systems we were led to analyze the growth rate at a {\em polynomial} scale. To this aim we introduce two
new distinct conjugacy invariants: the complexity index $\L$ and the weak complexity index $\L^*$. The values
of these invariants for the simplest possible Hamiltonian systems on the annulus $\A=\T\times\R$, endowed with the coordinates $(\th,r)$,
are depicted in the next figure (where $\ph$ is the time-one
map generated by the Hamiltonian flow).

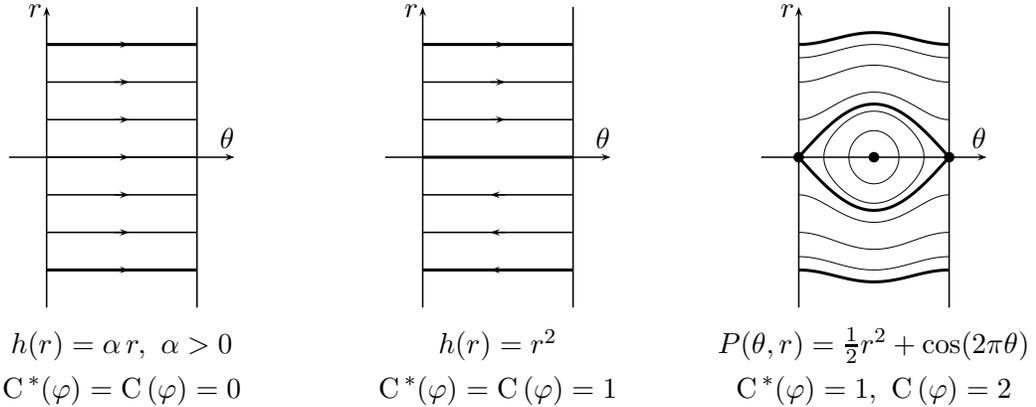
\begin{figure}[h]
\begin{center}
\begin{pspicture}(4cm,5.5cm)
\psset{xunit=.5cm,yunit=.5cm}
\rput(-6,6){
\psaxes[linewidth=.02,Dx=16,Dy=16]{->}(-2,0)(-3,-4)(3,4)
\psline[linewidth=.02](2,-4)(2,4)
\psline[linewidth=.04](-2,-3)(2,-3)
\psline[linewidth=.04](-2,3)(2,3)
\multido{\i=-2+1}{5} {%
\psline[linewidth=.02](-2,\i)(2,\i)}
\multido{\i=-3+1}{7} {%
\psline[linewidth=.02]{->}(-.2,\i)(.2,\i)}
\rput(2.8,.5){$\theta$}
\rput(-2.3,3.9){$r$}
\rput(0,-5){$h(r)=\al\, r,\ \al>0$}
\rput(0,-6.2){$\L^*(\ph)=\L(\ph)=0$}
}
\rput(4,6){
\psaxes[linewidth=.02,Dx=16,Dy=16]{->}(-2,0)(-3,-4)(3,4)
\psline[linewidth=.02](2,-4)(2,4)
\psline[linewidth=.04](-2,-3)(2,-3)
\psline[linewidth=.04](-2,3)(2,3)
\psline[linewidth=.04](-2,0)(2,0)
\multido{\i=-3+1}{7} {%
\psline[linewidth=.02](-2,\i)(2,\i)}
\multido{\i=-3+1}{3} {%
\psline[linewidth=.02]{<-}(-.2,\i)(.2,\i)}
\multido{\i=1+1}{3} {%
\psline[linewidth=.02]{->}(-.2,\i)(.2,\i)}
\rput(2.8,.5){$\theta$}
\rput(-2.3,3.9){$r$}
\rput(0,-5){$h(r)=r^2$}
\rput(0,-6.2){$\L^*(\ph)=\L(\ph)=1$}
}
\rput(14,6){
\psaxes[linewidth=.02,Dx=16,Dy=16]{->}(-2,0)(-3,-4)(3,4)
\psline[linewidth=.02](2,-4)(2,4)
\parametricplot[linewidth=.04,plotpoints=200]{-180}{180}{t  90 div   t cos 10 add sqrt}
\parametricplot[linewidth=.04,plotpoints=200]{-180}{180}{t  90 div   t cos 10 add sqrt -1 mul}
\parametricplot[linewidth=.01,plotpoints=200]{-180}{180}{t  90 div   t cos 8 add sqrt}
\parametricplot[linewidth=.01,plotpoints=200]{-180}{180}{t  90 div   t cos 8 add sqrt -1 mul}
\parametricplot[linewidth=.01,plotpoints=200]{-180}{180}{t  90 div   t cos 5 add sqrt}
\parametricplot[linewidth=.01,plotpoints=200]{-180}{180}{t  90 div   t cos 5 add sqrt -1 mul}
\parametricplot[linewidth=.01,plotpoints=200]{-180}{180}{t  90 div   t cos 2 add sqrt}
\parametricplot[linewidth=.01,plotpoints=200]{-180}{180}{t  90 div   t cos 2 add sqrt -1 mul}
\parametricplot[linewidth=.04,plotpoints=200]{-180}{180}{t  90 div   t cos 1 add sqrt}
\parametricplot[linewidth=.04,plotpoints=200]{-180}{180}{t  90 div   t cos 1 add sqrt -1 mul}
\parametricplot[linewidth=.01,plotpoints=200]{-120}{120}{t  90 div   t cos .5 add sqrt}
\parametricplot[linewidth=.01,plotpoints=200]{-120}{120}{t  90 div   t cos .5 add sqrt -1 mul}
\parametricplot[linewidth=.01,plotpoints=200]{-60}{60}{t  90 div   t cos -.5 add sqrt}
\parametricplot[linewidth=.01,plotpoints=200]{-60}{60}{t  90 div   t cos -.5 add sqrt  -1 mul}
\pscircle[fillstyle=solid,fillcolor=black](0,0){.07}
\pscircle[fillstyle=solid,fillcolor=black](-2,0){.07}
\pscircle[fillstyle=solid,fillcolor=black](2,0){.07}
\rput(2.8,.5){$\theta$}
\rput(-2.3,3.9){$r$}
\rput(0,-5){$P(\th,r)=\pdemi r^2+\cos(2\pi\th)$}
\rput(0,-6.2){$\L^*(\ph)=1,\ \L(\ph)=2$}
}
\end{pspicture}
\end{center}
\end{figure}

The two invariants are very similar in nature, but their definitions slightly differ and as a rule they take very different values
in general, even for very simple systems. However, they always satisfy the inequality $\L^*\leq \L$.  The invariant $\L^*$ enjoys more structure properties
than $\L$ which make it much easier to determine, while $\L$ is finer than $\L^*$ and discriminates between integrable behaviours: it actually
takes different values for the three examples above.

The second and biggest part of the paper is devoted to the relationship between our indices and the integrability properties of Hamiltonian systems. As both indices are infinite when the topological entropy is positive, our first
task is to give precise constraints under which the entropy of integrable systems vanishes: this yields the notion of strong
integrability (which first appeared informally in \cite{P}).

We then define a natural subclass of (strongly) integrable systems, which we call decomposable, for which the $\L^*$ index admits
{\em a priori} upper bounds: $\L^*$ is less than the number of degrees of freedom of a decomposable system. Most of known
examples of integrable Hamiltonian systems fall into this class (obviously, it is in particular the case for the three examples above),
so one can see our definition  as a natural
one for practical use. As a consequence, computing the weak complexity index may give rise to an obstruction theory
for practical integrability.

The second invariant $\L$ enables one to construct a first complexity scale for decomposable systems. As a preliminary study
we analyze its behaviour on a very simple class of systems, generated by Morse Hamiltonian functions on compact symplectic
surfaces with boundary. It turns out that for such systems the index $\L$ can take only three integer values: $0$ if
the system is conjugated to our first example, $1$ if it is conjugated to our second one, and $2$ if the Hamiltonian function admits
singular points of Morse index $1$.

The behaviour of $\L$ is drastically different from that of the topological entropy, and in particular do not only depend on the
restriction of the system on the non-wandering domain. To emphasize this new aspect of complexity, we also introduce a class
of two-dimensional systems with gradientlike behaviour, and prove that their complexity index takes the same values as in
the Hamiltonian case.

To keep this paper a reasonable length and get rid of many technical details, we limit ourselves here to these particular examples.
Two subsequent papers will be devoted to the extensive study of the complexity indices of higher dimensional non-degenerate integrable systems
and gradient systems.
In the rest of this introduction, we review the various necessary notions before stating our main results more precisely.

\vskip2mm

{\bf 1. Symplectic notions.}  Let $(M,\Om)$ be a symplectic manifold of dimension $2\ell$. Given a Hamiltonian function
$H\in C^\infty(M,\R)$,
the Hamiltonian vector field $X_H$ is the symplectic gradient of $H$, usually defined by the equality
$
i_{X_H}\Om=-dH.
$
The Poisson bracket associated with the symplectic form $\Om$ is defined by
$
\{f,g\}=\Om(X_f,X_g)
$
for any pair of $C^\infty$ functions $(f,g)$ on $M$.

We will denote by
$\A^\ell=\T^*\T^\ell=\T^\ell\times\R^\ell$ the standard annulus, equipped with the {\em angle-action} coordinates
$\th\in\T^\ell$ and $r\in\R^\ell$, and the  symplectic form $\Om_0=\sum_{k=1}^\ell dr_k\wedge d\th_k$.
In the following
we will frequently deal with Hamiltonian functions on (subsets of) $\A^\ell$ which depend only on the action variable
$r$, such systems are said to be in {\em action-angle form}. If $h:O\subset\R^n\to\R$ is in action-angle form, for $t\in \R$
the time--$t$ diffeomorphism generated by its
Hamiltonian flow is well-defined on $\T^n\times O$ and reads
\begin{equation}
(\th,r)\longmapsto\Big(\th+t\om(r)\  [{\rm mod}\,\ \Z^\ell], r\Big)
\end{equation}
with $t\in\R$ and $\om(r)=\d_rh(r)\in\R^\ell$.

We say that a smooth map $F=(f_i)_{1\leq i\leq \ell}:M\to \R^\ell$ is an {\em integral map}
when its components are in involution, that is
$
\{f_i,f_j\}=0
$
for $1\leq i,j\leq\ell$. Given a smooth Hamiltonian function $H:M\to\R$, we say that $F:M\to\R^\ell$ is an {\em integral
for $H$} when it is an integral map whose components are in involution with $H$.

If $F$ is an integral for $H$, then the classical Liouville-Mineur-Arnold theorem (or action-angle theorem for short) shows that for each compact
connected component $T$ of a regular level set of $F$, there exists a neighborhood $N$ and a symplectic diffeomorphism
$\Psi$ from a neighborhood of $\T^\ell\times\{0\}$ in $\A^\ell$ to $N$ such that $H\circ \Psi$ is in action-angle form.
Up to diffeomorphism, the Hamiltonian flow on $N$ is therefore immediately integrated and exhibits a very simple
dynamical behaviour.

\vskip2mm

{\bf 2. Topological entropy.}  Let us now recall the definition and some basic facts on the topological entropy.
Let $(X,d)$ be a compact metric space and let $\ph:X\to X$ be a continuous map.
For each integer $n\geq 1$ one defines the dynamical metric
\begin{equation}\label{Eq:metricph}
d_n^\ph(x,y)=\Max_{0\leq k\leq n-1} d\big(\ph^k(x),\ph^k(y)\big).
\end{equation}
It is easy to see that all the metrics $d_n^\ph$ define the same topology on $X$. In particular $(X,d_n^\ph)$
is compact and therefore, for any $\eps>0$, $X$ can be covered by a finite number of balls $B_n(x,\eps)$ of
radius $\eps$ for $d_n^\ph$. Let $G_n(\eps)$ be the minimal number of balls in such a covering.
Then  the topological entropy of $\ph$  is
\begin{equation}
\htop(\ph)
=\Sup_{\eps >0}\limsup_{n\to\infty} \Frac{1}{n}\Log G_n(\eps)
=\lim_{\eps\to 0}\limsup_{n\to\infty} \Frac{1}{n}\Log G_n(\eps).
\end{equation}
The topological entropy therefore detects the exponential growth rate of the minimal number of initial conditions
which are necessary to follow the $n$ first iterates of any point of the space within a precision of $\eps$ (more exactly
it is the limit of this growth rate when $\eps\to 0$). One can also define in the same way the topological entropy $\htop(\ph,Y)$ of $\ph$
on any (not necessarily invariant) subset $Y$ of $X$. See \cite{HK} and \cite{Pe} for more details.
\vskip1mm

The topological entropy enjoys several naturality properties  which we briefly recall here to allow us to compare with the properties of the complexity indices.

\vskip1mm\noindent
-- {\em Invariance.} $\htop$ is a $C^0$ conjugacy invariant and does not depend on the choice of topologically equivalent metrics on $X$.

\vskip1mm\noindent
-- {\em Factors.} $\htop(\ph,X)\geq \htop(\ph',X')$ when $(X',\ph')$ is a factor of $(X,\ph)$.

\vskip1mm\noindent
-- {\em Restriction.} When $Y$ is invariant under $\ph$, $\htop(\ph,Y)=\htop(\ph_{\vert Y})$.

\vskip1mm\noindent
-- {\em Monotonicity.} $\htop(\ph,Y)\leq \htop(\ph,Y')$ when $Y\subset Y'$.

\vskip1mm\noindent
-- {\em Transport.} $\htop(\ph,Y)=\htop(\ph,\ph(Y))$ when $\ph$ is a homeomorphism.
\vskip1mm

It moreover satisfies additional properties which will be crucial for our purposes.

\vskip1mm

The first one is the {\em $\sig$--union property}, which states
that the topological entropy of a continuous map on a countable union of invariant subsets is the supremum of the topological entropies
of the map on the subsets.

\vskip1mm

The second one is the so-called {\em variational principle}, which states that the topological entropy of a homeomorphism
on a compact space is the supremum of the metric entropies relative to ergodic invariant measures.

\vskip1mm

The third one is the {\em Bowen formula} (\cite{B}): if $\ph: X\to X$ and $\ph':X'\to X'$ are continuous and
if there exists a continuous surjective map $\pi:X\to X'$ which semi-conjugates $\ph$ and $\ph'$ (that is $\ph'$ is a factor of $\ph$),
then
$$
\htop(\ph)\leq\Sup_{x'\in X'} \htop(\ph, \pi\inv\{x'\}).
$$

It is not difficult to see that the topological entropy of a Hamiltonian system in action-angle form on a compact
subset of $\A^\ell$ is zero. One way to prove this is to remark that such a system admits an invariant foliation by Lagrangian tori on which the restriction is an isometry (and therefore has zero entropy) and use Bowen's formula, or the variational principle.

 Using countable covering arguments together with the last remark, the $\sig$--union property and the action-angle theorem prove that the topological entropy
of a completely integrable system on the complement of the singular set of its integral (that is the inverse image of the set of critical values) vanishes. So the topological entropy of such
systems is localized on the singular set of their integral maps.
Still,  it is possible to exhibit examples of smooth {\em geodesic} systems, or more precisely duals of such systems,
on the cotangent bundle of Riemannian manifolds,
which  possess smooth integrals which are regular on open and dense subsets, and whose flow nevertheless has {\em positive} topological entropy  (even when  restricted to the unit cotangent bundle),
see \cite{BT1,BT2}.

\vskip2mm

{\bf 3. Organization and main results of the paper.} The vanishing of the topological entropy of action-angle systems clearly proves that
it sees nothing of the transverse structure of a Lagrangian foliation. Our first goal is to construct finer invariants for which this structure
becomes apparent.
 It turns out that the {\em polynomial} growth rate of the quantity
$G_n(\eps)$ defined above is well-defined for these systems, and enjoys very interesting properties.
To be more precise, with the same notation as above, we define the {\em complexity index} of $\ph$ as the
quantity
\begin{equation}\label{Eq:freedin}
\L(\ph)=\Sup_{\eps >0}\Inf\Big\{\sig\geq0\mid \lim_{n\to\infty}\Frac{1}{n^\sig}\,  G_n(\eps)=0\Big\}.
\end{equation}
We will prove that for systems in action-angle form $h:O\subset\R^\ell\to\R$ the value of the complexity index
is exactly equal to the maximal rank of the Hessian of $h$ on $O$. For these systems, the index $\L$ therefore detects the
``effective'' number of degrees of freedom.

Section 2 is devoted to the extensive study of the properties of the complexity index $\L$, and to the introduction of another closely
related one, the weak complexity index $\L^*$. We closely follow a general approach developed by Pesin in \cite{Pe} for
lower and upper dimension capacities.

Both indices satisfy the naturality properties quoted above for the topological entropy. Moreover $\L^*\leq \L$, but they generally take different values, even for very simple systems (for instance, a gradient system on a segment, see proposition \ref{prop:indseg}).
However, a striking fact is that they {\em coincide} for action-angle Hamiltonian systems.

As a consequence, both indices do not satisfy any kind of variational principle, nor any analog of Bowen's formula (otherwise they would
vanish for action-angle systems). The main question is therefore to know whether they enjoy a $\sig$--union property. It turns out that
 only the weak index $\L^*$ admits such a $\sig$-union principle (which is in strong contrast with analogous
constructions for exponential growth rates, see \cite{Pe}). We therefore take advantage of this major difference between $\L$ and $\L^*$ to
obtain two different approaches of the notion of complexity of integrable systems.

The examples by Bolsinov and Ta\"\i manov prove that it is necessary to introduce some additional constraints to be able to control the global topological entropy of integrable systems in the $C^\infty$--class (and even in the Gevrey class).
There are several possible ones, some of them being of local nature (the non-degeneracy conditions of Ito and Eliasson, see Section 3),
other ones being semi-global.
The mildest of these global conditions was introduced by G. Paternain in  \cite{P},  where it underscores the whole approach
without deserving any particular terminology. In some respects it may be compared with Ta\"\i manov's notion of
``tame integrability'' for geodesic flows \cite{T},  even if it largely differs from this latter one.
In Section 3 of this paper we will give a formal definition for this condition, which we call {\em strong integrability}.
Paternain proved in \cite{P} that if a smooth Hamiltonian is strongly integrable,
then the topological entropy of its flow vanishes. The proof makes a crucial use of the variational principle, a slightly different version
will be given in Section 3.
Is is easy to see that an integral which satisfies the nondegeneracy conditions of Ito
or Eliasson also satisfies the previous strong integrability condition. Therefore  a great amount of examples of Hamiltonian systems
with zero topological entropy is at our disposal, which legitimates our attempt to say more about their dynamical complexity.

Still in Section 3, we introduce a refinement of the notion of strong integrability, which we call {\em decomposability}, and
prove that the weak complexity $\L^*$ of decomposable systems is upper bounded by their number of degrees of freedom (theorem \ref{th:decomp})
This way, the computation of the numerical invariant provides us with a new tool for proving obstructions to ``integrability''.
We then give sufficient conditions for a system with a non-degenerate integral to be decomposable. Again, many classical
examples prove to be decomposable (a general study of decomposability of classical systems will be the subject of a subsequent paper).

The lack of $\sig$--union property for the index $\L$ makes it much more difficult to determine (and so probably much richer) than $\L^*$. In particular, surprisigly enough, in the Hamiltonian case its value is not completely encoded by the infinity jet of the system at the singular set, but instead by its germ. In Section 4 and Section 5, we therefore limit
ourselves here to the easiest  we examine the behaviour of the complexity index $\L$ of simple systems on surfaces: Morse non-degenerate
Hamiltonian systems on symplectic surfaces in Section 4 (theorem \ref{thm:indham}), and particular gradient systems in the plane in Section 5 (theorem \ref{thm:indgrad}). This can be seen as a reasonably non technical introduction to more elaborated further work.

To conclude this introduction, let us mention the various interesting relations between the complexity indices and other complexity
measurements, such as for instance (a weak version of) Lyapounov exponents or the asymptotic behaviour of the number of orbits
connecting two points in Rienannian geometry (and notably the integrable cases of the multidimensional ellipsoids). Also many cases
of geodesic flows with zero topological entropy are known (the geodesic flows on rationally elliptic manifolds for instance), which
give rise to new problems at the polyomial level. Again we refer  to further work for these questions.

\vskip.5cm

\noindent{\bf Acknowledgements.} I wish to thank Laurent Lazzarini for numerous helpful conversations, notably on the determination
of the complexity indices on surfaces, and Eva Miranda for many stimulating discussions about non-degenerate singularities of integral maps.
I also thank Clémence Labrousse for a careful reading of the first draft.

The preparation of this paper was motivated and made possible by the
rich interaction initiated by the ANR {\em Int\'egrabilit\'e
r\'eelle et complexe en M\'ecanique Hamiltonienne} (JC05\_41465). I
wish to thank Alexei Tsigvintsev for the organization.


\section{The complexity indices}
In this section we introduce the two complexity indices $\LF$ and $\L$. Our approach is based on \cite{Pe}. We
state and prove their main properties and analyze their behaviour for two ``test'' systems:  gradient flows on the
segment and action-angle Hamiltonian systems.

We denote by $\N$ the set of non-negative integers and by $\N^*$ the set of positive ones.
Given a compact metric space $(X,d)$ together with a continuous map $\ph:X\to X$,   for $n\in\N^*$, we denote by $d_n^\ph$ the
 dynamical distances  associated with $\ph$, defined in (\ref{Eq:metricph}).  Note that $d_1^\ph=d$.

When there is no risk of confusion, the ball centered at $x\in X$ and of radius $\eps$ for
$d_n^\ph$ is denoted by $B_n(x,\eps)$. For each $\eps>0$, we consider the set
$$
\B_\eps=\{B_n(x,\eps)\mid (x,n)\in X\times\N\}
$$
of all open balls of radius $\eps$ for the distances $d_n^\ph$. In the following we also say that such a ball $B_n(x,\eps)$
is an $(n,\eps)$-ball.

One easily sees that the metric spaces $(X,d_n^\ph)$ are compact. For $Y\subset X$ we denote by $G_n(Y,\eps)<+\infty$ the
minimal number of $(n,\eps)$-balls
in a finite covering of $Y$ (note that the centers of the balls do not necessarily belong to $Y$, and that we do not require $Y$ to be invariant
under $\ph$).
We say that a (necessarily finite) subset $S$ of $Y$ is $(n,\eps)$--separated when for each pair $a,a'$ of elements of
$Y$ with $a\neq a'$, then $d_n^\ph(a,a')\geq\eps$. We denote by $S_n(Y,\eps)$ the maximal cardinality of an $(n,\eps)$--separated
subset of $Y$. Clearly
$$
G_n(Y,\eps/2)\geq S_n(Y,\eps)\geq G_n(Y,\eps).
$$
We abbreviate $G_n(X,\eps)$ and $S_n(X,\eps)$ in $G_n(\eps)$ and $S_n(\eps)$ respectively.


\subsection{The weak complexity index}
We consider a compact metric space $(X,\ph)$ together with a continuous map $\ph:X\to X$.

\parag
Given a subset $Y$ of $X$ (not necessariliy $\ph$-invariant), for $\eps>0$ we denote by $\jC(Y,\eps)$ the set of all coverings of $Y$ by balls of
$\B_\eps$,
so an element
of $\jC(Y,\eps)$ is a family $\big(B_{n_i}(x_i,\eps)\big)_{i\in I}$ of $(n_i,\eps)$-balls such that
$$
Y\subset \bigcup_{i\in I}B_{n_i}(x_i,\eps).
$$
 Again, we do not require that $x_i\in Y$.
Given $N\in \N^*$, we denote by $\jC_{\geq N}(Y,\eps)$ the subset of $\jC(Y,\eps)$ formed by the coverings
$\big(B_{n_i}(x_i,\eps)\big)_{i\in I}$ of $Y$ for which $n_i\geq N$.

\parag
Let the subset $Y\subset X$ be given and fix $\eps>0$.
Given an element $C=\big(B_{n_i}(x_i,\eps)\big)_{i\in I}$
in $\jC(Y,\eps)$ and a non-negative real parameter $s$, we set
$$
M(C,s)=\sum_{i\in I}\Big(\Frac{1}{n_i}\Big)^s\in[0,+\infty].
$$
Note that $M(C,s)$ depends on the family $C$ and not only on its image (the set of balls $B_{n_i}(x_i,\eps)$), actually, it is possible
that a same balls admits several representations of the form $B_{n_i}(x_i,\eps)$. This will cause no trouble in the following.
Let $N\in\N^*$. Since there exists finite coverings for $X$ by $(N,\eps)$--balls, there exists finite coverings $C\in\jC_{\geq N}(Y,\eps)$,
we set
$$
\de(Y,\eps,s,N)=\Inf\Big\{M(C,s)\mid C\in \jC_{\geq N}(Y,\eps)\Big\}\in[0,+\infty[.
$$
Obviously $\de(Y,\eps,s,N)\leq\de(Y,\eps,s,N')$ when $N\leq N'$, so one can define
$$
\De(Y,\eps,s)=\lim_{N\to+\infty} \de(Y,\eps,s,N)=\Sup_{N\in\N} \de(Y,\eps,s,N)\in[0,+\infty].
$$
The definition of the weak complexity index will be based on the following lemma.

\begin{lemma}\label{lem:critval}
There exists a unique critical value $s_c(Y,\eps)$ such that
\begin{equation}\label{eq:critval}
\De(Y,\eps,s)=0\ \text{if}\ s>s_c(Y,\eps)
\quad \text{and}\quad
\De(Y,\eps,s)=+\infty\ \text{if}\ s<s_c(Y,\eps).
\end{equation}
\end{lemma}

\begin{proof} Assume that $\De(Y,\eps,s)<+\infty$ for a given value of $s\in[0,+\infty[$, and fix $s'>s$.
Let $C=\big(B_{n_i}(x_i,\eps)\big)_{i\in I}$ be a  covering  in $\jC_{\geq N}(Y,\eps)$. Then:
$$
M(C,s')=\sum_{i\in I} \Big(\Frac{1}{n_i}\Big)^{s'}=\sum_{i\in I} \Big(\Frac{1}{n_i}\Big)^{s}\Big(\Frac{1}{n_i}\Big)^{s'-s}
\leq
\Big(\Frac{1}{N}\Big)^{s'-s}M(C,s).
$$
By definition, for every $N\in\N^*$, there exists a covering $C_N\in\jC_{\geq N}(Y,\eps)$ such that
$$
M(C_N,s)\leq  \De(Y,\eps,s)+1
$$
and therefore
$$
\de(Y,\eps,s',N)\leq M(C_N,s')\leq \Big(\Frac{1}{N}\Big)^{s'-s}\Big(\De(Y,\eps,s)+1\Big)
$$
which shows that $\De(Y,\eps,s')=\lim_{N\to\infty}\de(Y,\eps,s',N)=0$. This proves that the set of points $s$ such that
$0<\De(Y,\eps,s)<+\infty$ contains at most one element. One also sees that the set of  $s$ such that $\De(Y,\eps,s)=0$
is an interval, of the form $]a,+\infty[$ or $[a,+\infty[$, with $a\geq -\infty$. Analogously, that the set of $s$ such that $\De(Y,\eps,s)=+\infty$ is
of the form $]-\infty,a[$ or $]-\infty,a]$. Therefore
$$
s_c(Y,\eps)=\Inf\{s\in [0,+\infty]\mid \De(Y,\eps,0)=0\}=\Sup\{s\in [0,+\infty]\mid \De(Y,\eps,0)=+\infty\}
$$
(with the obvious convention on $\Inf$ and $\Sup$ of the empty set)
satisfies conditions (\ref{eq:critval}). Uniqueness is then obvious.
\end{proof}

\parag Remark now that $s_c(Y,\eps)\leq s_c(Y,\eps')$ when $\eps'<\eps$. This allows us to state the following definition.

\begin{Def}
We define the  {\em weak complexity index} $\LF(\ph,Y)$ of $\ph$ on the subset $Y$ as the limit of the critical
value $s_c(Y,\eps)$ when $\eps$ goes to $0$:
$$
\LF(\ph,Y):=\lim_{\eps\to 0} s_c(Y,\eps)=\Sup_{\eps>0}s_c(Y,\eps)\in[0,+\infty].
$$
\end{Def}


\subsection{The complexity index}
We now define the complexity index in much the same way as before, as well as other essentially equivalent quantities.

\parag
We keep the notation of the previous section. Given $Y\subset X$ and $N\in\N^*$, we now denote by $\jC_{= N}(Y,\eps)$
the set of all coverings of $Y$ of the form $\big(B_{N}(x_i,\eps)\big)_{i\in I}$, so now the balls all have the same order $N$.
Clearly $\jC_{= N}(Y,\eps)\subset\jC_{\geq N}(Y,\eps)$.

Given $Y\subset X$ and $\eps>0$, $s\geq 0$ and $N\in\N$, we set
$$
\ga(Y,\eps,s,N)=\Inf\big\{M(C,s)\mid C\in \jC_{= N}(Y,\eps)\big\}=G_N(Y,\eps) \Big(\Frac{1}{N}\Big)^s.
$$
Note that $\ga(Y,\eps,s,N)$ may have no limit when $N\to\infty$, so
we are now led to introduce {\em two} limiting quantities:
$$
\un \Ga(Y,\eps,s)=\liminf_{N\to+\infty}\ga(Y,\eps,s,N),\qquad\ov \Ga(Y,\eps,s)=\limsup_{N\to+\infty}\ga(Y,\eps,s,N).
$$
As in Lemma \ref{lem:critval}, one checks that there exists  critical values $\un s_c(Y,\eps), \ov s_c(Y,\eps)$
such that
$$
\begin{array}{ll}
\phantom{\Frac{A^A}{A^A}}&\un \Ga(Y,\eps,s)=0\ \text{if}\ s>\un s_c(Y,\eps)
\quad \text{and}\quad
\un \Ga(Y,\eps,s)=+\infty\ \text{if}\ s<\un s_c(Y,\eps);\phantom{\Frac{A^A}{A^A}}\\
&\ov \Ga(Y,\eps,s)=0\ \text{if}\ s>\ov s_c(Y,\eps)
\quad \text{and}\quad
\ov \Ga(Y,\eps,s)=+\infty\ \text{if}\ s<\ov s_c(Y,\eps),\\
\end{array}
$$

The following lemma is immediate.
\begin{lemma} One has the inequality $\ov s_c(Y,\eps)\geq \un s_c(Y,\eps)$ and the following properties hold true
$$
\begin{array}{lll}
\ov s_c(Y,\eps)=\Inf\{\sig\geq0\mid \lim_{n\to\infty}\Frac{1}{n^\sig}G_n(Y,\eps)=0\},\\
\un s_c(Y,\eps)=\Sup\{\sig\geq0\mid \lim_{n\to\infty}\Frac{1}{n^\sig}G_n(Y,\eps)=+\infty\}.
\end{array}
$$
\end{lemma}

\parag As in the previous section one sees that $\ov s_c(Y,\eps)$ and $\un s_c(Y,\eps)$ are monotone non-increasing functions of $\eps$.
We define the  {\em upper and lower complexity indices} $\ov\L(\ph,Y)$ and $\un\L(\ph,Y)$ of $\ph$ on the subset $Y$
as the following limits:
$$
\ov\L(\ph,Y):=\lim_{\eps\to 0} \ov s_c(Y,\eps),\qquad \un\L(\ph,Y):=\lim_{\eps\to 0} \un s_c(Y,\eps).
$$
One could also define complexity indices by means of the following limits:
$$
\L^{\bullet}(\ph,Y)=\lim_{\eps\to 0}\limsup_{n\to\infty}\Frac{\Log G_n(Y,\eps)}{\Log n},
\qquad
\L_{\bullet}(\ph,Y)=\lim_{\eps\to 0}\liminf_{n\to\infty}\Frac{\Log G_n(Y,\eps)}{\Log n}.
$$

\begin{lemma} The following relations hold true
$$
\L^*(\ph,Y)\leq \un \L(\ph,Y)=\L_{\bullet}(\ph,Y)\leq \ov \L(\ph,Y)=\L^{\bullet}(\ph,Y).
$$
\end{lemma}

\begin{proof}
The two equalities $\un \L(\ph,Y)=\L_{\bullet}(\ph,Y)$ and $\ov \L(\ph,Y)=\L^{\bullet}(\ph,Y)$ are direct consequences of
the equalities
$$
\liminf_{n\to\infty}\Frac{\Log G_n(Y,\eps)}{\Log n}=\un s_c(Y,\eps),\qquad
\limsup_{n\to\infty}\Frac{\Log G_n(Y,\eps)}{\Log n}=\ov s_c(Y,\eps),
$$
valid for all $\eps>0$, the proof of which are easy exercises.
The inequality $\L_{\bullet}(\ph,Y)\leq \L^{\bullet}(\ph,Y)$ is immediate. It only remains to check the first inequality, which
comes from the inclusion $\jC_{=N}(Y,\eps)\subset \jC_{\geq N}(Y,\eps)$, which immediately yields $\de(Y,\eps,s,N)\leq\ga(Y,\eps,s,N)$
and therefore $\Delta(Y,\eps,s)\leq \un\Ga(Y,\eps,s)$.
\end{proof}

\parag
It turns out that the lower and upper indices $\un \L$ and $\ov \L$  essentially exhibit the same behaviour in our examples, so we will
mainly focus on the upper index $\ov \L$ and introduce the following abbreviate definition.

\begin{Def} With the previous assumptions and notation, we define the {\em complexity index} of $\ph$ on the subset $Y$ as
$$
\L(\ph,Y):=\ov\L(\ph,Y)=\L^{\bullet}(\ph,Y).
$$
\end{Def}

In the following, we nevertheless indicate some properties of the lower index too, when they are straightforward.


\subsection{Main properties of the complexity indices}
We begin with the naturalness properties shared by all indices (as well as by the topological entropy).
In the following proposition the symbol $\LP$ indifferently stands for $\L$, $\L^*$ or $\un \L$. When necessary,
we recall the metric on the ambient space with a subscript.

\vskip3mm

\begin{prop}\label{prop:naturality} {\bf (Naturalness). }
 Let $(X,d)$ be compact and $\ph:X\to X$ be a continuous map. Then
 the following properties hold true.

\vskip1mm
\noindent {\bf (1) Invariance.}  $\LP$ is a $C^0$ conjugacy invariant and does not depend on the choice of topologically equivalent
metrics on $X$.

\vskip2mm

\noindent {\bf (2) Factors.} If $(X',d')$ is another compact metric space and if $\psi:X'\to X'$ is a factor of $\ph$,
that is if there exists a continuous surjective map $h:X\to X'$ such that $\psi\circ h=h\circ \ph$,
then
$$
\L^+(\ph)\geq\L^+(\psi).
$$

\vskip2mm

\noindent {\bf (3) Restriction.} If $Y\subset X$ is {\em invariant} under $\ph$ and endowed with the induced metric, then
$$
\LP_{\!\!  d}(\ph,Y)=\LP_{\!\! \ha d}(\ph_{\vert Y}).
$$
where $\ha d$ stands for the induced metric on $Y$.
\vskip2mm

\noindent {\bf (4) Monotonicity.}
If $Y\subset Y'$ are two subsets of $X$, then
$$
\LP(\ph,Y)\leq \LP(\ph,Y').
$$

\vskip2mm

\noindent {\bf (5) Transport.} For any $Y\subset X$:
$$
 \LP\big(\ph,\ph(Y)\big)\leq\LP(\ph,Y).
$$
As a consequence, if $\ph$ is a homeomorphism, then
$
 \LP\big(\ph,\ph(Y)\big)=\LP(\ph,Y)
$.
\end{prop}

\vskip4mm
\begin{proof} The proofs of the invariance and factor properties (1) and (2) go exactly along the same lines as for the topological entropy,
see \cite{HK}. The proof of the monotonicity property (4) is trivial. We give a sketch of proof of the other properties, for which we could not
find an explicit reference.

\vskip2mm
$\bullet$ We will prove the restriction property (3) for the complexity index $\L$, the proof for the other indices being essentially the same.
Consider a subset $Y\subset X$ invariant under $\ph$, endowed with
the metric $\ha d$ induced by $d$, and denote by $\ha d_n$ the metric of order $n$ defined on $Y$  by the restriction
$\ph_{\vert Y}$.  Remark that for $y,y'\in Y$  and $n\in\N^*$, $\ha d_n(y,y')=d_n(y,y')$. Thus for $\eps>0$ and $y\in Y$,
the ball $\ha B_n(y,\eps)\subset Y$ for the metric $\ha d_n$ satisfies
$$
\ha B_n(y,\eps)=Y\cap B_n(y,\eps).
$$
Let us write $\ha G_n(\eps)$ for the minimal number of $(n,\eps)$ balls for $\ha d_n$ in a covering of $Y$,
and $G_n(Y,\eps)$ for the minimal number of $(n,\eps)$ balls for $d_n$ (not necessarily centered on $Y$) in a covering of $Y$.
Clearly the previous remark shows that
$$
G_n(Y,\eps)\leq \ha G_n(\eps),\qquad \forall n\in\N^*,
$$
and, as a consequence, one sees that $\L(\ph,Y)\leq \L(\ph_{\vert Y})$. Conversely, one also sees that if
$B_n(x,\eps)\cap Y\neq\emptyset$ for some $x\in X$, then
for all $y\in B_n(x,\eps)\cap Y$
$$
B_n(x,\eps)\cap Y\subset \ha B(y,2\eps).
$$
Therefore $\ha G_n(2\eps)\leq G_n(Y,\eps)$,  $\forall n\in\N^*$, from which one  deduces that
$\L(\ph_{\vert Y})\leq \L(\ph,Y)$. This concludes the proof of the restriction property for the complexity index.

\vskip2mm
$\bullet$ For proving (5), first remark that for $x\in X$,  $n\geq 2$ and $\eps>0$:
$$
\ph\big(B_n(x,\eps)\big)\subset B_{n-1}(\ph(x),\eps),
$$
so
$$
G_{n-1}\big(\ph(Y),\eps\big)\leq G_(Y,\eps),
$$
which proves the property for $\L$ and $\un \L$.
As for the weak complexity index, for each covering
$C=(B_{n_i}(x_i,\eps))_{i\in I}\in \jC_{\geq N}(Y,\eps)$ with $N\geq 2$, the direct image
$\ph_*C$, which we denine as the covering $(B_{n_i-1}(\ph(x),\eps))_{i\in I}$,
belongs to $\jC_{\geq N-1}(\ph(Y),\eps)$
and satisfies
$$
\big(\frac{N-1}{N}\big)^sM(\ph_*C,s)\leq M(C,s)\leq M(\ph_*C,s).
$$
Therefore, in particular:
$$
\big(\frac{N-1}{N}\big)^s\de(\ph(Y),\eps,s,N-1)\leq\de(Y,\eps,s,N),
$$
which proves that
$$
\De\big(\ph(Y),\eps,s)\leq \De\big(Y,\eps,s),\quad \forall s\in\R,\ \forall \eps\in\R^{*+},
$$
and so $\LF(\ph,\ph(Y))\leq \LF(\ph,Y)$.
\end{proof}


The following product formula is useful for examples, while the power formula proves that the behaviour of the complexity indices
is genuinely different from that of the topological entropy.

\begin{prop}\label{prop:constructions}{\bf (Constructions).}
\vskip2mm
\noindent {\bf (1) Products.} If $\ph:X\to X$ and $\psi : X'\to X'$ are continuous, then
$$
 \L(\ph\times\psi)=\L(\ph)+\L(\psi),\qquad  \un\L(\ph\times\psi)\geq\un\L(\ph)+\un\L(\psi).
$$
\vskip2mm
\noindent {\bf (2) Powers.} If $\ph:X\to X$ is continuous, then for all $m\in\N^*$
$$
\LP(\ph^m)=\LP(\ph).
$$
\end{prop}

\begin{proof}
$\bullet$ (1) We denote by $\de\times \de'$ the product of any two metrics $\de$ and $\de'$ on $X$ and $X'$.
It is then easy to prove that
$
(d\times d')_n=d_n\times d'_n,\ \forall n\in\N^*.
$,
Therefore for $(x,x')\in X\times X'$ and $\rho>0$, $B_n((x,x'),\rho)=B_n(x,\rho)\times B_n(x',\rho)$ for
$n\geq1$, from which one sees that
$$
G_n(\ph\times \psi)\leq G_n(\ph,\eps)\, G_n(\psi,\eps),\qquad \forall \eps>0.
$$
If $\sig>0$ and $\sig'>0$ are such that
$
\lim_{n\to\infty}\frac{1}{n^\sig}\,G_n(\ph,\eps)=
\lim_{n\to\infty}\frac{1}{n^{\sig'}}\,G_n(\psi,\eps)=0,
$
one sees that
$$
\lim_{n\to\infty}\frac{1}{n^{\sig+\sig'}}\,G_n(\ph\times \psi)=0
$$
which proves that $\ov\sig_c(\ph\times\psi)\leq \ov\sig_c(\ph,\eps)+\ov\sig_c(\psi,\eps)$, so $\L(\ph\times\psi)\leq \L(\ph)+\L(\psi)$.

Conversely,  if the families $(x_p)_{p\in P}$  and  $(x'_{p'})_{p'\in P'}$ of points of
$X$ and $X'$ are $(\eps,n)$--separated, then  the family $(x_p,x'_{p'})_{(p,p')\in P\times P'}$
is $(\eps, n)$ separated in $X\times X'$. This proves that
$$
S_n(\ph\times\psi,\eps)\geq S_n(\ph,\eps)\,S(\psi,\eps)
$$
and therefore
$
G_n(\ph\times\psi,\eps)\geq G_n(\ph,2\eps)\,G_n(\psi,2\eps),
$
which yields
$\L(\phi\times\psi)\geq \L(\ph)+\L(\psi)$. The proof of the inequality for $\un \L$ is analogous.

\vskip2mm

$\bullet$ (2) Note that $d_n^{\ph^m}\leq d_{mn}^\ph$, so $G_{n}(\ph^m,\eps)\leq G_{mn}(\ph,\eps)$ and therefore $\L(\ph^m)\leq \L(\ph)$.
Conversely, by uniform continuity,  given $\eps>0$ there exists $\al>0$ such that $B(x,\al)\subset B_{m}(x,\eps)$ for all $x\in X$.
Then, with obvious notation:
$$
B_n^{\ph^m}(x,\al)=\bigcap_{k=0}^{n-1}\ph^{-km} \big(B(\ph^{km}(x),\al)\big)\subset
\bigcap_{k=0}^{n-1}\ph^{-km} \big(B_m(\ph^{km}(x),\eps)\big)=B_{nm}^\ph(x,\eps).
$$
So $G_{nm}(\ph,\eps)\leq G_n(\ph^m,\de)$, which proves that $\L(\ph)\leq \L(\ph^m)$ and so the equality.
The same holds for $\un \L$.

As for $\LF$, let us first prove that $\LF(\ph^m)\leq\LF(\ph)$. Fix $m\geq1$ and consider a covering $C=(B_{k_i}^\ph(x_i,\eps))\in\jC_{\geq Nm}(\ph,\eps)$ with $N> m$. Then
one easily checks that the family
$\ha C=(B_{[k_i/m]}^{\ph^m}(x_i,\eps))$ (where $[\ ]$ denotes the integer part)
is in $\jC_{\geq n}(\ph^m,\eps)$ and satisfies
$$
M(\ha C,s)\leq \Big(\Frac{Nm}{N-m}\Big)^s  M(C,s).
$$
So
$$
\de(\ph^m,\eps,s,nm)\leq \de(\ph,\eps,s,n)\Big(\Frac{Nm}{N-m}\Big)^s,
$$
and $\De(\ph^m,\eps,s)\leq m^s \De(\ph,\eps,s)$, which immediately yields $\LF(\ph^m)\leq \LF(\ph)$.

To prove the converse inequality, remark that given $\eps>0$ and $m\in\N^*$, with the same convention as above
for $\al$, any covering $\ha C=(B_{n_i}^{\ph^m}(x_i,\al))\in\jC_{\geq N}(\ph^m,\de)$ yields a covering
$C=(B_{mn_i}^{\ph}(x_i,\eps))\in\jC_{\geq Nm}(\ph,\eps)$ which satisfies
$$
M(C,s)=\Frac{1}{m^s} M(\ha C,s),
$$
which easily shows that
$\LF(\ph)\leq\LF(\ph^m)$.
\end{proof}

Now we come to the $\sig$--union property, which is satisfied by $\LF$ only (we will see in the next section a couterexample proving
that $\L$ does not enjoy this property).

\vskip2mm
\begin{prop} {\bf (The $\sig$-union property for $\LF$).}
Let $(Y_m)_{m\in \N}$ be a sequence of subsets of $X$. Then
$$
 \LF\Big(\ph,\bigcup_{m\in\N}Y_m\Big)=\Sup_{m\in\N} \LF(\ph,Y_m).
 $$
\vskip2mm
\end{prop}

\begin{proof}
Set $Y=\cup_{m\in\N}Y_m$. Then by the monotonicity property
$$
 \LF(\ph,Y)\geq\Sup_{m\in\N} \LF(\ph,Y_m).
$$
To prove the converse inequality, given $\eps>0$, consider $s>\Sup_{m\in\N} s_c(Y_m,\eps)$, so $\De(Y_m,\eps,s)=0$
for all $m\in\N$.
Therefore, given $N\in\N$, for every index $m$ there exists
$N_m\geq N$ and
$C_m\in\jC_{\geq N_m}(Y_m,\eps)$ such that
$$
\de(Y_m,\eps,s,N_m)\leq \Frac{1}{2^{m+2}}\quad
\text{and}
\quad
M(C_m,s)-\de(Y_m,\eps,s,N_m)\leq \Frac{1}{2^{m+2}},
$$
so
$$
M(C_m,s)\leq \Frac{1}{2^{m+1}}.
$$
Now the union $C=\sqcup_{m\in\N} C_m$ of the previous families is a covering of $Y$ and belongs to $\jC_{\geq N}(Y_m,\eps)$,
therefore
$$
\de(Y,\eps,s,N)\leq M(C,s)=\sum_{m\in\N} M(C_m,s)\leq 1.
$$
This inequaliy holds true for all $N\in\N^*$ and, as a consequence:
$$
\De(Y,\eps,s) \leq 1,
$$
which proves that $s\geq s_c(Y,\eps)$. So $\Sup_{m\in\N} s_c(Y_m,\eps)\geq s_c(Y,\eps)$ and therefore
$\Sup_{n\in\N} \LF(\ph,Y_m)\geq \LF(\ph,Y)$.
\end{proof}

As for the indices $\L$ and $\un \L$, one only has the following trivial properties.
\begin{prop} {\bf (Union properties for $\L$).}
Let $(Y_m)_{m\in \N}$ be a sequence of subsets of $X$. Then
$$
 \L\Big(\ph,\bigcup_{m\in\N}Y_m\Big)\leq\Sup_{m\in\N} \L(\ph,Y_m),\qquad
 \L\Big(\ph,\bigcup_{1\leq m\leq M}Y_m\Big)=\Sup_{1\leq m\leq M} \L(\ph,Y_m).
$$
\end{prop}

\subsection{Comparison of the complexity indices}
We already proved the inequality $\LF\leq\un \L\leq \L$.
The following proposition exhibits a very simple example for which the first inequality is  strict,  while the other one is an equality.

\begin{prop}\label{prop:indseg}
Consider on the segment $I=[0,1]$ an arbitrary $C^1$ vector field $X$  which satisfies
$X(0)=X(1)=0$ and $X(x)>0$ for $x\in\,]0,1[$,  and which is decreasing on an interval
$[\ka,1]$ with $0<\ka<1$. Let $\ph$ be the time-one map of $X$. Then
$$
\LF(\ph)=0,\qquad \un\L(\ph)=\L(\ph)=1.
$$
\end{prop}

\begin{proof} We first consider the weak complexity index. As $X$ decreases on $[\ka,1]$, one sees that
$\ph$ contracts the distances on the same interval, and therefore $\LF(\ph,[\ka,1])=0$. Then, as $X(x) > 0$
for $x\in\,]0,1[$,
$$
]0,1]=\bigcup_{n\in\N}\ph^{-n}([\ka,1])
$$
and the transport and $\sig$--union properties prove that $\LF(\ph,]0,1])=0$. As $\LF(\ph,\{0\})=0$,
one concludes  that $\LF(\ph)=0$.

\vskip2mm

We will now prove that $ \L(\ph)\leq 1$.
Fix $\eps>0$ and let us construct an explicit covering of $I$ with balls of radius $\eps$ for $d_n^\ph$, for any prescribed $n$.
We assume $\eps <1/2$ and we introduce  the  intervals $I_0=[0,\eps]$, $I_1=[1-\eps,1]$ and $J=[\eps, 1-\eps]$. We will
separately construct coverings for each of these intervals.

\vskip 2mm

-- Since $\ph(x)\geq x$ for all $x\in X$, $I_1$ is exactly the ball $B_n(1,\eps)$, for all $n\in\N^*$.

\vskip 2mm

-- To cover $J$, remark that there exists $n_0$ such that $\ph^{n_0}(\eps)\in I_1$. Therefore, for any two points
$x$ and $y$ in $J$, $d(\ph^n(x),\ph^n(y))\leq \eps$ when $n\geq n_0$, and so one hast just to consider the iterates of
order $n$ with $0\leq n\leq n_0$. Since the maps $\ph,\ph^2,\ldots,\ph^{n_0}$ are uniformly continuous on $J$
there exists $\al>0$ such that if $d(x,y)<\al$ then $d(\ph^n(x),\ph^n(y))\leq \eps$ for all $n\in\{0,\ldots,n_0\}$.
Therefore one also has $d(\ph^n(x),\ph^n(y))\leq \eps$ for all $n\geq0$.
We divide $J$ in subintervals $J_1, J_2, \ldots, J_p$ of equal length less than $\al$ and pick a point $x_i$ in $J_i$,
so $J_i\subset B_n(x_i,\eps)$ for all $n\in\N^*$.
This way we get a covering of $J$ with $p$ such balls. Note that $p$ depends on $\eps$ but not on $n$.

\vskip 2mm

--  It only remains to cover the first interval $I_0$. For this part only we fix $n\geq 1$.
First remark that the interval $[\eps,\ph(\eps)]$ is covered by some intervals, say
$J_1,\ldots,J_q$ of the previous family. Thus the interval $[\ph\inv(\eps),\eps]=\ph\inv([\eps,\ph(\eps)])$ is
covered by $\ph\inv(J_j)$, $1\leq j\leq q$.
It is clear by construction and by the previous point that $\ph\inv(J_j)$ is contained in the ball $B_n(\ph\inv(x_j),\eps)$,
therefore $[\ph\inv(\eps),\eps]$ can be covered  by $q$ such $(n,\eps)$-balls.

By the same argument, taking the pullbacks of order $k$, $1\leq k\leq n$, we obtain a covering of each interval
$[\ph^{-(k+1)}(\eps),\ph^{-k}(\eps)]$ by a number $q$ of $(n,\eps)$-balls, and as
a result a covering of $[\ph^{-n}(\eps),\eps]$ by $nq$ such $(n,\eps)$-balls.

It therefore only remains to consider the interval $[0,\ph^{-n}(\eps)]$, which exactly coincides with the ball $B_n(0,\eps)$,
so is convered with one $(n,\eps)$-ball.

\begin{figure}[h]
\begin{center}
\begin{pspicture}(4cm,1.2cm)
\rput(2,0){
\psset{xunit=1cm,yunit=1cm}
\psline[linewidth=.02](-5,0)(5,0)
\pscircle[linewidth=.01, fillstyle=solid,fillcolor=white](-5,0){.07}
\pscircle[linewidth=.01, fillstyle=solid,fillcolor=white](5,0){.07}
\pscircle[linewidth=.01, fillstyle=solid,fillcolor=white](-4,0){.07}
\pscircle[linewidth=.01, fillstyle=solid,fillcolor=white](4,0){.07}
\pscircle[linewidth=.01, fillstyle=solid,fillcolor=white](0,0){.07}
\rput(-5,.3){$0$}
\rput(5,.3){$1$}
\rput(-4,.3){$\eps$}
\rput(4,.3){$1-\eps$}
\rput(.2,1){$\ph(\eps)$}
\psline[linewidth=.01](0,.3)(0,.7)
\psline[linewidth=.02](-3.2,-.1)(-3.2,.1)
\psline[linewidth=.02](-2.4,-.1)(-2.4,.1)
\psline[linewidth=.02](-1.6,-.1)(-1.6,.1)
\psline[linewidth=.02](-0.8,-.1)(-0.8,.1)
\psline[linewidth=.02](0,-.1)(0,.1)
\psline[linewidth=.02](3.2,-.1)(3.2,.1)
\psline[linewidth=.02](2.4,-.1)(2.4,.1)
\psline[linewidth=.02](1.6,-.1)(1.6,.1)
\psline[linewidth=.02](0.8,-.1)(0.8,.1)
\rput(-4.5,-.3){$I_0$}
\rput(4.5,-.3){$I_1$}
\rput(-3.6,.3){$J_1$}
\rput(-2.8,.3){$J_2$}
\rput(-1.5,.3){$\cdots$}
\rput(-.4,.3){$J_q$}
\rput(1.9,.3){$\cdots$}
}
\rput(2,-1.8){
\psset{xunit=1cm,yunit=1cm}
\psline[linewidth=.02](-5,0)(5,0)
\pscircle[linewidth=.01, fillstyle=solid,fillcolor=white](-5,0){.07}
\pscircle[linewidth=.01, fillstyle=solid,fillcolor=white](5,0){.07}
\pscircle[linewidth=.01, fillstyle=solid,fillcolor=white](-3,0){.07}
\pscircle[linewidth=.01, fillstyle=solid,fillcolor=white](3.1,0){.07}
\pscircle[linewidth=.01, fillstyle=solid,fillcolor=white](1.7,0){.07}
\rput(-5,.3){$0$}
\rput(5,.3){$\eps$}
\rput(-3,.4){$\ph^{-N}(\eps)$}
\rput(3,.4){$\ph^{-1}(\eps)$}
\rput(1.7,.4){$\ph^{-2}(\eps)$}
\rput(-.5,.4){$\cdots$}
\rput(5.5,-.8){$\ph^{-1}(J_q)$}
\rput(3,-.8){$\ph^{-1}(J_1)$}
\rput(4.3,-.8){$\cdots$}
\psline[linewidth=.01](5.5,-.5)(4.85,-.02)
\psline[linewidth=.01](3,-.5)(3.2,-.02)
\psline[linewidth=.02](4.7,-.1)(4.7,.1)
\psline[linewidth=.02](4.4,-.1)(4.4,.1)
\psline[linewidth=.02](4.15,-.1)(4.15,.1)
\psline[linewidth=.02](3.9,-.1)(3.9,.1)
\psline[linewidth=.02](3.7,-.1)(3.7,.1)
\psline[linewidth=.02](3.55,-.1)(3.55,.1)
\psline[linewidth=.02](3.4,-.1)(3.4,.1)
\psline[linewidth=.02](3.25,-.1)(3.25,.1)
}
\end{pspicture}
\end{center}
\vskip2.4cm
\caption{The covering of $I$}
\end{figure}
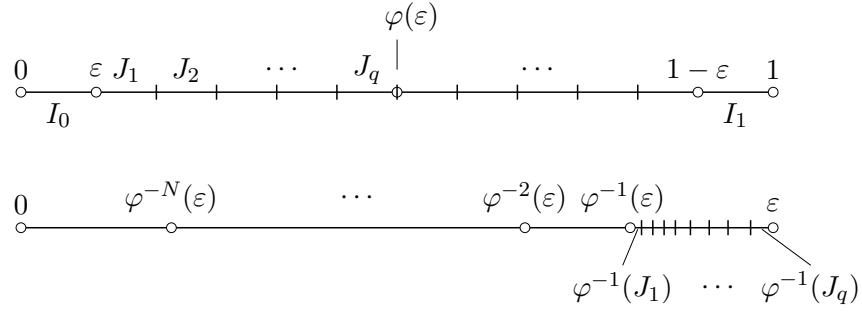

Gathering all the elements of the previous reasoning, we end up with a covering of $I$ with
$nq+p+2$ balls of radius $\eps$ for $d_n^\ph$, which proves that
$$
G_n(\eps)\leq (n+1)p+2,
$$
where $p$ is independent of $n$. Therefore $\L(\ph)\leq 1$.

\vskip2mm

We will now prove that $\un \L(\ph)\geq 1$, and for this we only need now to find suitable separated subsets.  Consider a point
$a\in\,]0,1[$ and let $\eps_0=\ph(a)-a$, so $\eps_0>0$. Then one easily checks that for $0<\eps <\eps_0$ and $n\geq 1$, the points
$$
\ph^{-n+1}(a),\ldots, \ph\inv(a),a
$$
are $(n,\eps)$--separated, so $G_n(\ph,\eps/2)\geq S_n(\ph,\eps)\geq n$ and therefore
$$
\liminf_{n\to\infty} G_n(\ph,\eps/2)/n\geq 1,
$$
which proves that
$\un\L(\ph)\geq 1$ and so that $\un\L(\ph)=\L(\ph)=1$.
\end{proof}

\begin{cor}
The indices $\un\L$ and $\L$ do not enjoy the $\sig$--union property.
\end{cor}

The proof is obvious, since any index which enjoys the $\sig$--union property and the transport property would vanish for
the previous system on the segment, by the same argument as for $\LF$.

Note finally that the same reasoning would apply and yield the same indices for any gradient vector field on a sphere
$S^n$, $n\geq 2$, with only two  singularities.


\subsection{The complexity indices of action-angle systems}
A remarkable fact, in view of the previous proposition, is that for integrable systems in action-angle
form all the complexity indices do coincide, as we will now prove.

\begin{prop}\label{prop:actang}  Consider a $C^2$ Hamiltonian function $h$ which depends only
on the action variable $r$ and is defined on a given closed ball $\BB:=\ov B(r_0,R)$ of $\R^n$.
Its  time-one map therefore reads:
$$
\ph(\th,r)=\big(\th+\om(r) \ [{\rm mod}\, \Z^n],\,r\big),
$$
where $\om(r)=\d h(r)$  is $C^1$ on  $\BB$.
Then the complexity indices of $\ph$ on $\T^n\times \BB$
satisfy:
$$
\LF(\ph)=\un\L(\ph)=\L(\ph)=\Max_{r\in \BB} \rk \om(r).
$$
\end{prop}

\begin{proof} Recall that given a compact metric space $(X,d)$, the ball dimension $D(X)$ is by definition
$$
D(X):=\limsup_{\eps\to 0}\Frac{\Log c(\eps)}{\abs{\Log\eps}}
$$
where $c(\eps)$ is the minimal cardinality of a covering of $X$ by $\eps$--balls. We will use the fact that the ball
dimension of a compact manifold is equal to its usual dimension, and that the ball dimension of the image of a compact
manifold by $C^1$ map of rank $\ell$ is  $\leq \ell$.

\vskip2mm

We endow $\R^n$ with the product metric defined by the $\Max$ norm $\norm{\ }$ and the ball $\BB$ with the induced
metric. We endow the torus $\T^n$ with the quotient metric. As the pairs of points  $\th$, $\th'$ of $\T^n$ we will have to consider are close enough
to one another, we still denote by $\norm{\th-\th'}$ their distance. Finally
we endow the annulus $\T^n\times \BB$ with the product metric of the previous ones.

\vskip2mm

Let $\ell:=\Max_{r\in B} \rk \om(r)$.
We will first prove that $\L(\ph)\leq \ell$. Let $\eps>0$ be fixed and consider $N\geq 1$.
Remark by elementary computation that if two points $(\th,r)$ and $(\th',r')$ of $\T^\ell\times \BB$ satisfy
\begin{equation}\label{eq:ball}
\norm{\th-\th'}< \Frac{\eps}{2},\qquad
\norm{\om(r)-\om(r')}\leq \Frac{\eps}{2N},\qquad
\norm{r-r'}<\eps
\end{equation}
then $d_N^\ph\big( (\th,r),(\th',r'))< \eps$. We are thus led to introduce the following coverings :

\vskip1mm $\bullet$   a minimal covering $C_{\T^n}$ of $\T^n$ with balls of radius $\eps/2$, so its cardinality
$i^*$ depends only on $\eps$;

\vskip1mm $\bullet$  a minimal covering  $(\ha B_j)_{1\leq j\leq j^*}$  of $\BB$ by balls of radius $\eps/2$, so again
$j^*$ depends only on $\eps$;

\vskip1mm $\bullet$ for $N\geq1$, a  minimal covering  $(\til B_k)_{1\leq k\leq k^*}$ of the image $\om(\BB)$ with balls of radius
$\eps/(4N)$.

\vskip1mm

Using the last two coverings, we get a covering $C_{\BB}=(\ha B_j\cap\om\inv(\til B_k))_{j,k}$ of $\BB$  such that any two points $r,r'$ in the same
set $\ha B_j\cap\om\inv(\til B_k)$ satisfy the last two conditions of (\ref{eq:ball}).
We finally obtain a covering of $\T^n\times\BB$ by considering the products of the elements of $C_{\T^n}$ and $C_{\BB}$, the elements
of which are contained in balls of $d_N^\ph$ radius $\eps$.

\vskip1mm

Note that given $\ell'>\ell$,  for $N$ large enough, $k^*\leq (2N/\eps)^{\ell'}$ since the ball dimension of $\om(\BB)$ is less than $\ell$.
This proves that
$$
G_N(\ph,\eps)
\leq i^*\,j^*\,k^*
\leq c(\eps) N^{\ell'}.
$$
As a consequence $\ov s_c(\ph,\eps)\leq \ell'$ and $\L(\ph)\leq\ell$ since $\ell'>\ell$ is arbitrary.

\vskip2mm

Let us now prove that $\LF(\ph)\geq \ell$.
We first need to describe the $(N,\eps)$--balls of the system more precisely. Let $(\th,r)$ be given, with $r$ in the open ball $B(0,R)$,
and fix $\eps>0$ such that $B(r,\eps)\subset B(0,R)$.
Then a point $(\th',r')$  belongs to the ball $B_N((\th,r),\eps)$ if and only if
$$
\norm{r'-r}<\eps,\qquad  \norm{k(\om(r')-\om(r))+(\th'-\th)}<\eps,\quad \forall k\in\{0,\ldots, N-1\}.
$$
Writing the various vectors in component form, the second condition is equivalent to
$$
\abs{\th'_i-\th_i}<\eps,\quad
\om_i(r')\in \Big]\Frac{(\th_i-\th'_i)-\eps}{N-1}+\om_i(r),\Frac{(\th_i-\th'_i)+\eps}{N-1}+\om_i(r)\Big[,\qquad 1\leq i\leq n.
$$
Therefore the ($2n$--dimensional) ball $B_N((\th,r),\eps)$ has a fibered structure over the
($n$--dimensional)  ball $B(\th,\eps)$, that is
$$
B_N((\th,r),\eps)=\bigcup_{\th'\in B(\th,\eps)}\{\th'\}\times F_{\th'},
$$
the fiber over the point
$\th'$ being the curved polytope
$$
F_{\th'}= \om\inv\Big(\prod_{1\leq i\leq n}\Big]\Frac{(\th_i-\th'_i)-\eps}{N-1}
 +\om_i(r),\Frac{(\th_i-\th'_i)+\eps}{N-1}+\om_i(r)\Big[\Big)\bigcap B(r,\eps).
$$

 Let now $r_0$ be in $B(0,R)$ and such that $\rk \om(r_0)=\ell$, and let
$\al>0$ be small enough so that  $B(r_0,2\al)\subset \BB$ and $\rk \om=\ell$ on $B(r_0,2\al)$.
Assume that a covering $C=(B_{n_i}((\th_i,r_i),\eps))_{i\in I}$ of $\jC_{\geq N}(\T^\ell\times B(0,\al))$ is given, with $\eps<\al$,
and denote by $F_0^{i}$ the fiber of $\th=0$ in the ball $B_{n_i}((\th_i,r_i),\eps)$ (which may be empty).
Then the set $\{0\}\times B(0,\al)$ is contained in the union of the fibers $F_0^i$. Let $\nu=(2\al)^n$ the $n$-dimensional
Lebesgue volume of this set.

Due to the assumption on the rank of $\om$, if $\al$ is small enough, there exists a constant $c>0$ such
that the Lebesgue volume of  the fiber $F_0^i$ satisfies
$$
{\rm Vol\,}(F_0^i)\leq c \Big(\Frac{2\eps}{n_i-1}\Big)^\ell.
$$
The sum of the volumes of the fibers  must be larger than $\nu$, so
$$
\sum_{i\in I}c \Big(\Frac{2\eps}{n_i-1}\Big)^\ell\geq \nu
$$
Assume that $s<\ell$. Then
$$
M(C,s)=\sum_{i\in I}\Frac{1}{n_i^s}= \Frac{1}{c(2\eps)^\ell}\sum_{i\in I}c\Big(\Frac{2\eps}{n_i-1}\Big)^\ell
\Frac{(n_i-1)^\ell}{n_i^s}\geq \Frac{\nu}{c(2\eps)^\ell}\Frac{1}{2^\ell} N^{\ell-s}
$$
so
$$
\De(\T^\ell\times B(0,\al),\eps,s)=\lim_{N\to\infty}\de(\T^\ell\times B(0,\al),\eps,s,N)=+\infty.
$$
This shows that $s_c(\T^\ell\times \BB,\eps)\geq\ell$, and finally that $\LF(\ph)\geq \ell$, which concludes the proof.
\end{proof}

Note that the previous proposition shows that the complexity indices cannot enjoy any analog of the Bowen formula.
Indeed, the restriction of the map $\ph$ to each invariant Lagrangian torus $\T^n\times\{r\}$ is an isometry, and so have zero complexity
relative to any measurement. So a ``Bowen formula'' would yield a vanishing index for $\ph$, which is not the case.

Note also that the complexity indices of action-angle systems detect the ``effective number of degrees of freedom'' of such systems,
and are in any case smaller than half the dimension of the ambient manifold, a remark which will be used in the next section.


\subsection{The complexity indices of continuous systems}

For the sake of completeness, we conclude this section with the definition of the complexity indices semi-flows on compact spaces $(X,d)$,
that is for continuous
maps $\phi$ of  $[0,+\infty[\,\times X$ to $X$ which satisfy the condition $\phi^0=\Id$ and $\phi^s\circ\phi^t=\phi^{s+t}$ where, as usual,
we denote by $\phi^t$ the map $\phi(t,.)$. For $t\geq 1$ one defines the continuous family of dynamical distances
$$
d_t^\phi(x,y)=\Sup_{0\leq \tau\leq t-1} d(\phi^\tau(x),\phi^\tau(y))
$$
which all define the same compact topology on $X$, note moreover that $d_t^\phi\geq d_{t'}^\ph$ for $t\geq t'\geq 1$.

\parag For each $t\geq 1$, we denote by $\jC^\phi_{\geq t}(\eps)$ the set of coverings of $X$ of the form $C=(B_{\tau_i}(x_i,\eps))_{i\in I}$ with
$\tau_i\geq t$ and, for such a covering $C$, we set $M(C,s)=\sum_{i\in I}\frac{1}{\tau_i^s}$ for $s\geq 0$. Finally we introduce the quantity
$$
\de^\phi(\eps,s,t)=\Inf\{M(C,s)\mid C\in \jC^\phi_{\geq t}(\eps)\}
$$
which is monotone non-decreasing with $t$, and wet set $\De^\phi(\eps,s)=\lim_{t\to\infty}\de^\phi(\eps,s,t)$. One sees that there exists
a unique $s_c^\Phi(\eps)$ such that $\De^\phi(\eps,s)=0$ if $s>s_c^\phi(\eps)$ and $\De^\phi(\eps,s)=+\infty$ if $s<s_c^\phi(\eps)$.
Finally, we define the weak complexity index for the continuous system $\phi$ as
$$
\L^*_c(\phi)=\lim_{\eps\to 0}s_c^\phi(\eps)=\Sup_{\eps>0}s_c^\phi(\eps).
$$
It turns out that if $\ph=\phi^1$, then
$$
\L^*_c(\phi)=\LF(\ph).
$$
To see this, first note that for $\eps>0$, there exists $\al_\eps>0$ such that if $d(x,y)<\al_\eps$ then $d(\phi^\tau(x),\phi^\tau(y))<\eps$ for all $x,y$ in $X$ and $\tau\in[0,1]$.
Therefore, for all $x\in X$ and $t\geq 1$, $B_{[t]}^\ph(x,\al_\eps)\subset B_t^\phi(x,\eps)$. From this one easily deduces that
$$
s_c^\phi(\eps)\leq s_c(\ph,\al_\eps)
$$
and therefore
$\L^*_c(\phi)\leq\LF(\ph)$. To prove the converse inequality one only has to remark that $d_[t]^\ph\leq d_t^\ph$ for all $t\geq 1$,
so clearly $s_c^\phi(\eps)\geq s_c(\ph,\eps)$.

\parag We now denote by $G_t^\phi(\eps)$ the minimal number of $d_t^\phi$--balls of radius $\eps$ in a covering of $X$, and set
$$
\L_c(\phi)=\limsup_{t\to\infty}\Frac{\Log G_t^\phi(\eps)}{\Log t}=\Inf\{\sig\geq 0\mid \lim_{t\to\infty}\Frac{1}{t^\sig}G_t(\eps)=0\}.
$$
One immediately sees that
$$
G_{[t]}(\ph,\eps)\leq G_t^\phi(\eps)\leq G_{[t]}(\ph,\al_\eps)
$$
from which one again deduces the equality
$$
\L_c(\phi)=\L(\ph).
$$

In the following, we therefore limit ourselves the the complexity indices of discrete systems.


\section{Strong integrability, decomposability and complexity}
We introduce here the notion of strong integrability, already used at an informal level by Paternain \cite{P}, which is a very mild global
assumption on the singularities of the first integrals.
We extensively study the geometric structure of strongly integrable systems and, as an application,
 give a short proof of Paternain's result on the
vanishing of their topological entropy.
We then introduce the new notion of decomposability, which is a refinement of the notion of strong integrability and relies on the previous
geometric study. Finally we prove that the weak complexity index of decomposable systems  is smaller than their number of degrees of
freedom.


\subsection{Strong integrability: structure and entropy}
The notion of strong integral we introduce below could be given in the general framework of Poisson manifolds, however we limit ourselves to the significantly simpler but fundamental case of symplectic manifolds, equipped with their canonical Poisson structure (see \cite{LMV}
for recent geometric results of action-angle type in the Poisson case). All objects are supposed to be smooth.

\subsubsection{Strong integrability and non-degenerate integrals}\label{sssec:int}
We consider a symplectic manifold
$(M,\Om)$ of dimension $2\ell$ and denote by $\{.,.\}$ its Poisson bracket. The Hamiltonian vector field associated with a function $f:M\to\R$
will be denoted by $X^f$.

\vskip2mm

1. We say that a map $F=(f_i):M\to\R^\ell$ is an {\em integral map} (or simply an integral)
when its components are in involution, that is
$$
\{f_i,f_j\}=0,\quad 1\leq i\leq \ell, \ 1\leq j\leq \ell.
$$
An integral map $F$ defines a local action  $\Phi_F$ on a neighborhood of $\{0\}\times M$ in $\R^\ell\times M$,
by
$$
(\tau,x)\mapsto \Phi_F(\tau,x)=\Phi^{\tau_\ell f_\ell}\circ\cdots\circ\Phi^{\tau_1 f_1}(x),$$
for  $x\in M$ and $\tau=(\tau_1,\ldots,\tau_\ell)$ in a small enough neighborhood of $0$ in $\R^\ell$.
When the fields $X^{f_i}$ are complete, $\Phi_F$ is an action of $\R^\ell$ on $M$, called the joint flow of $F$.  The orbits of this action are
isotropic immersed submanifolds of $M$.

\vskip2mm

2. Given $H\in C^\infty(M,\R)$, we say that the integral map $F$ is an {\em integral for $H$} when it is constant on the orbits of $X^H$,
or equivalently when $\{f_i,H\}=0$ for $1\leq i\leq\ell$.
In the following we define an {\em integrable system} as a quadruple
$(M,\Om,H,F)$, where $H$ is a Hamiltonian function on $M$ and $F$ an integral for $H$ which is of rank $\ell$ on an open and dense subset of $M$. There is obviously no uniqueness property for $H$ and $F$, but we will not address this question here.

Integrable systems have a twofold nature, according to the distinct roles of the Hamiltonian function and the integral: the integral
determines a decomposition of $M$ into its level sets, while the Hamiltonian function governs the dynamics on these level sets.

\vskip2mm

3. We already mentioned in the introduction that the topological entropy of the Hamiltonian flow of an integrable system
is localized on the singular set of its integral map.
The following definition gives global constraints on this set, which enables one to control the entropy.

\begin{Def}\label{def:strongint}
Let $(M,\Om)$ be a symplectic manifold of dimension $2\ell$, $\ell\geq1$, and consider an integral map
$F=(f_i)_{1\leq i\leq\ell}:M\to\R^\ell$.
For $d\in\{0,\ldots,\ell\}$, let
$$
\Sigma_d=\{x\in M\mid \rk F(x)=d\}.
$$
We say that $F$ is a {\em strong integral map} when:
\vskip2mm\noindent
$(i)$  for each $d\in\{0,\ldots,\ell\}$, $\Sigma_d$ is either empty or an embedded submanifold of $M$, of dimension $2d$,
on which $\Omega$ induces a symplectic form (that is $j^*\Om$ is non-degenerate, where $j:\Sigma_d\to M$ is the canonical inclusion);
\vskip2mm\noindent
$(ii)$ for each $d\in\{0,\ldots,\ell\}$, $\rk(F_{\vert \Sigma_d})=d$;
\vskip2mm
\noindent We then say that the collection $(\Sig_d)_{0\leq d\leq \ell}$ is the {\em partition associated with $F$}.
\vskip2mm
We say that a Hamiltonian function $H:M\to\R$ is {\em strongly integrable} if there exists a strong integral
map which is an integral for $H$, and we define a {\em strongly integrable system} as a quadruple $(M,\Om,H,F)$, where $H:M\to\R$ and
$F$ is a strong integral for $H$.
\end{Def}

\vskip2mm

For instance, when $M$ is two dimensional,  a Hamiltonian function $H$ with
isolated singularities is strongly integrable, with strong integral $F=H$.  But there exist strongly integrable Hamiltonian functions with non-isolated singularities, as shown
by the Hamiltonian function $H(x,y)=\pdemi\, y^2$ of the free particle on a line, which admits the strong integral
$F(x,y)=y$.

\vskip2mm

More generally, let $F:M\to\R^\ell$ be an integral map and consider $x\in \Sigma_d$
with $0\leq d\leq \ell-1$. One can assume that the first $d$ components $f_i$ of $F$ are independent at $x$, set $\ha F=(f_i)_{1\leq i\leq d}$,
and assume that $\ha F(x)=0$. The following easy results will allow us to set down a sufficient condition for $F$ to be a strong integral.

\vskip1mm

1. There exists a neighborhood $N$ of $x$ in $M$ and a transverse section $S$ to the (local) joint flow,  diffeomorphic to some ball $B^{2\ell-d}$
and containing $x$, such that $N=\Phi_F(B^d,S)$, where $B^d$ is some
ball centered at $0$ in $\R^d$.

\vskip1mm

2. The intersections $N_a=N\cap \ha F\inv(\{a\})$ are $(2\ell-d)$ embedded coisotropic submanifolds, for $a\in V= \ha F(S)$, such that
$N=\cup_{a\in V}N_a$. The characteristic distribution of $N_a$ at $y$ is the subspace of $T_yM$ spanned by the vector fields
$X^{f_i}(y)$, $1\leq i\leq d$. For $a\in V$, the intersection $S_a=S\cap N_a$ is a symplectic submanifold of dimension $2(\ell-d)$, isomorphic to the quotient of $N_a$ by the orbits of the joint flow.

\vskip1mm

3. The remaining components $F^r=(f_{d+1},\ldots,f_\ell)$ pass to the quotient by the joint flow and give rise to a function $\til F$
which can be identified with the restriction of $F^r$ to $S$. For $a\in V$, the restriction $\til F_a=\til F_{\vert S_a}$ is an integral
map on $S_a$.

\vskip1mm

4. One can choose $S$ in such a way that for $a\in V$, there exists  a symplectic diffeomorphism $\psi_a$  from $B^{2(\ell-d)}$ to $S_a$.
Let $(\xi_i,\eta_i)$ be a symplectic
coordinate system on $B^{2(\ell-d)}$.  For $y\in N$ and $a=\ha F(y)$, there exists a unique $y_0\in S_{a}$ and a unique $\tau\in B^{d}$
such that $y=\Phi_F(\tau,y_0)$. So one can associate to $y$ the set of local coordinates $(\tau,a,\xi,\eta)\in B^d\times V\times B^{2(\ell-d)}$,
where $(\xi,\eta)=\psi_a\inv(y_0)$.

\vskip1mm

5. In these coordinates, the local expression of $F$ reads
$$
\F(\tau,a,\xi,\eta)=\big(a,\til F_a(\psi_a(\xi,\eta))\big).
$$

We are now in a position to set our main definition and our criterion.

\begin{Def}
We say here that the point $x$ is {\em simple} when there exists a smooth function $\ze:V\to S$, with $\ze(0)=x$ and $\ze(a)\in S_a$ for $a\in V$, such that
$\rk \til F_a(\ze(a))=0$ and $\rk \til F_a(\ze)>0$ when $\ze\in S_a\setminus\{ \ze(a)\}$.
\end{Def}

\begin{lemma}
Let $F$ be a integral map on the manifold $(M,\Om)$. Then if each singular point $x$ of $F$ is simple, the integral $F$ is strong.
\end{lemma}

\begin{proof} Let $x\in \Sig_d$, $0\leq d\leq \ell-1$ and assume that $x$ is simple.  In the previous coordinate system, one sees that
$$
\rk\F(\tau,a,\xi,\eta)=d+\rk\til F_a(\psi_a(\xi,\eta)).
$$
So, in these coordinates, the set $\Sig_d$ has the simple form
$$
\big\{\big(\tau,a,\psi_a\inv(\ze(a))\big)\mid \tau\in B^d,\ a\in V\big\}.
$$
and one easily checks conditions $(i)$ and $(ii)$ of Definition \ref{def:strongint}.
\end{proof}

The previous local form for $F$ is the starting point for an analysis of the {\em differential} local structure of integral maps at their singularities.
In spite of many partial results, see \cite{Z} for a survey, we still do not have a complete {\em symplectic} singularity theory for
such maps.  We therefore limit ourselves to a very useful non-degeneracy condition introduced by Eliasson (\cite{E84,E90}), which provides us with the main examples of strongly integrable systems and will also prove
useful in the following.

 Assume first that the rank $d$ of $F$ at $x$ is $0$.
Then the Hessian quadratic forms $Q_i=d^2f_i(x)$ are well-defined on $T_xM$ and
generate an abelian subalgebra $Q_x$ of the Lie algebra $Q(T_xM)$ of all quadratic forms equipped with the linearized Poisson bracket.
One says that $F$ is {\em Eliasson non-degenerate} at $x$ if $Q_x$ is a Cartan subalgebra of $Q(T_xM)$. When it is the case, by Williamson
theorem, there exists a triple of nonnegative integers $(n_e,n_h,n_{f})$ satisfying $n_e+n_h+2n_{f}=\ell$, and a (linear) system of symplectic
coordinates $(\xi_i,\eta_i)_{1\leq i\leq\ell}$
in $T_xM\sim \R^{2\ell}$  such that the algebra $Q_x$ is generated by the following quadratic forms:
\vskip1mm
$\bullet$  $(q_i^{(e)})$, for ${1\leq i\leq n_e}$, with \ \  $q_i^{(e)}(\xi,\eta)=\xi_i^2+\eta_i^2$;
\vskip1mm
$\bullet$  $(q_i^{(h)})$, for ${n_e+1\leq i\leq n_e+n_h}$, with \ \ $q_i^{(h)}(\xi,\eta)=\xi_i\eta_i$;
\vskip1mm
$\bullet$ $\big((q_i^{(f)},q_{i+1}^{(f)})\big)$, for $i=n_e+n_h+2j-1$ and $1\leq j\leq n_f$, with
\vskip-2mm
$$
q_i^{(f)}(\xi,\eta)=\xi_i\eta_{i+1}-\xi_{i+1}v_i,\qquad q_{i+1}^{(f)}(\xi,\eta)=\xi_i\eta_i+\xi_{i+1}\eta_{i+1}.
$$

 The superscripts $e,h,f$ stand for elliptic, hyperbolic and focus-focus respectively; note that the quadratic form of focus-focus type always come by pairs. Eliasson theorem states that there exists a local {\em symplectic} diffeomorphism from a neighborhood $N_x$ of $x$ in $M$ to a neighborhood $N$ of $0$ in $\R^{2\ell}$ which exchanges the Lagrangian fibrations defined by the level sets of $F$ in $N_x$ with those defined by the previous
quadratic forms in $N$.

Consider now a point $x\in M$ such that $d=\rk F(x)\in\{1,\ldots,\ell-1\}$. One says that $F$ is Eliasson non-degenerate at $x$ when the reduced map $\til F_0$ introduced above is Eliasson non-degenerate at $\til x$ as an integral map $S_0\to\R^{\ell-d}$ (note that $\rk \til F_0(x)=0$). Then a parametrized form of Eliasson theorem applies to our previous analysis and proves that  the local structure of the Lagrangian fibration defined by the integrals $\til F_a$ remains symplectically invariant when $a$ varies in a small enough neighborhood of $0$. As a consequence, the point $x$ is simple. We therefore have proved the following lemma.

\begin{lemma}
Eliasson non-degenerate integrals are strong integrals.
\end{lemma}

This is particularly interesting since most of the classical integrable systems admit Eliasson non-degenerate integrals.
See also \cite{D} for a direct  analysis of the structure of non-degenerate integral maps. See also \cite{I} for another interesting non-degeneracy condition. In \cite{MM}, non-degenerate systems on 4 dimensional
symplectic manifolds will be extensively studied from the point of view of complexity.


\subsubsection{The structure of strongly integrable systems}
The conditions we impose on strong integral maps are essentially the mildest ones  enabling one to apply the
Liouville-Mineur-Arnold theorem  (see for instance \cite{Z}) everywhere in the ambient manifold.
The next easy lemma give a first description of the structure of strongly integrable systems. We say that an integral
map is {\em complete} when the Hamiltonian vector fields of its components are complete.

\begin{lemma}\label{lem:struct1}
Assume that $F$ is a complete strong integral on $(M,\Om)$ and
consider a Hamiltonian function $H$ on $M$ which admits $F$ as an integral.
Then:
\vskip1mm\noindent
(i) for each orbit $O$ of the joint flow $\Phi_F$ of $F$, the rank of $F$ is constant on $O$,
so $O$ is contained in  some submanifold $\Sig_d$;
\vskip1mm\noindent
(ii) an orbit of $\Phi_F$ is contained in $\Sig_d$ if and only if it is $d$-dimensional;
\vskip1mm\noindent
(iii) if $O$ is a $d$-dimensional orbit of $\Phi_F$, there exists $k\in\{0,\ldots,d\}$ and an immersion
$j_O:\T^k\times\R^{d-k}\to M$
with image $O $ such that the pull-back $j_O^*(X^H)$ is constant;
\vskip1mm\noindent
(iv) if moreover $O$ is compact, there exists a symplectic diffeomorphism $\J$ from  a neighborhood of $O$ in $\Sig_d$ to
a neighborhood of the zero set
$Z=\T^d\times\{0\}$
in $T^*\T^d$ such that  $\J(O)=Z$ and the pull-back $F\circ\J\inv$ and $H\circ \J\inv$ depend
only on the action  variable;
\vskip1mm\noindent
(v) for each $c\in\R^\ell$ and
for  $d\in\{0,\ldots,\ell\}$, the (nonempty) connected components of the intersection $\Sig_d\cap F\inv(c)$
are $d$-dimensional isotropic tori or cylinders which are orbits of the joint flow and open in $\Sig_d\cap F\inv(c)$
(in particular, the orbits are embedded submanifolds of $M$).
\end{lemma}
\begin{proof}
$(i)$ The rank of $F$ is constant on the orbits of the joint flow, since
\begin{equation}\label{eq:comp}
F=F\circ \Phi_F(\tau,.),\qquad \forall \tau\in\R^\ell.
\end{equation}
Therefore  each orbit $O$ lies inside the symplectic submanifold $\Sig_d$,  where $d=\rk F_{\vert O}$.
\vskip1mm\noindent
$(ii)$ Obvious.
\vskip1mm\noindent
$(iii)$ Here we implicitly identify $T(\T^k\times\R^{d-k})$ with its canonical trivialisation to be able to speak of constant
vector fields.
Let $x\in O$ and assume, changing the ordering if necessary, that the map $\ha F=(f_1,\ldots, f_d)$ formed by the first $d$
components of $F$ has rank $d$ at $x$.
Then (\ref{eq:comp})  proves that $\rk \ha F(y)=d$ for each $y\in O$, and one easily sees that $O$ is the orbit of
$x$ under the action of the joint flow $\Phi_{\ha F}$. So our assertion directly comes from the Liouville theorem applied to
this latter action.
\vskip1mm\noindent
$(iv)$ Using item $(ii)$ of Definition~\ref{def:strongint}, (\ref{eq:comp}) again shows that one can assume that the rank of the restriction
$\ha F_{\vert O}$ is equal to $d$. Therefore the Liouville-Mineur-Arnold
theorem applies to $\ha F_{\vert \Sig_d}$
in the neighborhood of $O$ and yields the desired conclusion, since $\Sig_d$ is a $2d$--symplectic manifold by
item $(i)$ of Definition~\ref{def:strongint}.
\vskip1mm\noindent
$(v)$ Obvious using the previous items.
\end{proof}

The next definition  underlines the geometric features of a strongly integrable system.

\begin{Def}\label{def:objects}
 Let $H$ be a strongly integrable Hamiltonian on $(M^{2\ell},\Om)$ with complete strong integral $F$ and associated partition
 $(\Sig_d)_{0\leq d\leq \ell}$.
\vskip1mm
-- For $d\in\{1,\ldots,\ell\}$, a {\em $d$-action-angle chart} is a pair $(\U,\J)$, where $\U$ is a subset of $\Sig_d$, open in
$\Sig_d$, and $\J$ is a symplectic diffeomorphism from $\U$ to a subset
of $T^*\T^d$ of the form $\T^d\times B^d$, where $B^d$ is some open ball centered at the origin in $\R^d$,
such that $F\circ \J\inv$ and $H\circ\J\inv$ depend only on the action variables.
\vskip1mm
--  The {\em action-angle domain} $\jD$ of $F$ is the union of all the domains $\U$ of action-angle charts. An invariant torus
contained in $\jD$ will be called a {\em proper torus}, so the action-angle domain is the union of all proper tori.
\vskip1mm
-- A non-compact orbit of the joint flow will be called a {\em cylinder}. We define the {\em cylinder domain} $\jC$ as the union of all cylinders.
\vskip1mm
-- A {\em neutral torus} is an invariant torus of $X^H$ which is contained in a cylinder.
We define the {\em neutral domain} $\jN$ as the union of all neutral tori.
\vskip1mm
-- An {\em asymptotic cylinder} is a cylinder which does not contain any neutral torus.
We define the {\em asymptotic domain} $\jA$ as the union of all asymptotic cylinders.
\end{Def}

Remark that the notions of action-angle domain, proper tori and cylinders are purely geometric and depend only on the integral
$F$, while the notions of neutral tori and asymptotic cylinders  are dynamical ones and depend also on the Hamiltonian $H$.
Not also that a neutral torus is not an orbit of the joint flow.

We already noticed that the notion of constant vector field makes sense on the standard torus or cylinder $\T^k\times\R^{d-k}$.
In the following, when a vector field on an embedded torus or cylinder is conjugate to a constant vector field on the standard model,
we say that it is {\em linearizable}.

\begin{lemma}\label{lem:struct3}  Let $H$ be a strongly integrable Hamiltonian on $(M^{2\ell},\Om)$ with complete strong integral $F$.
Then a cylinder of the system is either asymptotic or completely foliated by neutral tori.
The subsets $\jD$, $\jN$ and $\jA$ are pairwise disjoint, and the following equalities
$$
M=\jD\cup\jC, \qquad \jC=\jN\cup\jA,
$$
hold true. The vector fields on the proper tori, cylinders and neural tori are linearizable.
\end{lemma}
\begin{proof} Consider a cylinder $C$ of dimension $d$.  Then there exist $k\geq 1$ and an embedding $\I$ from $\T^k\times\R^{d-k}$
to $C$ such that, up to the canonical identification of the tangent space
$T_x(\T^k\times\R^{d-k})$ with $\R^k\times\R^{d-k}$,
the vector field $\I^*X^H(x)$ reads
\begin{equation}\label{eq:Liouvilleneut}
(v_1,\ldots,v_k,v_{k+1},\ldots,v_d),\qquad v_j\in\R,
\end{equation}
with $v_i$ independent of $x$. Clearly $C$ contains a neutral torus if and only if $v_{k+1}=\cdots=v_d=0$.
In this case, $C$ admits a foliation $(T_a)_{a\in \R^{d-k}}$ by the parallel invariant tori of equation $(v_{k+1},\ldots,v_d)=a$,
which are therefore all neutral tori.
In the case where one of the $d-k$ last
components of $X$ does not vanish, $C$ contains no compact invariant subset and thus is asymptotic.
This proves our first assertion, together with $\jC=\jN\cup\jA$ and $\jN\cap\jA=\emptyset$.

Now remark that the definition of an action-angle domain ensures that it cannot intersect a cylinder, since it
is entirely foliated by compact orbits of the joint flow $\Phi_F$, so
$\jD\cap \jC=\emptyset$, and thus also  $\jD\cap \jA=\emptyset$ and $\jN\cap \jC=\emptyset$.

Then, each point $x\in M$ belongs to its orbit under the joint flow, which is a proper torus or a cylinder.
This proves that $M=\jD\cup\jC$. The linearizability property is immediate by lemma \ref{lem:struct1} for proper tori and cylinders and
by (\ref{eq:Liouvilleneut}) for neutral tori.
\end{proof}

In view of the last lemma, we say that a cylinder is {\em neutral} when it contains a neutral torus, so a cylinder is either asymptotic
or neutral. We now briefly examine the case when the integral $F$ is no longer assumed to be complete

\begin{lemma}\label{lem:struct4} Assume that $F$ is a strong integral on $(M^{2\ell},\Om)$ and that $M^*$ is a compact subset
of $M$, invariant under the joint flow $\Phi_F$. Then for each $c\in\R^\ell$ and
for  $d\in\{0,\ldots,\ell\}$, the (nonempty) connected components of the intersection $\Sig_d\cap F\inv(c)\cap M^*$
are $d$-dimensional isotropic tori or cylinders which are orbits of the joint flow and open in $\Sig_d\cap F\inv(c)$.
The set $M^*$ admits a partition by proper tori, asymptotic cylinders of neutral tori on which the Hamiltonian vector
field is linearizable.
\end{lemma}

\begin{proof} Simply note that for each $x\in M^*$, the solution of $X^{f_i}$ of initial condition $x$ is defined over $\R$.
This easily implies that the joint flow restricted to $M^*$ is well-defined and complete.  The rest of the proof follows the same lines
as those of lemmas \ref{lem:struct1} and \ref{lem:struct3}.
\end{proof}


\subsubsection{Topological entropy and strongly integrable systems}

Using now our setting, we can give a short proof of Paternain's result.

\vskip2mm

\noindent{\bf Theorem (Paternain).} {\it Let $(M^{2\ell},\Om)$ be a symplectic manifold. If $H$ is a proper function $M\to\R$ which defines
a strongly integrable Hamiltonian system on $M$,   and if $M^*$ is a compact subset of $M$
invariant under the flow of $X^H$, then the topological entropy of the restriction of this flow to $M^*$
vanishes.}

\vskip2mm

\begin{proof} We refer to \cite{HK} for an excellent presentation of the notion and basic properties of topological entropy.
We use the so-called variational principle, which asserts that
the topological entropy of an homeomorphism $\ph$ of a compact metric space
is the upper bound of the metric entropies relative to the invariant measures:
$$
\htop(\ph)=\sup_{\mu\in\M(\ph)} {\rm h}_\mu(\ph)
$$
where $\M(\ph)$ is the (nonempty convex compact) set of all probability measures invariant by $\ph$.
One readily deduces from the ergodic decomposition theorem that the upper bound is indeed reached
within the set $\M_e(\ph)$ of {\em ergodic} invariant measures.

Let $\ph$ be the time-one flow of $X^H$, which is well-defined since $H$ is proper.
Let $a\in\R$ be fixed. We will first prove that the topological entropy of the restriction $\ph_a$  of $\ph$ to the compact
level $H\inv(\{a\})$ vanishes. By the variational principle,  this amounts to proving that
${\rm h}_\nu(\ph_a)=0$ for all $\nu\in\M_e(\ph_a)$.

Let $F$ be a strong first integral for $H$.
Since $H$ is invariant under the joint flow $\Phi_F$, by lemme \ref{lem:struct4} the level $H\inv(\{a\})$ admits a partition by
proper tori, asymptotic cylinders or neutral tori on which the vector field $X^H$ is linearizable.
Consider  $\nu\in\M_e(\ph_a)$. Since each element of the previous partition
is an invariant set for $\ph_a$, the support of $\nu$ is contained in one of them, which we denote by
$O$, so $\nu(O)=1$. Let $d$ be the dimension of $O$.

Assume first that $O$ is an asymptotic cylinder and consider a compact subset $K$ of $O$ such that $\nu(K)>0$ (such a compact
exists by regularity of Borel measures). Then equation (\ref{eq:Liouvilleneut}) shows the existence of an
increasing  sequence $(n_k)_{k\in \N}$  of integers such that
$\ph_a^{n_k}(K)\cap \ph_a^{n_{k'}}(K)=\emptyset$ when $k\neq k'$. This contradicts the assumption that $\nu(O)=1$.

Therefore $O$ is necessarily either a proper torus or a neutral torus $T$ on which the system is conjugate to the time-one map
of a constant vector field.
Since this latter system is an isometry, its topological entropy vanishes, and so does the topological entropy of
${\ph_a}_{\vert T}$. Using  the variational principal again, one sees that the metric entropy ${\rm h}_\nu({\ph_a}_{\vert T})$ vanishes,
and since $T$ contains the support of $\nu$ this proves that ${\rm h}_\nu(\ph_a)=0$. Finally, the variational principle proves that
$\htop(\ph_a)=0$.

To conclude, it suffices now to see that the invariant set $M^*$ is the union of the invariants sets $M^*\cap H\inv(\{a\})$ for $a\in \R$,
and to apply the same argument : if $\nu$ is an ergodic  measure invariant under $\ph_{\vert M^*}$, then its support is contained
in some $M^*\cap H\inv(\{a\})$ on which $\ph$ has zero topological entropy, so the metric entropy ${\rm h}_\nu(\ph_{\vert M^*})$
vanishes and so does the topological entropy of $\ph_{\vert M^*}$.
\end{proof}

Note that the proof of the previous theorem is even simpler if one uses the Bowen formula.


\subsection{Decomposability and weak complexity of Hamiltonian systems}\label{subsec:decomp}
Decomposable Hamiltonian systems are particular cases of strongly integrable systems, with additional assumptions on the dynamical
asymptotic behaviour of orbits.

\subsubsection{Decomposability}
In order to formulate the definition of decomposability, we first need to introduce the notion of maps with contracting fibered
structure.

\begin{Def}\label{def:contract}
Let $(E,d)$, $(X,\de)$ be metric spaces and consider two continuous maps  $\ph:E\to E$ and $\psi:X\to X$.
We say that $(E,\ph)$ has a contracting fibered structure over $(X,\psi)$ when the following conditions hold true.
\vskip1mm
(i) $E$ is metrically fibered over $X$ : there exists a surjective continuous map $\pi:E\to X$, a metric space $(F,\ha d)$
 and a finite open covering $(U_i)_{1\leq i\leq m}$ of $X$ such that for each $i$ there exists an isometry
$\phi_i:\pi\inv(U_i)\to U_i\times F$ (this latter space being equipped with the product metric), such that
$$
\pi(\phi_i\inv(x,y))=x,\qquad \forall (x,y)\in U_i\times F.
$$
We write $\phi_i(z)=(\pi(z),\varpi_i(z))\in U_i\times F$.
 \vskip1mm
(ii) $(X,\psi)$ is a factor of $(E,\ph)$ relative to $\pi$: $\psi\circ \pi=\pi\circ \ph$.
 \vskip1mm
(iii) If $z,z'$ are two points of $E$ such that there exists $i$ and $j$ in $\{1,\ldots,m\}$ such that $z,z'\in\pi_i\inv(U_i)$
and $\ph(z),\ph(z')\in\pi_j\inv(U_j)$, then
$$
\ha d\big(\varpi_j(\ph(z)),\varpi_j(\ph(z'))\big)\leq \ha d\big(\varpi_i(z),\varpi_i(z')\big).
$$
\end{Def}

Maps with contracting fibered structure are natural generalizations of  diffeomorphisms restricted to the stable manifolds of their normally hyperbolic invariant manifolds. Indeed, if a diffeomorphism $\ph$ of a manifold $M$ admits a compact invariant manifold $N$
which is normally hyperbolic, then its stable manifold $W^+(N)$ admits an invariant foliation by the stable manifolds of the points of $N$. Moreover,
there exists a projection $\pi$ from a neighborhood $E$ of $N$ in $W^+(N)$ to $N$, which to each point $x$ associates the unique point $n\in N$
such that $x\in W^+(n)$. It is not difficult to see that one can choose a metric on $M$ in such a way
that $(E,\ph)$ admits a contracting fibered structure over $(N,\ph_{\vert N})$.

We also need a definition enabling us to control the behaviour of a strongly integrable system on its neutral domain.

\begin{Def}
Let $(M^{2\ell},\Om,H,F)$ be a strongly integrable system, with $F$ complete. We say that the neutral domain $\jN$ is {\em regular} when
for $d\in\{0,\ldots,\ell-1\}$, the set $\jN_d=\jN\cap\Sig_d$
it admits a finite or countable covering $(\U_i)_{i\in I}$ satisfying the following two properties for $i\in I$:
\vskip1mm
\noindent (i)  there exists a diffeomorphism $\psi_i$ $\U_i\to T^{d_i}\times \ov B^{d'_i}(0,1)$, with $d_i\leq d-1$ and $d_i+d'_i\leq 2d$;
\vskip1mm
\noindent (ii) the time-one Hamiltonian flow is conjugate by $\psi_i$ to the following normal form
\begin{equation}\label{eq:neutform}
(\th,r)\mapsto(\th+\om(r),r),\qquad (\th,r)\in T^{d_i}\times B^{d'_i}(0,1).
\end{equation}
\end{Def}

Examples of systems with regular neutral domains will be given in \cite{M2}.
Given a vector field $X$ on a manifold $M$, let $x\in M$ et let $\ga : I\to M$ be the solution of $X$
such that $\ga(0)=x$.  Recall that the $\om$-limit set of  $x$ is the set of points
$y$ such that there exists a increasing sequence $(t_n)_{n\geq 0}$ in $I$ with $t_0>0$, such that $\lim_{n\to\infty}\ga(t_n)=x$.

\begin{Def} Let $(M^{2\ell},\Om,H,F)$ be a strongly integrable system, with $F$ complete.
We define the $\om$-limit domain $\jL$ of the system as the union of all $\om$-limit  of points
of the asymptotic domain $\jA$.
\end{Def}

We are now in a position to give our main definition.

\begin{Def}\label{def:decomp}
 Let $(M,\Om)$ be a symplectic manifold of dimension $2\ell$. We say that a Hamiltonian function
$H\in C^\infty(M,\R)$ is {\em decomposable} when it admits a strong  integral $F$ and when in addition:
\vskip1mm\noindent
(i) the neutral domain $\jN$ is regular;
\vskip1mm\noindent
(ii) for $d\in \{0,\ldots,\ell-1\}$,  there exists a neighborhood $\jV_d$ of $\jL_d:=\jL\cap\Sig_d$, invariant under the Hamiltonian flow,
such that, setting
$$
W^+_{\jV_d}(\jL_d)=\{x\in \jV_d\mid \om(x)\subset \jL_d\},
$$
the system $\Big(W^+_{\jV_d}(\jL_d), \ph_{\vert W^+_{\jV_d}(\jL_d)}\Big)$ has a contracting fibered structure over $(\jL_d,\ph_{\vert \jL_d})$.
\end{Def}

Clearly, any Morse function of a symplectic surface defines a decomposable system. More generally, Hamiltonian systems with
Eliasson non-degenerate integrals very often define decomposable systems. To see this we first need an auxiliary definition.
Let $(M,\Om,H,F)$ be an integrable system, with $F$ non-degenerate.  Let $x\in \Sig_d$, $0\leq d\leq\ell-1$.
Then, with the notation of Section \ref{sssec:int}, the Hamiltonian $H$ pass to the quotient in the reduction process by the orbits of the joint flow of
$\ha F$, we denote  the quotient
Hamiltonian function on the level $\ha F\inv \{0\}$ by  $\til H_0:S_0\to\R$. Clearly $d_{x}\til H_0=0$, so the Hessian $q_x=d^2_xH_0$
is a well-defined quadratic form of $Q(\R^{2(\ell-d)})$. We denote by $Q_x$ the Cartan subalgebra spanned by the Hessians of
the components of $\til F_0$ at $x$. As $q_x$ commutes with each of these components, $q_x\in Q_x$.

We say that $H$ is {\em dynamically coherent with $F$ at $x$} when all the coefficients in the linear development $q_x=\sum_{i=1}^{2(\ell-d)} \al_i q_i$ are non zero.
It is easy to see that this last condition is satisfied when the coefficients of the developpement on one single  basis of $Q_x$
are non zero. We then say that $H$ is dynamicall coherent with $F$ when it is dynamically coherent at each point $x$. It turns out that dynamically coherent non-degenerate systems are simple and important examples of decomposable systems,
as will be proved in \cite{M2} where more details and specific examples will be given.


\subsubsection{Weak complexity of decomposable Hamiltonian systems}

We begin with an auxilliary proposition on maps with contracting fibered structure.

\begin{prop}\label{prop:contract} Let $(E,d)$, $(X,\de)$ be metric spaces, and $\ph:E\to E$, $\psi:X\to X$ be continuous maps,
such that $(E,\ph)$ admits a contracting fibered structure over $(X,\psi)$.
Then
$$
\L(\ph)=\L(\psi).
$$
\end{prop}

\begin{proof} We already know that $\L(\ph)\geq \L(\psi)$ by the factor property.
To prove the converse inequality, consider a finite open covering $(U_i)_{i\in I}$ of $X$ adapted to the fibered structure
and let $\eps_0>0$ be the Lebesgue number of this covering (so each set of diameter less than $\eps_0$ for the metric $\de$
is contained in one of the open sets $U_i$).

\vskip1mm

Let now $N\geq1$ be fixed, choose $\eps<\eps_0/2$ and consider a ball $B^X\subset X$ of  $\de_N^\psi$--radius less than
$\eps$. In particular, $B^X$ has diameter less than $\eps_0$ (for $\de$) and is therefore contained in an element $U_{i_0}$ of
the covering. Consider then a ball $B^F$  of radius $\eps$ in the fiber $(F,\ha d)$.
As $B^X\subset U_{i_0}$, one can define the set
$$
P=\phi_{i_0}\inv \big(B^X\times B^F\big).
$$
We want to prove that $P$ has diameter less than $2\eps$ for the distance $d_N^\ph$.

\vskip1mm

For $z,z'$ in $P$, let $x=\pi_{i_0}(z)$ and $x'=\pi_{i_0}(z')$, then $x$ and $x'$ lie in $B^X$.
Note that for $0\leq k\leq N$, $\psi^k(B^X)$ has diameter less then $\eps_0$ for $\de$, and so is contained in some open set $U_{i_k}$
of the covering.
So, for $0\leq k\leq N-1$, the fibered structure yields the equality:
$$
d(\ph^k(z),\ph^k(z'))=\Max \Big(\de\big(\psi^k(x), \psi^k(x')\big),
\ha d\big(\varpi_{i_k}(\ph^k(z)), \varpi_{i_k}(\ph^k(z'))\big)\Big).
$$
Now by induction, using the inclusion $\psi^k(B^X)\subset U_{i_k}$:
$$
\ha d\Big(\varpi_{i_k}(\ph^k(z)), \varpi_{i_k}(\ph^k(z'))\Big)
\leq \ha d\Big(\varpi_{i_0}(z), \varpi_{i_0}(z')\Big)<2\eps
$$
and on the other hand
$
\de\big(\psi^k(x), \psi^k(x')\big)<2\eps
$
since $x,x'\in B^X$, so
$$
d(\ph^k(z),\ph^k(z'))<2\eps.
$$
This proves that $P$ has diameter less than $2\eps$ for $d_N^\ph$. We denote by $P(B^X,B^F)$ this set.

\vskip1mm

We now fix a minimal covering $B^X_1,\ldots,B^X_n$ of $X$ by balls of radius  $\eps$ for $\de_N^\psi$, and
a finite covering $B^F_1,\ldots,B^F_m$ of the fiber $F$ by balls radius $\eps$ for $\ha d$.
To each pair $(B^X_i, B^F_j)$, we associate the subset $P_{ij}=P(B^X_i, B^F_j)$ of $E$. It is easy to see that $(P_{ij})_{1\leq i\leq n, 1\leq j\leq m}$ is a covering of $E$ by subsets of diameter less than $2\eps$ for $d_N^\ph$,  which shows that
$$
G_N(E,\ph,2\eps)\leq m\, G_N(X,\psi,\eps)
$$
and proves that $\L(\ph)\leq \L(\psi)$.
\end{proof}

\vskip3mm

We are now in a position to state and prove the main result of this section. Recall that if $(M,\Om,H,F)$ is an integrable system, the {\em torsion}
of an invariant Lagrangian torus $T$ is defined as the rank of the Hessian of the normal form of $H$ in any action-angle chart.

\begin{thm}\label{th:decomp}
Let $(M,\Om,H,F)$ be a decomposable system, with $M$ compact and of dimension $2\ell$, and let $\ph$ be its time-one map. Then
$$
\LF(\ph)\leq \ell.
$$
If the system admits a Lagrangian torus with non-degenerate torsion, then $\LF(\ph)=\ell$.
\end{thm}

\begin{proof} The system is strongly integrable, let $(\Sig_d)_{0\leq d\leq\ell}$ be the decomposition associated with $F$ and denote by
$\jD_d,\jN_d,\jA_d$ the intersections of $\jD,\jN,\jA$ with $\Sig_d$ respectively.
We will consider the restriction of $\ph$ to these domains and prove a suitable inequality in each of them.

\vskip2mm

{\em $\bullet$ The proper action-angle domains $\jD_d$.} The domain $\jD_d$ is a countable union of domains of proper action-angle
charts. In each of those domains, the time-one map $\ph$ is conjugate to a system in action-angle form on a subset of $\T^d\times\R^d$, whose complexity index is at most $d$. So by the restriction and the $\sig$--union properties,  $\LF(\ph,\jD_d)\leq d$.

\vskip2mm

{\em $\bullet$ The neutral domain $\jN$.}  First note that the normal form (\ref{eq:neutform}) shows that the weak complexity index
of $\ph$ on the domain $\U_i$ is less than $d-1$. This is an immediate consequence of the proof of proposition \ref{prop:actang} since
the rank of the map $\om$ is less than $d-1$. Therefore, the regularity assumption ensures the existence of a finite or countable covering
of $\jN_d$ by domains $\U_i$ on which $\LF(\ph,\U_i)\leq d-1$. As a consequence,  $\LF(\ph,\jN_d)\leq d-1$.

\vskip2mm

{\em $\bullet$ The asymptotic domain $\jA$.}  We will prove the inequality
$$
\LF(\ph,\ov\jA_d)\leq d-1,
$$
by induction on $d\geq 1$. For this we first need the following lemma.

\begin{lemma}\label{lem:struct5} For $d\in\{0,\ldots,\ell\}$, we set
$$
\ha \Sig_d=\bigcup_{0\leq d'\leq d} \Sig_{d'}.
$$
We consider  a compact invariant subset $M^*\subset M$, invariant under the joint flow. Then the following properties hold true.

\vskip1mm\noindent
(i) For $d\in\{0,\ldots,\ell\}$, the set $\ha \Sig_d$  is closed and contains the closure $\ov {\Sig_d}$.

\vskip1mm\noindent
(ii) Let $d\in\{1,\ldots,\ell\}$ and consider a $d$-dimensional asymptotic cylinder $C$ contained in $M^*$.
Then $\ov C\setm C$ is nonempty and
contained in
$\ha \Sig_{d-1}\cap M^*$.

\vskip1mm\noindent
(iii)   For $d\in\{1,\ldots,\ell\}$, the union $\om(\jA_d)$ of the $\om$-limit sets of the points of $\jA_d$ is contained in
$\ha\Sig_d$.
\end{lemma}

\begin{proof}
$(i)$ By lower-semicontinuity of the rank, if $x\in \ov \Sig_d$ then $\rk F(x)\leq d$. This proves that
$$
\ov{\Sig_d}\subset \bigcup_{0\leq d'\leq d} \Sig_{d'}.
$$
Now
$
\ov{\bigcup_{0\leq d'\leq d} \Sig_{d'}}=\bigcup_{0\leq d'\leq d} \ov{\Sig_{d'}}\subset \bigcup_{0\leq d'\leq d} \Sig_{d'}
$,
so the union $\ha \Sig_d$ is a closed set.
\vskip1mm\noindent
$(ii)$ By compactness of $M^*$, $\ov C\setm C$ is nonempty. By the previous lemma one sees that $\ov C\subset \ha \Sig_{d}$,
it is therefore enough to prove that a point $x$ of $\Sig_d\setm C$
cannot be  a limit of points of $C$. Let $c$ be the value of $F$ on $C$.

Assume first that $F(x)\neq c$, then by continuity $x$ cannot be a limit of points of $C$. If now $F(x)=c$, then $x$ is contained in its orbit $O$ under the joint flow,  which is also a connected
component of $\Sig_d\cap F\inv(c)$ and therefore closed in $\Sig_d\cap F\inv(c)$. This proves that
$\ov O\cap C=\emptyset$, since $\ov O\subset \ha\Sig_d\cap F\inv(c)$.

As a consequence, the set $\ov C\setm C$ is contained in $\ha \Sig_{d-1}$. As $C$ is obviously invariant under the joint
flow, $\ov C\setm C$ is also invariant, thus it is the union of orbits of dimensions $\leq d-1$. A joint orbit $O$ of minimal
dimension cannot be an asymptotic cylinder, otherwise $\ov O\setm O$ would be contained in $\ov C\setm C$ and would contain
orbits of dimension strictly smaller than that of $O$.
\vskip1mm\noindent
$(iii)$ In view of (ii), it is enough to prove that if $x$ is in an asymptotic cylinder $C$, then the $\om$-limit set $\om(x)$ is contained
in $\ov C\setminus C$, that is $\om(x)\cap C=\emptyset$. This is an immediate consequence of equation (\ref{eq:Liouvilleneut}), since when
one of the components $v_i$ does not vanish, no point of $C$ can be a limit point of a point of $C$.
\end{proof}

\vskip1mm

We now turn back to the induction. With the notation of Definition \ref{def:decomp}, note first that if $\ha\jV_k:=\jV_0\cap\cdots\cap \jV_k$, then
$$
\jA_d\subset \bigcup_{m\in\N}\ph^{-m}(\ha \jV_{d-1}).
$$
Indeed,  if $x\in\jA_d$, then by the previous lemma
$$
\om(x)\subset \ha\Sig_{d-1}\cap \jL=\jL_0\cap\cdots\cap\jL_{d-1}.
$$
As a consequence, there exists $t\geq 0$ such that $\ph^t(x)\in \ha \jV_{d-1}$. Therefore, as each neighborhood $\jV_k$ is invariant
under the Hamiltonian flow, $\ph^N(x)\in\ha \jV_{d-1}$ for $N$ large enough. This proves our claim.

\vskip1mm

-- Let us prove that $\LF(\ph,\jA_1)=0$.  We have seen that $\jA_0=\cup_{m\in\N}\ph^{-m}(\ha \jV_{0})$, and we know that $\jL_0$ is a subset
of $\Sig_0$, which is finite by compatness. Therefore $\LF(\ph,\jL_0)=0$. Since $(\jV_0,\ph_{\vert \jV_0})$ has a contracting fiber structure over
$(\jL_0,\ph_{\vert \jL_0})$, this proves that $\LF(\ph,\jV_0)=0$ and our claim follows from the restriction and $\sig$-union properties.

\vskip1mm

-- Assuming now that $\LF(\ph,\jA_k)\leq d-1$ for $1\leq k\leq d$, we will prove that $\LF(\ph,\jA_{d+1})\leq d$. Clearly, it is enough to prove that
$\LF(\ph, \ha \jV_{d})\leq d$. Remark that
$$
\LF(\ph, \ha \jV_{d})=\Max_{1\leq k\leq d}\LF(\ph, \jV_k)=\Max_{1\leq k\leq d}\LF(\ph, \jL_k).
$$
Now $\LF(\ph, \jL_k)\leq \LF(\ph, \Sig_k)$ since $\jL_k\subset \Sig_k$. Moreover, $\Sig_k=\jD_k\,\cup\,\jN_k\,\cup\,\jA_k$, therefore
$
\LF(\ph, \Sig_k)\leq k
$
by the previous results on $\jD$ and $\jN$ and by the induction hypothesis on $\jA_k$. This proves that $\LF(\ph,\jA_d)\leq d-1$.

\vskip2mm

$\bullet$ Now a simple finite union argument shows that $\LF(\ph)\leq\ell$.

\vskip2mm

Finally, when $\ph$ admits a Lagrangian torus with non-degenerate torsion, the system is locally conjugate to a system
whose weak index is equal to $\ell$, which proves that the global weak index is indeed equal to $\ell$.
\end{proof}


\section{Complexity indices of hamiltonian systems on surfaces}

Hamiltonian systems on surfaces are geometrically integrable, still their
level sets may be extremely intricate. Here we analyze the complexity of Hamiltonian flows associated with
Morse Hamiltonian functions only. The main result of this section is the following.

\vskip1mm

\begin{thm}\label{thm:indham}
Let $\S$ be a smooth compact symplectic surface with boundary $\d \S$, and let $H$ be a smooth Hamiltonian function on $\S$ with non-degenerate critical points, which is constant and regular on each connected component of $\d\S$.  Let $\phi$ be the time-one map of the Hamiltonian vector field of $H$. Then $\L(\phi)\in\{0,1\}$ if $H$ has no critical point of index~$1$, and $\L(\phi)=2$ if $H$ admits at least
one such critical point.
\end{thm}

\vskip1mm

One can be more precise in the case when $H$ has no critical point of index $1$.
First note that this can only happen when the surface $\S$ is a disc, an annulus or a sphere. In the first case, $H$ has a single critical point
of index $0$ or $2$, in the second one $H$ has no critical point at all, and in the last one $H$ has two critical points, one of index $0$ and one of index $2$. The complement of the critical points is entirely foliated by periodic orbits on which $\phi$ is conjugated to a rotation, whose angle
depends on the orbit. Then one easily sees that $\L(\phi)=0$ if and only if this dependence is trivial, that is  $\phi$ globally acts as a rotation
on the complement of the critical points, with constant angle (this will be stated more precisely and proved below).

\vskip3mm

The main difficulty therefore comes from the critical levels which contain index $1$ critical points. Such a level may contain several critical points
(a``polycycle''), which makes the study of the index more complicated. In order to overcome this difficulty we introduce a method of {\em
dynamical desingularization} which amounts to semi-conjugating the system in suitable  ``partial neighborhoods'' of a polycycle to a model
system (a $p$--model on an annulus) in such a way that the complexity index is preserved. The main task of this section will be first to
define these $p$--models and compute their complexity index (under suitable conditions)  and second to prove that such dynamical desingularizations allows one to get coverings of neighborhoods of polycycles with computable complexity index. This desingularization method can be extended
to multidimensional systems, as we will show in \cite{MM}.


\subsection{The singular model on the annulus}
In the following we denote by $d$ the usual metric on $\T$ and we write $\jA$ for the compact annulus $\T\times[0,1]$.
We equip $\jA$ with the canonical product metric, which we still denote by $d$ when there is no risk of confusion.

\begin{Def}\label{def:model}
Given an integer $p\geq1$, we call {\em $p$--model}  on $\jA$ any continuous vector field of the form
\begin{equation}\label{eq:genform}
V(\th,r)=\Th(\th,r) \Dron{}{\th}
\end{equation}
on $\jA$, smooth on $\T\times\,]0,1]$, which satisfies  the following conditions:
\vskip 1mm
\noindent
{\rm (C1)}   $\Th(\th,r)>0$ if $r >0$, so each solution with $r>0$ is periodic;
\vskip 1mm
\noindent
{\rm (C2)}  the points $o_k=(k/p,0)$, $k\in\{0,\ldots, p-1\}$, are the only singular points of $V$ on $\jA$, that is $\Th(k/p,0)=0$
and $\Th(\th,0) > 0$ if $\th \notin\{ k/p\mid k\in\{0,\ldots, p-1\}\}$;
\vskip 1mm
\noindent
{\rm (C3)} in the neighborhood $\O_k$ of $o_k$  defined by
$\abs{\th-k/p}<1/(8p)$,  the function $\Th$ admits the following normal form
\begin{equation}\label{eq:genormform}
\Th(\th,r)=\ell_k(r)\sqrt{\rho_k (r)+(\th-\frac{k}{p})^2},
\end{equation}
where $\ell_k$ and  $\rho_k$ are positive smooth functions on $[0,1]$, with $\rho_k$ monotone increasing.
\vskip 1mm
\end{Def}

Note that the neighborhoods $\O_k$ are pairwise
disjoint. In the following a $p$--model will generally be denoted by a pair $(\jA,V)$.
Remark that a $p$--model have a well-defined continuous flow, smooth on $\T\times\,]0,1]$. We could indeed have worked with continuous flows instead
of vector fields, but is seems that the present framework makes the main ideas  more transparent.

\begin{figure}[h]
\begin{center}
\begin{pspicture}(4cm,5.3cm)
\psset{xunit=.8cm,yunit=.8cm}
\rput(2.5,4){
\pscircle(0,0){2}
\pscircle[linewidth=.4mm](0,0){1.2}
\pscircle[linewidth=.1mm](0,0){1.8}
\pscircle[linewidth=.1mm](0,0){1.6}
\pscircle[linewidth=.1mm](0,0){1.4}
\pscircle[fillstyle=solid,fillcolor=black](2.5,0){.07}
\pscircle[fillstyle=solid,fillcolor=black](-1.25,-2.16){.07}
\pscircle[fillstyle=solid,fillcolor=black](-1.25,2.16){.07}
\psarc[linewidth=.6mm]{->}(0,0){2}{58}{62}
\psarc[linewidth=.6mm]{->}(0,0){2}{178}{182}
\psarc[linewidth=.6mm]{->}(0,0){2}{298}{302}
\psline[linewidth=.6mm](1.45,.39)(2.41,.65)
\psline[linewidth=.6mm](1.45,-.39)(2.41,-.65)
\psline[linewidth=.6mm](-.388,1.45)(-.647,2.415)
\psline[linewidth=.6mm](-1.06,1.06)(-1.77,1.77)
\psline[linewidth=.6mm](-.388,-1.45)(-.647,-2.415)
\psline[linewidth=.6mm](-1.06,-1.06)(-1.77,-1.77)
\pscircle[linewidth=.01, fillstyle=solid,fillcolor=white](2,0){.3}
\rput(2,0){$\O_1$}
\pscircle[linewidth=.01, fillstyle=solid,fillcolor=white](-1,1.73){.3}
\rput(-1,1.73){$\O_2$}
\pscircle[linewidth=.01, fillstyle=solid,fillcolor=white](-1,-1.73){.3}
\rput(-1,-1.73){$\O_3$}
\rput(3,0){$o_1$}
\rput(-1.5,2.5){$o_2$}
\rput(-1.45,-2.45){$o_3$}
}
\end{pspicture}
\end{center}
\vskip-1.6cm
\caption{A $3$--model.}
\end{figure}
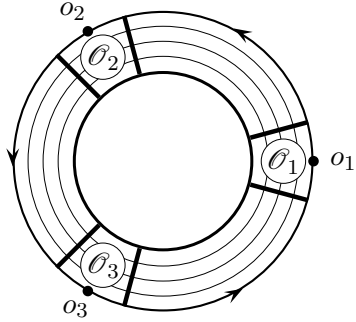

In order to facilitate the determination of the index of a $p$--model,  we have to add two technical conditions that we now define.
Given a $p$--model $(\jA,V)$, in the following we will denote by $\Phi:\R\times\jA\to\jA$ its flow, by $\ph_t$ the deduced time-$t$ map
and by $\ph$ the time-one map. We also set $\til\jA=\R\times[0,1]$ for the universal covering of $\jA$, which we endow with the coordinates
$(x,r)$.

\begin{Def}\label{def:torsion} {\bf (Torsion condition).} Consider a $p$--model $V$ on the strip $\jA$ and a lift $\til \Phi$ of its flow $\Phi$
to $\til \jA$, with associated maps $\til \ph_t$, $t\in\R$.
We say that $V$ satisfies the {\em torsion condition} when,
given $x\in\,\R$ and $0\leq r_1<r_2\leq 1$,
and setting $\til \ph_t(x,r_i)=(x_i(t),r_i)$, $i=1,2$, then the inequality $x_1(t) < x_2(t)$ holds true for each $t >0$
(so the vertical is twisted to the right by the map $\til \ph_t$).
\end{Def}

One easily checks that the previous definition makes sense, the twist condition being independent of the choice of the lift $\til\Phi$.

\vskip2mm

To introduce the second condition we first need to define what we call the {\em separation function} for two points on the same orbit of
a $p$--model $(\jA,V)$. With the same notation as above, we consider two points
$a=(\th,r)$ and $a'=(\th',r)$ of $\jA$, two lifts $\til a,\til a'$ located in the same fondamental domain of the covering $\jA$, and we
set $\til \ph_t(\til a)=(x(t),r)$ and
$\til \ph_t(\til a')=(x'(t),r)$. Then we define the separation of $a$ and $a'$ as the function $E_{a,a'}:\R\to \R$ defined by
$$
E_{a,a'}(t)=\abs{x'(t)-x(t)},
$$
so $E_{a,a'}$ is independent of the lift, $C^\infty$, non negative and  periodic (withe the same period as $a$ and $a'$). We are interested in the behaviour of the maxima of $E$.

\vskip2mm

We define here a {\em fundamental domain} for the flow $\Phi$ on $(\jA,V)$ as a subset $\K$ of $\jA$ of the form $\Phi([0,1],\De)=\cup_{t\in[0,1]}\ph_t(\De)$, where $\De$ is a vertical segment of equation $\th=\th_0$.

\begin{Def}\label{def:tameness}{\bf (Tameness).}  We say that a $p$--model $(\jA,V)$
is {\em tame} when there exists a fundamental domain $\K$ for $\Phi$ and a constant $\varpi>0$ such that, given two points $a$ and $a'$  on the
same orbit, then for each $t_0$ such that $E_{a,a'}(t_0)$ is maximum,
the points $\ph_{t_0}(a)$ and $\ph_{t_0}(a')$ are located inside the domain $\K$.
\end{Def}

Here the distance $d(a,a')$ is just the distance on the circle $\T$, since the points are on the same orbit.
 The tameness and torsion conditions will be used below to compute the complexity index of a $p$--model in a (quite) simple way.


\subsection{The complexity index of a $p$--model}
It turns out that the complexity index of a tame $p$--model with torsion does not depend on $p$. This will be the main result
of this section.

\begin{prop}\label{prop:pmodel}
Let $(\jA,V)$ be a tame $p$--model with torsion, $p\geq1$, and let  $\ph$ be its  time-one map. Then $\L(\ph)=2$.
\end{prop}

\begin{proof} We will first prove that $\L(\ph)\leq 2$ by exhibiting suitable coverings
of $\jA$, and then that $\L(\ph)\geq 2$ by finding separated sets. Let us introduce some notation.
\vskip1mm\noindent
-- Everywhere, when necessary, we consider the index $k$ as an element of $\Z_p$.
\vskip1mm\noindent
-- Given $r\in\,]0,\de]$, we denote by $\Ga_r$ the orbit with ordinate $r$ in $\jA$ and  by $T(r)$  the period  of motion on $\Ga_r$.
Note that $T(r) \to+\infty$ when $r\to 0$.
\vskip1mm\noindent
-- Due to the torsion condition, $T$ is a decreasing function from $]0,1]$ to $[q^*,+\infty[$, where $q^*$ is the period of motion on
$\Ga_1$. Therefore one can also label the orbits by their period: we write $C_q$ the orbit with period $q$, so $C_{T(r)}:=\Ga_r$.
\vskip1mm\noindent
-- We write $C_{\infty}$ for the boundary $r=0$.
\vskip1mm\noindent
-- Given two periods $q$ and $q'$ with $q'\leq q\leq+\infty$, we denote by $S_{q,q'}$ the annulus
bounded  by the curves $C_q$ and $C_{q'}$.
\vskip1mm\noindent
-- For $m\in\N$, we denote by $d_m$ the dynamical distance of order $m$ on $(\jA,d)$ associated with the map $\ph$ (where
$d$ is the usual product distance on $\jA$).

\vskip2mm

{\bf 1. Proof of $\L(\ph)\leq 2$.} Given $\eps>0$ and an integer $N$, we want to construct an $(N,\eps)$--covering of $\jA$.
We will have to discriminate between two different regimes for the system: close enough to the boundary $C_\infty$ the $N$ first iterates
roughly behave as those of a gradient system, and in particular do not experience any recurrence phenomenon, while on the complement
one has to take into account such phenomena together with  the drift between nearby orbits.
So we will split  $\jA$ into two $N$--depending annulus and separately construct $(N,\eps)$--coverings for these two domains.

\vskip2mm \noindent

{\bf 1.a.  Choice of suitable domains.} The following lemma will enable us to introduce a suitable
cutoff for the transition time and discriminate between the two regimes.

\vskip1mm

\begin{lemma}\label{lem:choice} Let $\eps>0$ be fixed.
For $k\in\{1,\ldots,p\}$, let  $\B_k$ be the ``block'' of $\jA$ limited by the  vertical
segments $\De^+_k$ and $\De^-_k$ of equations $\th=k/p-\eps/2$ and $\th=k/p+\eps/2$ respectively.
Then,  there exists a constant $\ka$ and an integer $N_0$  (both depending on $\eps$) such that if $N\geq N_0$,  for each index $k\in\{0,\ldots,p-1\}$:
$$
\ph^n(\De^+_k(\ka N))\subset \B_k,\qquad \forall n\in\{0,\ldots,N\},
$$
where we write $\De^-_k(q)$ for the intersection of the left vertical $\De^-_k$ of $\B_k$ with the annulus $S_{\infty,q}$.
\end{lemma}

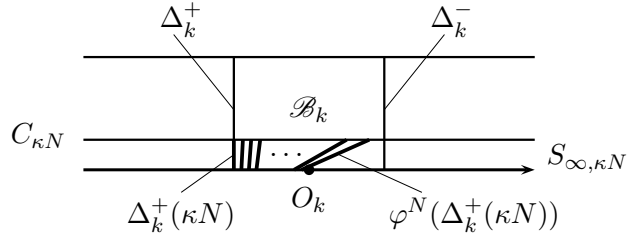
\begin{figure}[h]
\begin{center}
\begin{pspicture}(4cm,6.2cm)
\rput(2,4){
\psline[linewidth=.4mm]{->}(-3,0)(3,0)
\psline[linewidth=.3mm](-3,1.5)(3,1.5)
\psline[linewidth=.3mm](-3,.4)(3,.4)
\psline[linewidth=.3mm](-1,0)(-1,1.5)
\psline[linewidth=.3mm](1,0)(1,1.5)
\pscircle[fillstyle=solid,fillcolor=black](0,0){.07}
\psline[linewidth=.5mm](-1,0)(-1,.4)
\psline[linewidth=.5mm](-.9,0)(-.88,.4)
\psline[linewidth=.5mm](-.8,0)(-.77,.4)
\psline[linewidth=.5mm](-.7,0)(-.65,.4)
\psline[linewidth=.5mm](-.2,0)(.5,.4)
\psline[linewidth=.5mm](-.1,0)(.8,.4)
\rput(-.3,.2){$\ldots$}
\rput(0,.8){$\B_k$}
\rput(-1.7,-.6){$\De_k^+(\ka N)$}
\rput(2.2,-.6){$\ph^N(\De_k^+(\ka N)$)}
\rput(0,-.4){$O_k$}
\psline[linewidth=.1mm](-1.7,-.3)(-1,.2)
\psline[linewidth=.1mm](1.5,-.3)(.4,.2)
\rput(-1.7,2){$\De_k^+$}
\psline[linewidth=.1mm](-1.8,1.8)(-1,.8)
\rput(1.9,2){$\De_k^-$}
\psline[linewidth=.1mm](1.8,1.8)(1,.8)
\rput(3.7,.15){$S_{\infty,\ka N}$}
\rput(-3.6,.45){$C_{\ka N}$}
}
\end{pspicture}
\end{center}
\vskip-3.4cm
\end{figure}

\begin{proof}
Let us first compute the transition time in the bloc $\B_k$ as a function of $r$, that is the time $\tau_k(r)$ needed
to go from the entrance boundary $\De_k^+$ to the exit boundary $\De_k^-$, on the orbit $\Ga_r$. By condition
(C3) we have to integrate the linear equation
$$
\dot u=\ell_k(r)\sqrt{\rho_k(r)+u^2},
$$
which immediately yields
$
\tau_k(r)=\Frac{2}{\ell_k(r)}\,{\rm Argsh\,} \Big(\Frac{\eps}{\sqrt {\rho(r)}}\Big).
$
Therefore one gets the following equivalent when $r\to 0$:
$$
\tau_k(r)\sim_{r\to 0} -\Frac{1}{\ell_k(r)} \,\Log \rho(r).
$$
For small
enough $r$, the period $T(r)$ is clearly equivalent to the sum of the transition times in the blocks, and therefore
$$
\Frac{T(r)}{\tau_k(r)}\sim_{r\to 0} \Big(\sum_{i=1}^p\Frac{1}{\ell_i(0)}\Big)\ell_k(0):=\mu_k
$$
Let $\ka=[2\Max_{1\leq k\leq p}\mu_k]+1\in\N$, then for $r$ small enough
$
\tau_k(y)\geq \frac{1}{\ka}\, T(r)
$
for all $k\in\{1,\ldots,p\}$. Therefore for $N$ large enough $\ph^N(\De_k^-(\ka N))\in\B_k$ for all $k$ (recall that the period
is decreasing with $r$), which proves the proposition.
\end{proof}

\begin{figure}[h]
\begin{center}
\begin{pspicture}(4cm,5.5cm)
\rput(2,4){
\psset{xunit=1cm,yunit=1cm}
\pscircle[linewidth=.4mm](0,0){1.5}
\pscircle[linewidth=.4mm](0,0){.9}
\pscircle[linewidth=.4mm](0,0){1.35}
\rput(0.1,0){$\jA^*_{N}$}
\psline[linewidth=.2mm](.5,0)(1.1,0)
\rput(3,0){$\jA_{N}=S_{\infty,\ka N}$}
\psline[linewidth=.2mm](1.4,0)(1.8,0)
}
\end{pspicture}
\end{center}
\vskip-2.7cm
\caption{The two sub-annuli adapted to the different dynamical behaviours.}
\end{figure}
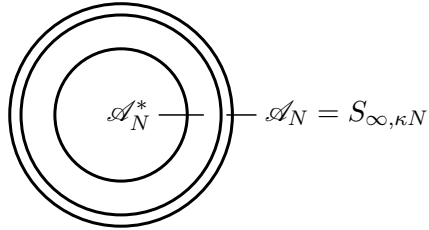

As we will see, the period $\ka N$ is the natural cutoff we were looking for. So we introduce the two $N$-dependent sub-annuli
$$
\jA_{N}=S_{\infty,\ka N},\quad \jA^*_{N}=\jA\setm \jA_{N}
$$
for which we will separately construct adapted coverings.

\vskip2mm

{\bf 1.b. Covering of the annulus $\jA_{N}$.} In this domain, on a time scale of length $N$, the system behaves almost like a gradient
vector field and our arguments will be quite similar to those of proposition \ref{prop:indseg}.  Given $\eps>0$ and $N$ large enough, we will contruct a covering of $\jA_{N}$, by subsets of $d_N$-diamter less than $\eps$, by separately considering the blocks $\B_k$ and their complement.

For $k\in\{1,\ldots,p-1\}$, we define the transition time between $\B_k$ and $\B_{k+1}$ on the boundary
$C_\infty$  as the smallest integer $\nu_k$ such that $\ph^{\nu_k}(a_k)\in \B_{k+1}$, where $a_k$ is the point of coordinates $(\frac{k}{p}+\frac{\eps}{2},0)$,
that is the intersection point of $\De_k^-$ with $C_\infty$.

Let $\nu=\Max_k\nu_k$. Then clearly for $N$ large enough the inclusion
$\ph^{\nu}(\De_k^-(\ka N))\subset \B_{k+1}$ holds true for each $k$, we will assume this condition fulfilled in the following.

We set $\B=\cup_{1\leq k\leq p}\B_k$.
By compactness, it is possible to find a {\em finite} number of subsets $B_1,\ldots, B_{i^*}$ with $d_\nu$--diameter
$\leq\eps$, which  cover $\jA_N\setm \B$. Moreover, one can obviously assume that each $B_i$ is contained in some connected component of
$\jA_N\setm \B$.

Now the point is that, due to our choice of the cutoff of period $\ka N$,  each iterate of rank $n\in\{\nu, \ldots, N\}$ of each domain $B_i$ is contained in some $\B_k$. Indeed, assume that $B_i$ is contained in the zone limited by the curves $\De_k^-(\ka N)$ and
$\De_{k+1}^+(\ka N)$  (according to the direct
orientation on $\T$). Then the iterate $\ph^n(B_i)$ is contained in the region limited by  $\ph^\nu(\De_k^-(\ka N))(\subset \B_{k+1})$ and
$\ph^N(\De_{k+1}^+(\ka N))$,  which is itself contained
in $\B_{k+1}$ by lemma~\ref{lem:choice}. Therefore, since the $d$--diameter of $\B_{k+1}$ is $\eps$:
$$
\diam_N(B_i)\leq \eps.
$$
We therefore have produced a good covering for $\jA\setminus \B$. The covering of $\B$ will be done in two steps.
Consider first the regions $\U_k$ in $\jA_{N}$ bounded by $\De_k^-(\ka N)$ and $\ph(\De_{k}^-(\ka N))$, where we assume $\eps$ small enough and $N$ large enough so that $\ph(\De_{k}^-(\ka N))\cap \B_{k+1}=\emptyset$.

We denote by $U_1,\ldots, U_{i^{**}}$ the nonempty intersections of the domains $B_i$ with the union
$\cup_{k\in\Z_p}\U_k$. So $i^{**}\leq i^*$ and
$$
\diam_N(U_i)\leq \eps,\qquad 1\leq i\leq i^{**}.
$$

Let $\jV_k$ be the region bounded by $\ph^{-N}(\De_k^-(\ka N))$ and $\De_k^-(\ka N)$ (relatively to the direct orientation of $\T$). By the same arguments as in the beginning, one
sees that $\jV_k\subset \B_k$.
Moreover the inverse images:
$$
B_{n,i}=\ph^{-n}(U_i),\qquad 1\leq n\leq N, \ 1\leq i\leq i^{**},
$$
form a covering of the region $\jV=\cup_{k\in\Z_p} \jV_k\ (\subset \B)$ and each of these subsets clearly satisfy $\diam_N(B_{n,i}\leq \eps$,
by construction.

Finally,  remark that for each $k$, the complement $\B_k\setm \jV$ satisfies
$
\ph^n(\B_k\setm \jV)\subset \B_k
$
for $0\leq n\leq N$, and therefore $\diam_N(\B_k\setm \jV)\leq\eps$.

Gathering the previous remarks, one sees that the subsets
$$
(B_i)_{1\leq i\leq i^*},
\quad(B_{n,k})_{1\leq n\leq N,\ 1\leq i\leq i^{**}},\quad
(\B_k\setm \jV)_{1\leq k\leq p},
$$
form a covering of $\jA_{N}$ and have $d_N$--diameter $\leq\eps$, which proves that
$$
G_N(\jA_{N},2\eps)\leq i^*+Ni^{**}+p.
$$
Therefore the ``complexity index'' on the $N$-dependent part $\jA_N$ is at most $1$. We will see that the main source of complexity is located
in its complement.

\vskip2mm

{\bf 1.c. Covering of the annulus $\jA^*_{N}$.}
The main step consists in estimating from above the minimal number of elements in a covering of
thin enough sub-annuli  $S_{q,q'}$ with subsets of $d_N$--diameter less than $\eps$.

\vskip1mm

\begin{lemma}\label{lem:fatten}
 Let $m\geq 1$ be a integer, and let $\eps>0$ be given. There exists positive constants
$c_1$  and $c_2$, depending only on $\eps$,
such that if  the pair $(q,q')\in [q^*,m]^2$ satisfy
$$
0\leq q'-q \leq \Frac{c_1\,\eps}{[m/q]}
$$
then the sub-annulus $S_{qq'}$ satisfies
$$
G_m(S_{qq'},\eps)\leq c_2\, q.
$$
\end{lemma}

\begin{proof}  We will first analyze the dynamics on a single curve $C_q$, and then deduce  from this study an estimate on the covering
number for a thin enough strip $S_{qq'}$.

\vskip2mm

1. Let $q\in [q^*,+\infty[$ be a period of the system, $q\leq m$, and consider the orbit $C_q$, let $\Phi:\R\times\jA\to \jA$ be the flow of $V$. Let $\la$ be the Lipschitz constant of $\Phi$ on the compact set  $[-1,1]\times\jA$. We will take advantage of the tameness property: let $I_q$ be the interval $C_q\cap \K$, where $\K$ is the fundamental domain introduced in definition \ref{def:tameness}. Consider two points $a\leq a'$ contained in
$I_q$. Then by the tameness property the maximum $\mu$ of the separation function $E_{a,a'}$ is achieved for $t$ such that $\ph_t(a)$ and $\ph_t(a')$ are
located inside $I_q$, and therefore $t\in[-1,1]$. As a consequence $\mu\leq \la\, d(a,a')$, and thus the $d_m$--diameter of $[a,a']$ is less than
$\la\, d(a,a')$ for all positive integers $m$.

Now we choose a finite covering of $I_q$ by consecutive subintervals $J_1,\ldots,J_{j^*_q}$ of $d$--diameter $\eps/(2\la)$.
As $I_q,\ph(I_q),\ldots,\ph^{[q]}(I_q)$ is a covering of $C_q$, one sees that the intervals $I_{ij}=\ph^i(J_j)$, ${0\leq i\leq [q], 1\leq j\leq j^*_q}$ form a covering of $C_q$ by subsets of $d_m$--diameter $\leq \eps/2$, for {\em each} integer $m$.

Note finally that the number $j^*_q$ is  bounded above by $j^*_{q^*}$, by the torsion property. Therefore, setting $c_2=2j^*_{q^*}$, for each $q\in[q^*,m[$,  each orbit $C_q$ admits a covering by at most $c_2\,q$ subsets whose  $m$--diameter is $\leq \eps/2$, for each positive integer $m$.

\vskip2mm

2. Now we fix a positive integer $m$. We will use the previous covering of a curve to fatten it a little bit and obtain a covering of a thin strip. Namely, given the initial perioq $q$, we want to find a period $q'\leq q$ such that for any pair of points $a\in C_q$ and $a'\in C_{q'}$ {\em with the same abscissa
$\th$}, the (maximal) difference of the abscissas of any pair of iterates $\ph^n(a)$ and $\ph^n(a')$, $n\in\{0,\ldots,m\}$, is at most $\eps/2$.

Assume that it is the case and consider again the covering of $C_q$ by the intervals $I_{ij}$.
Then, let $R_{ij}$  be the rectangle  limited by the curves $Cq$ and $C_{q'}$ and the vertical lines passing through
the extremities of $I_{ij}$. It is clear that these rectangles form  a covering of the strip $S_{qq'}$  and that the
$d_m$--diameter of each $R_{ij}$ is less than
$\eps$. This covering clearly has at most $c_2\,q$ elements, which is our claim.

So the problem is to analyze the mutual drift of the points $a$ and $a'$ on $C_q$ and $C_{q'}$ in order to choose
$q'$ close enough to $q$.
To this aim,  we fix a lift to the universal covering and consider the associated the lifted flow $\til\ph_s$, together
with lifts $\til a$, $\til a'$ of $a$ and $a'$ located on the same vertical. As usual, we set
$\til a(s)=\til\ph_s(a)=(x(s),r)$, $\til a'(s)=\til\ph_s(a')=(x'(s),r')$, so  $r'\geq r$ since $q'\leq q$. Given $t\geq 0$, let $t'$ be
 the time   needed for the point $a'$ to reach the vertical through $a(t)$, so $t'$ is characterized by the equality
$
x'(t')=x(t).
$
We set
$$
D(a,t)=t-t'
$$
so, by the torsion property,  $D(a,t)\geq0$. Moreover,  one easily checks that
$$
D(a,t_1+t_2)=D(a,t_1)+D(\ph_{t_1}(a),t_2),
$$
therefore $D(a,.)$ is an increasing function, which satisfies $D(a,k q)=k D(a,q)$ for each $a\in C_q$ and each positive integer
$k$. It is also easy to see that
$$
D(a,q)= q-q',\quad  \forall a\in C_q.
$$
Now let $\ell$ be the maximum of the function $\Th$ (see definition \ref{def:model}), that is the maximal length of the vector field $V$.
Then obviously
$$
0\leq x'(t)-x(t)\leq \ell\, D(a,t).
$$
We can now pass to the main estimates.  Fix a positive integer $m$ and, for $a\in C_q$, note that
$$
D(a,m)\leq  D(a,([\frac{m}{q}]+1)q)=([\frac{m}{q}]+1)(q-q').
$$
Consequently, for $0\leq n\leq m$,
$$
0\leq x'(n)-x(n)\leq x'(m)-x(m)\leq \ell \,([\frac{m}{q}]+1)(q-q').
$$
which proves our statement for $c_1=1/\ell$.
\end{proof}

\vskip2mm

We are now in a position to estimate the number  $G_{\ka N}(\jA_N^*,\eps)$ from above, from which we deduce
the estimate of $G_N(\jA_N^*,\eps)\leq G_{\ka N}(\jA_N^*,\eps)$. Set $k^*=[\ka N/q^*]-1$. For $1\leq k\leq k^*$ we introduce
the subset ${\bf S}_k\subset \jA_N$
formed by the curves $C_q$ such that
\begin{equation}\label{eq:defsk}
q\in \Big ]\Frac{\ka N}{(k+1)},\Frac{\ka N}{k}\Big],
\end{equation}
that is ${\bf S}_k=S_{\frac{\ka N}{k},\frac{\ka N}{(k+1)}}$. Clearly, the family $({\bf S}_k)$ covers $\jA_{N}$.

We want to apply lemma \ref{lem:fatten} with $m=\ka N$ and $q$ satisfying (\ref{eq:defsk}), to choose $q'\leq q$ close enough to $q$.
Remark that if $q$ satisfies (\ref{eq:defsk}), then
$$
\Big[\Frac{\ka N}{q}\Big]=k.
$$
Therefore,  by lemma \ref{lem:fatten}, if $q'-q\leq c_1\eps/k$, the strip $S_{qq'}$ satisfies
$$
G_{\ka N}(S_{qq'},\eps)\leq c_2\, q\leq c_2 \, \Frac{\ka N}{k}.
$$
This upper bound is therefore constant on ${\bf S}_k$. Now the strip ${\bf S}_k$ is covered
by the strips $(S_{q_{i+1},q_i})_{0\leq i\leq i^*(k)}$, with
$$
q_i=\Frac{\ka N}{k+1}+ i \Frac{c_1\eps}{k},\qquad i^*(k)=\Big[\Frac{\ka\, N}{c_1\,\eps\,(k+1)}\Big]+1\leq c_3\Frac{\ka\, N}{c_1\,\eps\,k}
$$
for $c_3>0$ large enough.
Therefore
$$
G_{N}({\bf S}_k,\eps)\leq c_2 \, \Frac{\ka N}{k}\,i^*(k)\leq  c_\eps\Frac{N^2}{k^2},\qquad c_\eps=\Frac{c_2c_3\ka^2}{c_1\eps}.
$$
Unsing now the fact that the strips ${\bf S}_k$ cover $\jA_N$, one gets
$$
G_{N}(\jA^*_N,\eps)\leq\sum_{k=1}^{k^*}G_{N}({\bf S}_k,\eps)\leq \sum_{k=1}^\infty c_\eps\Frac{N^2}{k^2}=\al_\eps N^2,
$$
with $\al_\eps=c_\eps\zeta(2)$. Therefore the complexity in $\jA_N^*$ is at most quadratic.
\end{proof}

\vskip1mm

{\bf 1.d. Final estimate.} We only have now to gather the estimates in $\jA_N$ and $\jA^*_N$ to get
$$
G_{N}(\jA,\eps)\leq \al(\eps) N^2+ o(N^2)
$$
which proves that $\L(\ph,\jA)\leq 2$.

\vskip4mm

{\bf 2. Proof of $\L(\ph)\geq 2$.} Given $\eps>0$ and an integer $N$, we want to find an $(N,\eps)$ separated set contained in $\jA$. It turns out that for $N$ large enough one can produce such a separated set contained (for instance) in the strip $S_{N/2,N/3}$ limited by the
curve $C_{N/2}$ from below and by  $C_{N/3}$ from above. This is reminiscent of the boundary layer phenomenon in fluid dynamics.

When $a$ and $b$ are two points on the same curve $C_q$, we denote by
$[a,b]$ the set of all points of $C_q$ located between the points $a$ and $b$, relatively to the direct ordering of $C_q$.

\vskip1mm

1. As above, we begin by the restriction to a single curve $C_q$ contained in $S_{N/2,N/3}$ and we will prove that $C_q$ contains an $(N,\eps)$--separated   subset with $[q/2]$ elements.

Fix a vertical segment $\{\th=\th_0\}$ in $\jA$, let $a_q$ be its intersection point with the curve $C_q$, for $q\in [q^*,+\infty]$,  and assume that
$a_\infty$ is not a singular point of $V$, so $\ph(a_\infty)\neq a_\infty$ and $\ph\inv(a_\infty)\neq a_\infty$. Note that for each $q\geq 3$, due to
the torsion condition,
the projection on $C_\infty$ of the interval $[\ph(a_q),\ph^{[q/2]}(a_q)]$ is contained in $[\ph(a_\infty),\ph\inv(a_\infty)]$. Assume that
$$
\eps<\Min[d(a_\infty, \ph(a_\infty)), d(a_\infty, \ph\inv(a_\infty))].
$$

Let $q$ be fixed, and for $k\geq0$ set $a^{(k)}=\ph^{-k}(a_q)$.
Remark that for
$0\leq k \leq [q/2]$, $0\leq k'\leq [q/2]$ and $k< k'$, the pair $(a^{(k)}, a^{(k')})$ is $(N,\eps)$--separated. Indeed:
$$
d_N\big(\ph^{k'}(a^{(k)}),\ph^{k'}(a_{k'})\big)\geq d(\ph^{k'-k}(a_q),a_q)> \eps,
$$
since $\ph^{k'-k}(a_q)\in[\ph(a_q),\ph^{[q/2]}(a_q)]$.
The set $\{a^{(k)}\mid 1\leq k\leq [q/2]\}$ is therefore $(N,\eps)$--separated, which proves our statement.

\vskip1mm

2. We will now prove that if the periods $q$ and $q'$ are separated enough and chosen  in the interval $[N/3,N/2]$,  with $N\geq 18$,
than any pair of points $\{a,a'\}$ in $C_q$ and $C_{q'}$ respectively, is $(N,\eps)$--separated.

Let us  fix
$$
\eps< \Min \Big(d(\ph\inv(a_{\infty}),a_{\infty}),d(\ph(a_{\infty}),\ph^2(a_{\infty}))\Big).
$$
On the curves $C_q$ and $C_{q'}$, with $q\geq q'$, we introduce the domains
$$
I_q=[a_q,\ph(a_q)[\ \subset C_q,\qquad J_{q'}=[\ph\inv(a_{q'}),\ph^2(a_{q'})[\ \subset C_{q'}.
$$
Thanks to the torsion condition,  the distance between $I_q$ and the complement  $C_{q'}\setm J_{q'}$ is therefore
larger than $\eps$, for each pair $(q,q')$ in $[q^*,+\infty[$.

Now assume that $q$ and $q'$ are contained in the interval $[N/3,N/2]$ and satisfy $q-q'\geq4$.
Consider two points $a\in C_q$ and $a'\in C_{q'}$. Clearly, there exists a unique positive integer $n_0\in \{0,\ldots,q-1\}$ such that
$\ph^{n_0}(a)\in I_q$.

\vskip1mm
\noindent$\bullet$
If $\ph^{n_0}(a')\in C_{q'}\setm J_{q'}$, than $d_{n_0}(a,a')>\eps$ and the pair $\{a,a'\}$ is $(N,\eps)$--separated since $N\geq q\geq n_0$.

\vskip1mm
\noindent$\bullet$
If $\ph^{n_0}(a')\in C_{q'}$, then remark that $\ph^{q+n_0}(a)\in I_q$, by $q$--periodicity.
Moreover, $\ph^{q+n_0}(a')\notin J_{q'}$. To see this,  note that
$$
\ph^{q+n_0}(a')=\ph^{q} (\ph^{n_0}(a'))=\ph^{q-q'} (\ph^{n_0}(a')),
$$
with $q-q'\geq 4$, $q\leq 3/2 q'$, $q'\geq 6$, so clearly $\ph^{q-q'}(J_{q'})\cap J_{q'}=\emptyset$.
Therefore
$$d(\ph^{q+n_0}(a),\ph^{q+n_0}(a'))>\eps$$
and,
 as $q+n_0\leq 2q\leq N$, this proves that $d_N(a,a')>\eps$.

\vskip2mm

We only have now to gather the previous constructions. In the  interval $[N/3,N/2]$, there exist at least
$[N/24]$ distinct integers elements  $(q_i)$ with $q_j-q_i\geq 4$ if $i\neq j$. On each curve $C_{q_i}$, one can find an $(\eps, N)$--separated
subset with $[q/2]\geq [N/6]$ elements, and the union of all these subsets is still $(N,\eps)$--separated.
Therefore the strip limited by the curves $C_{N/3}$ and $C_{N/2}$ contains an $(N,\eps)$--separated subset
with more than $N^2/150$ elements, for $N$ large enough, which proves that $\L(\ph)\geq 2$.


\subsection{Hamiltonian systems on surfaces: proof of theorem \ref{thm:indham}}
We consider a Morse Hamiltonian function $H$ on the compact symplectic surface $(\S,\Om)$, which is non critical and constant on each
component of $\d\S$. For $t\in\R$ we denote by $\phi_t$ the time $t$ diffeomorphism generated by the Hamiltonian vector field $V^H$
defined by $H$, and we write $\phi$ in place of $\phi_1$.

\parag {\bf Sketch of proof.}  We want to compute the complexity index $\L(\phi)$. To this aim, we will
exhibit a suitable covering of $\S$ by domains of the form $CC(H\inv([\al,\be]))$ (where $CC(A)$ stands for  a connected component of the subset $A$)
and compute the index of the restriction of the system
to each of these domains. We will of course choose intervals $[\al,\be]$ such that $]\al,\be[$ contains no critical value, with
either $\al$ or $\beta$ a critical value.

Critical points of index $0$ or $2$ for $H$ are isolated, while critical points of index $1$ belong to {\em polycycles} (recall that a polycycle
in this setting is an embedded quadrivalent graph
whose vertices  are critical points of index $1$ of $H$, and whose edges are heteroclinic or homoclinic orbits of $V^H$).
We refer to \cite{DMT} for basic facts on polycycles and non degenerate Hamiltonian foliations on surfaces.

\vskip1mm

We are thus led to examine three cases for the domains $CC(H\inv([\al,\be])$.

\vskip2mm

$\bullet$ {\bf The regular case.} The case when $[\al,\beta]$ contains no critical value has already been encountered.
Indeed, it is easy to define action-angle coordinates in $\jC=CC(H\inv([\al,\be]))$. The set $\jC$ is diffeomorphic to an annulus and foliated
by invariant circles on which $\phi$ is conjugate  to a rotation, whose angle may depend on the circle.
The index $\L(\phi,\jC)$  is therefore either $0$ if the angle is independent of the circle (and $\phi$ is then
globally conjugate to a  rotation) or $1$ when the angle depends non trivially on the circle at some place,
as proved in \ref{prop:actang}.

\vskip2mm

$\bullet$  {\bf The neighborhood of points of index $0$ or $2$.}
Let $z$ be a critical point of index $0$ of $H$, let $\al=H(z)$ and fix $\be>\al$ such that $]\al,\be]$ contains no critical value.
Set $\jC=CC(H\inv([\al,\be]),z)$ (where $CC(A,z)$ stands for the connected component of $A$ containing $z$),
so $\jC$ is diffeomorphic to a disk.  As above, the set $\jC\setminus \{z\}$ is
foliated by invariant circles on which $\phi$ is conjugate
to a rotation whose angle may depend on the circle.

Fix $\eps>0$. There exists $\de>0$ such that the diameter of $\jC'=CC(H\inv([\al,\al+\de]),z)$ is less than $\eps$. The complement
$\jC''=\jC\setminus\jC'$ is a regular annulus on which action-angle coordinates can be defined. Therefore proposition
\ref{prop:actang} proves the existence of a constant $c_\eps>0$ such that, for $N\geq 1$,
$$
G_N(\phi,\jC'',\eps)\leq c_\eps N.
$$
As a consequence, since $\jC'$ is invariant under the Hamiltonian flow:
$$
G_N(\phi,\jC,\eps)\leq c_\eps N+1,
$$
which proves that $\L(\phi,\jC)\leq 1$. Conversely, one readily sees that $\L(\phi,\jC)=0$ if $\phi$ is conjugate to a rotation
on $\jC$, and that  $\L(\phi,\jC)=1$ if it is not the case.
The study of neighborhoods of points of index 2 is completely analogous.

\vskip2mm

$\bullet$  {\bf The neighborhood of polycycles.} The remaining problem is therefore to investigate the behaviour of the system in the
neighborhood of polycycles. To this aim, we introduce the following definition.

\begin{Def}
Let $\P$ be a polycycle for the system $(\S,\phi)$, let $\al$ be the value of $H$ on $\P$ and fix $\be>\al$ such that
the interval $]\al,\be]$ contains no critical value of $H$. Fix a connected component $\jC$ of $H\inv(]\al,\be])$ such that $\P\cap\ov\jC$
is non empty. A {\em dynamical desingularization} of the system $(\jC,\phi)$ is a tame $p$--model with torsion $(\jA,\ph)$, together with a
homeomorphism $\Psi : \T\times\,]0,1]\mapsto \jC$ such that
$$
\Psi\circ \ph=\phi\circ\Psi
$$
with $\Psi$ and $\Psi\inv$ uniformly continuous.
\end{Def}

Since $\Psi$ is bi-uniformly continuous,
$$
\L(\phi,\jC)=\L(\ph,\T\times\,]0,1]).
$$

We add the condition $\P\cap\ov\jC\neq\emptyset$ to avoid trivial cases (regular components), which may only happen when $\P\neq H\inv(\{\al\})$.
Note also that
the integer $p$ is connected with the number of critical points in $\P\cap\ov\jC$, but is generally not equal to it as we will
see in the next section, where we will prove the following proposition.

\begin{prop}\label{prop:desing}  With the same notation,
there exists $\rho>0$ such that for each connected component $\jC$ of $H\inv(]\al,\al\pm\rho])$
such that $\P\cap\ov\jC$ is nonempty, the system $(\jC,\phi)$ admits a dynamical desingularization.
\end{prop}

Now fix a component $\jC$ as in the proposition above. Since $(\jC,\phi)$ admits a dynamical desingularization which leaves the complexity index invariant,
$\L(\phi,\jC)=2$, by proposition \ref{prop:pmodel}. Moreover, there exists a finite covering of
$$
H\inv([\al-\de,\al+\de])\setminus \P
$$
by such domains $\jC$. As a consequence
$$
\L(\phi,H\inv([\al-\de,\al+\de])\setminus \P)=2.
$$
Finally, one easily proves using proposition \ref{prop:indseg}
that $\L(\phi,\P)=1$. As a consequence
$$
\L(\phi,H\inv([\al-\de,\al+\de]))=2.
$$

Gathering the previous arguments, one sees that there exists a covering of $\S$ by invariant domains $\jD$ such that the
index $\L(\phi,\jD)$ belongs to $\{0,1,2\}$, and that there exists a domain $\jD$ such that $\L(\phi,\jD)=2$ if and only if there
exists a critical point of $H$ with index $1$. This concludes the proof of theorem \ref{thm:indham} if one admits the result of
proposition \ref{prop:desing}.

\parag {\bf Proof of proposition \ref{prop:desing}.} We consider a polycycle $\P$ for the system $(\S,\phi)$ with $H(\P)=\{0\}$. Moreover,
to avoid trivial cases, we assume that $\P=H\inv(\{0\})$ (the general case immediately follows). We fix $\rho_0>0$ such that $]0,\rho_0]$ contain no critical value
of $H$ and we consider a connected component $\jC(\rho_0)$  of  $H\inv(]0,\rho_0])$ (so $\P\cap\ov{\jC(\rho_0)}\neq\emptyset$).
We denote by $Z$  the set of critical points of $H$ which are contained in $\ov \jC(\rho_0)$.  In the following, for $0<\rho<\rho_0$,
we will denote by $\jC(\rho)$ the connected component of $H\inv(]0,\rho])$ which is contained in $\jC(\rho_0)$.

\begin{figure}[h]
\begin{center}
\begin{pspicture}(4cm,6.5cm)
\rput(-1.5,3){
\psccurve(0,0)(1.5,2)(3,0)(4.5,-1)(6,0)(7.5,1)(9,0)(7.5,-1)(6,0)(4.5,2)(3,0)(0,-1)(-1,3)(2,2)(-2,0)
\psarc{->}(0,1.8){1}{70}{140}
\pscircle[fillstyle=solid,fillcolor=black](1.96,1.8){.1}
\rput(2.7,1.9){$z_1=z_8$}
\pscircle[fillstyle=solid,fillcolor=black](-1.26,0.745){.1}
\rput(-1.6,0.9){$z_2$}
\pscircle[fillstyle=solid,fillcolor=black](-.77,-.255){.1}
\rput(-.8,-.55){$z_3$}
\pscircle[fillstyle=solid,fillcolor=black](3,0){.1}
\rput(3,-.4){$z_4$}
\rput(3,.4){$z_7$}
\pscircle[fillstyle=solid,fillcolor=black](6,0){.1}
\rput(6,-.4){$z_5$}
\rput(6,.4){$z_6$}
}

\end{pspicture}
\end{center}
\vskip-2cm
\caption{Polycycle and labelling of the vertices for the outer component}\label{polycycle}
\end{figure}
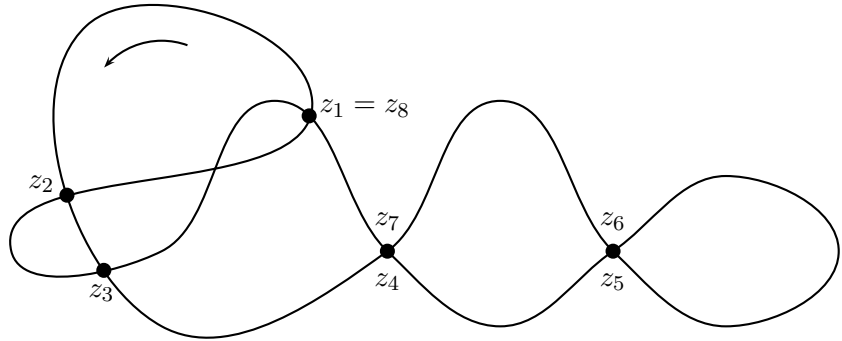

We consider the value of $H$ on $\jC(\rho_0)$ as a global coordinate, that we will denote by $r$, so $r\in\,]0,\rho_0]$.
By classical Morse theory, for $0<r<\rho_0$, the subset $C_r=H\inv(\{r\})\cap \jC(\rho_0)$ is non singular, compact and connected,
so it is a closed orbit of the Hamiltonian  flow. It is therefore oriented according to the vector field $V^H$.
The boundary component $C_0:=H\inv(\{0\})\cap\ov{\jC(\rho_0)}$  is a subpolycycle of $\P$, whose set of vertices we denote by
$Z$.  Note that $Z$
is in general strictly contained in the set of all critical points of $H$ contained in $P$.
When $r\to 0$, $C_r$ accumulates on $C_0$, which allows one to endow $Z$ with a cyclic  labelling
$z_1,\ldots,z_p,z_{p+1}=z_1$ such that  the points $z_j$ and $z_{j+1}$ are linked by a heteroclinic or homoclinic orbit of $\phi$
(see figure~\ref{polycycle}, where  $\jC(\rho_0)$ is assumed to be located inside the ``outer'' component of the complement of $\P$).
Note that the labelling need not be injective.

The conjugacy homeomorphism $\Psi$ we are seaching for will be obtained by gluing together local conjugacies locally defined
 in the neighborhood of the critical points
$z_k$ and along the homoclinic or heteroclinic orbits joining them. The shape of these domains is depicted in figure \ref{fig:conjug}.
In the following, we consider a given
$k$ and we perform the construction of the conjugacies in the neighborhoods of the point $z_k$ (domains $O_k$) and along the orbit
$\Gamma_k$ joining $z_k$ and $z_{k+1}$ (domain $R_k$).

The neighborhood $O_k$ will be limited by suitable entrance and exit sections $\Sigma_k^+$ and $\Sigma_k^-$ (according to the orientation
of the Hamiltonian vector field), a level $H\inv(\{\rho\})$ and the stable and unstable manifolds of $z_k$ (see figure \ref{fig:conjug}). The
domain $R_k$ is the ``intermediate zone'' in $\jC(\rho)$ located between $O_k$ and $O_{k+1}$.
The local conjugacies we will construct will be of the form  $F_k: \O_k\to O_k$ and $G_k:\jR_k\to R_k$, where
\begin{equation}\label{eq:defal1}
\O_k=\{(\th,r)\in\jA\mid \th\in[\frac{k}{p}-\al,\frac{k}{p}+\al]\}.
\end{equation}
and
\begin{equation}\label{eq:defal2}
\jR_k=\{(\th,r)\in\jA\mid \th\in[\frac{k}{p}+\al,\frac{k+1}{p}-\al]\}
\end{equation}
where $\al>0$ is small enough, for instance $\al=1/(8p)$.

\begin{figure}[h]
\begin{center}
\begin{pspicture}(4cm,5.8cm)
\rput(-1.5,3){
\psline[linewidth=0.3mm](0,0)(1,.8)
\psline[linewidth=0.3mm](0,0)(-1,.8)
\psline [linewidth=0.3mm]{->}(0,0)(.5,.4)
\psline [linewidth=0.3mm]{->}(-1,.8)(-.5,.4)
\psline[linewidth=0.3mm](7,0)(8,.8)
\psline[linewidth=0.3mm](7,0)(6,.8)
\psline [linewidth=0.3mm]{->}(7,0)(7.5,.4)
\psline [linewidth=0.3mm]{->}(6,.8)(6.5,.4)
\psline[linewidth=0.3mm](1,.8)(.8,1.05)
\psline[linewidth=0.3mm](8,.8)(7.8,1.05)
\psline[linewidth=0.3mm](6,.8)(6.4,.95)
\psline[linewidth=0.3mm](-1,.8)(-.6,.95)
\pscircle[fillstyle=solid,fillcolor=black](0,0){.05}
\pscircle[fillstyle=solid,fillcolor=black](7,0){.05}
\parametricplot[linewidth=0.3mm] {1}{6}{t  t  -1 add 36 mul sin 1.2 mul .8 add}
\parametricplot[linewidth=0.3mm] {1}{6}{t  t  -1 add 36 mul sin 1.15 mul 1.135 add}
\parametricplot[linewidth=0.3mm] {-1}{1}{t  t t mul  1 add sqrt .8 mul}
\rput(7,0){
\parametricplot[linewidth=0.3mm] {-1}{1}{t  t t mul  1 add sqrt .8 mul}
}
\rput(0,-.3){$z_k$}
\rput(7,-.3){$z_{k+1}$}
\rput(0,1.2){$O_k$}
\psline[linewidth=.2mm]{->}(0,1)(0,.4)
\rput(7,1.2){$O_{k+1}$}
\psline[linewidth=.2mm]{->}(7,1)(7,.4)
\rput(3.5,1.6){$\Ga_k$}
\psline[linewidth=.2mm]{->}(3.7,1.5)(5,1.5)
\rput(3.5,2.8){$R_k$}
\psline[linewidth=.2mm]{->}(3.5,2.6)(3.5,2.1)
\rput(-.7,2){$\Sigma^+_k$}
\psline[linewidth=.2mm]{->}(-.8,1.8)(-.8,.9)
\rput(.97,2){$\Sigma^-_k$}
\psline[linewidth=.2mm]{->}(.9,1.8)(.9,.9)
\rput(6,2.5){$H\inv(\{\rho\})$}
\psline[linewidth=.2mm]{->}(6,2.3)(6,1.15)
}
\end{pspicture}
\end{center}
\vskip-3cm
\caption{The domains of the conjugacies}\label{fig:conjug}
\end{figure}
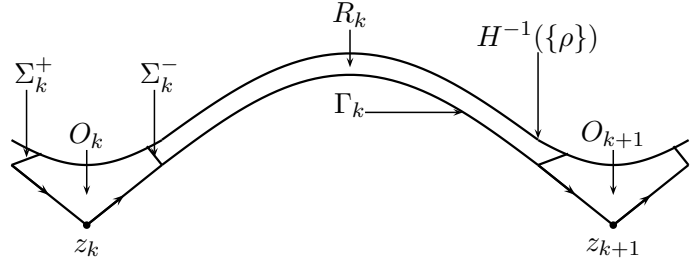

\vskip3mm

Our main task will be to choose the various objects in order for the (global) conjugate system on the annulus $\jA$ to satisfy the
tameness and torsion conditions. We will adopt the following strategy:
\vskip1mm
$\bullet$ construction of a conjugacy to a suitable desingularized system in the neighborhood of $z_k$
and definition of the exit sections $\Sigma_k^-$ for all $k$,
\vskip1mm
$\bullet$ given a suitable function $\sig_k:\Sig_k^-\to\R$, choice of the entrance section $\Sigma_{k+1}^+$ in such
a way that the transition time of the flow between $\Sigma_k^-$ and $\Sigma_{k+1}^+$ coincides with $\sigma_k$,
\vskip1mm
$\bullet$ computation of the transition time $\tau_k$ between the sections $\Sig_k^+$ and $\Sig_k^-$,
\vskip1mm
$\bullet$ construction of the conjugacy  $F_k=\O_k\to O_{k}$ and local torsion condition in $\O_k$,
\vskip1mm
$\bullet$ construction of the conjugacy in $G_k:\jR_k\to R_k$ and local torsion condition in $\jR_k$,
\vskip1mm
$\bullet$ gluing of the conjugacies for all $k$ and verification of torsion and tameness properties for the conjugate system.

\vskip3mm

{\bf 1. Desingularization in the neighborhood of the critical points and choice of the exit sections.}
The first result is a ``straightening lemma''  for orbits in sectors such as $O_k$  in figure \ref{fig:conjug}.
Very similar results have been proved by Eliasson \cite{E84} in the more difficult case of multidimensional systems
but here we will deduce it from the well-known following symplectic Morse lemma (\cite{CV}) in the plane.

\begin{lemma}
Let $z$ be a critical point of $H$ with index $1$, with $H(z)=0$. Then there exists a local symplectic diffeomorphism $\eta :(\R^2,0)\to(\S,z)$
such that $H\circ \eta (u,v)=h(v^2-u^2)$, where $h:(\R,0)\to(\R,0)$ is an increasing $C^\infty$ local diffeomorphism.
\end{lemma}

The function $h$ is not unique in general, but it will cause no trouble in the following.
In the following, for $\de>0$ and $\rho>0$,  we denote by  $D(\de,\rho)$ the rectangle $[-\de,\de]\times [0,\rho]\subset \R^2$,
while $B(\de)$ stands for the radius $\de$ ball  in $\R^2$ centered at the origin, for the Sup norm. Let us now state our first result.

\begin{lemma} Let $z$ be a critical point of $H$ contained in $Z$.
Then, for $\rho^*>0$ and $\de^*>0$ small enough, for all $\rho\in\,]0,\rho^*[$ and $\de\in\,]0,\de^*[$,
there exists a local homeomorphism
$\xi: (D(\de,\rho),0) \to (\ov\jC(\rho),z)$, smooth on $D(\de,\rho)\setminus\{0\}$,  such that $H\circ\xi(u,r)=r$ and
\begin{equation}\label{eq:normform}
\xi^*V^H(u,r) = \ell(r) \sqrt{u^2+\mu(r)}\,\Dron{}{u},\qquad \forall (u,r)\in D(\de,\rho)\setm\{0\},
\end{equation}
where $\ell$ and $\mu$ are smooth functions on $[0,\rho]$, with $\mu(0)=0$, $\ell > 0$, and $\mu$ increasing.
\end{lemma}

\begin{proof}
By the symplectic Morse lemma,  there exists $\de^*>0$ and a smooth local symplectic diffeomorphism
$\eta:(B(\de^*),0)\to(\S, z)$ such that and $H\circ \eta (u,v)=h(v^2-u^2)$, where $h:(\R,0)\to(\R,0)$ is a smooth local
increasing diffeomorphism. One easily checks that for $\de^*$ and $\rho^*$ small enough,
 in the coordinates $(u,v)$, the domain
$\eta\inv (\jC(\rho^*))$ has the form
$$
\abs{u}\leq\de^*,\quad \abs{v}\leq\de^*, \quad 0 <v^2-u^2 \leq h\inv(\rho^*).
$$

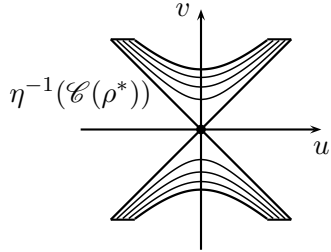
\begin{figure}[h]
\begin{center}
\begin{pspicture}(4cm,4cm)
\psset{unit=.8cm}
\rput(3,3){
\psline{->}(-2,0)(2,0)
\psline{->}(0,-2)(0,2)
\psline(1.1,1.5)(1.5,1.5)
\psline(-1.1,1.5)(-1.5,1.5)
\psline(1.1,-1.5)(1.5,-1.5)
\psline(-1.1,-1.5)(-1.5,-1.5)
\psline[linewidth=0.3mm](-1.5,-1.5)(1.5,1.5)
\psline[linewidth=0.3mm](1.5,-1.5)(-1.5,1.5)
\pscircle[fillstyle=solid,fillcolor=black](0,0){.08}
\parametricplot[linewidth=0.2mm] {-1.4}{1.4}{t  t t mul  .25 add sqrt}
\parametricplot[linewidth=0.2mm] {-1.3}{1.3}{t  t t mul  .5 add sqrt}
\parametricplot[linewidth=0.2mm] {-1.2}{1.2}{t  t t mul  .75 add sqrt}
\parametricplot[linewidth=0.3mm] {-1.1}{1.1}{t  t t mul  1 add sqrt}
\parametricplot[linewidth=0.2mm] {-1.4}{1.4}{t  t t mul  .25 add sqrt -1 mul}
\parametricplot[linewidth=0.2mm] {-1.3}{1.3}{t  t t mul  .5 add sqrt -1 mul}
\parametricplot[linewidth=0.2mm] {-1.2}{1.2}{t  t t mul  .75 add sqrt -1 mul}
\parametricplot[linewidth=0.3mm] {-1.1}{1.1}{t  t t mul  1 add sqrt -1 mul}
\rput(2,-.3){$u$}
\rput(-.3,2){$v$}
\rput(-2,.6){$\eta\inv(\jC(\rho^*))$}
}
\end{pspicture}
\end{center}
\vskip-1.5cm
\caption{The $(u,v)$-coordinates and the zone $\eta\inv(\jC(\rho^*))$.}
\end{figure}

To desingularize the situation we select only the part of $\eta(B(\rho^*))\cap\jC(\rho^*)$ corresponding to the orbits of the vector field which
have positive $\dot u$ in the $(u,v)$--coordinates, that is the zone of equation $v\geq 0$ in the same coordinates.
We denote by $\jD(\de^*,\rho^*)\subset \jC(\rho^*)$ this domain, whose equation in the
$(u,v)$ coordinates reads
$$
\abs{u}\leq\de^*,\quad 0\leq v\leq\de^*, \quad 0< v^2-u^2 \leq \rho^*.
$$

\begin{figure}[h]
\begin{center}
\begin{pspicture}(4cm,4cm)
\psset{unit=.8cm}
\rput(-2,3){
\psline{->}(-2,0)(2,0)
\psline{->}(0,-.6)(0,2)
\psline(1.1,1.5)(1.5,1.5)
\psline(-1.1,1.5)(-1.5,1.5)
\psline[linewidth=0.3mm](0,0)(1.5,1.5)
\psline[linewidth=0.3mm](0,0)(-1.5,1.5)
\pscircle[fillstyle=solid,fillcolor=black](0,0){.08}
\parametricplot[linewidth=0.2mm] {-1.4}{1.4}{t  t t mul  .25 add sqrt}
\parametricplot[linewidth=0.2mm] {-1.3}{1.3}{t  t t mul  .5 add sqrt}
\parametricplot[linewidth=0.2mm] {-1.2}{1.2}{t  t t mul  .75 add sqrt}
\parametricplot[linewidth=0.3mm] {-1.1}{1.1}{t  t t mul  1 add sqrt}
\rput(2,-.3){$u$}
\rput(-.3,2){$v$}
\rput(-2.5,.6){$\eta\inv(\jD(\de^*,\rho^*))$}
\psline{->}(3.7,.7)(5.7,.7)
\rput(4.7,1.2){$\eta$}
}
\rput(7,3){
\psline(1.1,1.5)(1.5,1.5)
\psline(-1.1,1.5)(-1.5,1.5)
\psline[linewidth=0.3mm](0,0)(1.5,1.5)
\psline[linewidth=0.3mm](0,0)(-1.5,1.5)
\pscircle[fillstyle=solid,fillcolor=black](0,0){.08}
\parametricplot[linewidth=0.2mm] {-1.4}{1.4}{t  t t mul  .25 add sqrt}
\parametricplot[linewidth=0.2mm] {-1.3}{1.3}{t  t t mul  .5 add sqrt}
\parametricplot[linewidth=0.2mm] {-1.2}{1.2}{t  t t mul  .75 add sqrt}
\parametricplot[linewidth=0.3mm] {-1.1}{1.1}{t  t t mul  1 add sqrt}
\rput(2.5,.6){$\jD(\de^*,\rho^*)$}
\rput(0,-.3){$z$}
}
\end{pspicture}
\end{center}
\vskip-2.5cm
\caption{The domains $\eta\inv(\jD(\de^*,\rho^*))$ and $\jD(\de^*,\rho^*)\subset \S$.}
\end{figure}
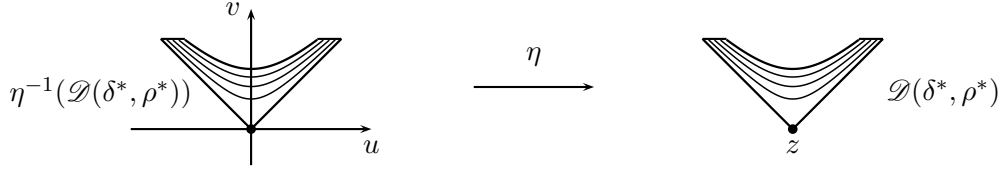

In the domain $\ov{\jD(\de^*,\rho^*)}$ one can then choose the pair
$$
(u,r:=h(v^2-u^2))
$$
as continuous local coordinates, this way the coordinate $r$ coincides with the value of the Hamiltonian $H$.
These (non symplectic) new coordinates $(u,r)$ define a homeomorphism $\xi$ from the rectangle
$D(\de^*,\rho^*)$ onto  $\ov{\jD(\de^*,\rho^*)}$, which sends $0$ on $z$.  Note that $\xi$ is smooth on $D(\de^*,\rho^*)\setm\{0\}$.

Using the symplectic character of $\eta$, one readily checks that in the coordinates $(u,r)$ the Hamiltonian vector field takes the required form
$$
\dot u = \ell(r) \sqrt{u^2+\mu(r)},\qquad \dot r=0,
$$
with $\mu=h\inv$ and $\ell=h'\circ \mu$. Note that $\mu$ is increasing, since $h$ is, and $\ell>0$.

It is now easy to check that the construction still remain valid on the subdomains $D(\de,\rho)$ for $\de\in\,]0,\de^*]$ and
$\rho\in\,]0,\rho^*]$, by considering the restrictions of the previous maps.
\end{proof}

In the following we will assume that the constants $\de^*$ and $\rho^*$ are uniformly chosen, so that the previous construction
and result are valid for each critical point $z_k\in Z$.
For each $k$, we therefore have at our disposal a ``sectorial neighborhood''
$\ov{\jD_k(\de^*,\rho^*)}$ equipped with a coordinate system $(u,r)\in D(\de^*,\rho^*)$ and a homeomorphism
$\xi_k:D(\de^*,\rho^*)\to \ov{\jD_k(\de^*,\rho^*)}$. We assume that $\de^*$ is so small that two distinct domains
$\ov{\jD_k(\de^*,\rho^*)}$ and $\ov{\jD_{k'}(\de^*,\rho^*)}$ have empty intersection. This is the only constraint we will
have to impose on $\de^*$, while several new ones for $\rho^*$ will be introduced in the following.

\vskip2mm

{\bf Exit sections.} We define the exit section $\Sigma^-_k\subset\ov{\jD_k(\de^*,\rho^*)}\subset \S$ as the image by $\xi_k$ of
the vertical of  $D(\de^*,\rho^*)$ of equation
$$
u=\de^*/2.
$$
This is indeed a section of the Hamiltonian flow, since the vertical is a section in
the coordinates $(u,r)$ (see figure \ref{fig:section}). We denote by $a_k$ the point $\xi_k(\de^*/2,0)$, that is the intersection
of $\Sigma_k^-$ with the
unstable manifold of $z_k$.

\begin{figure}[h]
\begin{center}
\begin{pspicture}(4cm,4.5cm)
\psset{unit=.8cm}
\rput(-1.5,3){
\pscircle[fillstyle=solid,fillcolor=black](0,0){.08}
\psline{->}(-2,0)(2,0)
\psline{->}(0,-1)(0,2)
\psline[linewidth=0.2mm] (-1.5,.25)(1.5,.25)
\psline[linewidth=0.2mm] (-1.5,.5)(1.5,.5)
\psline[linewidth=0.2mm] (-1.5,.75)(1.5,.75)
\psline[linewidth=0.5mm] (.75,0)(.75,1)
\pspolygon[linewidth=.03](-1.5,0)(1.5,0)(1.5,1)(-1.5,1)
\rput(2,-.3){$u$}
\rput(-.3,2){$r$}
\rput(-1.8,1.5){$D(\de^*,\rho^*)$}
\psline{->}(2.5,.6)(4.7,.6)
\rput(3.55,1){$\xi_k$}
}
\rput(5.6,3){
\psline(1.1,1.5)(1.5,1.5)
\psline(-1.1,1.5)(-1.5,1.5)
\psline[linewidth=0.3mm](0,0)(1.5,1.5)
\psline[linewidth=0.3mm](0,0)(-1.5,1.5)
\pscircle[fillstyle=solid,fillcolor=black](0,0){.08}
\parametricplot[linewidth=0.2mm] {-1.4}{1.4}{t  t t mul  .25 add sqrt}
\parametricplot[linewidth=0.2mm] {-1.3}{1.3}{t  t t mul  .5 add sqrt}
\parametricplot[linewidth=0.2mm] {-1.2}{1.2}{t  t t mul  .75 add sqrt}
\parametricplot[linewidth=0.3mm] {-1.1}{1.1}{t  t t mul  1 add sqrt}
\psline[linewidth=0.1mm]{<-}(.6,.8)(2.2,.7)
\psline[linewidth=0.4mm](.75,.7)(.2,1.02)
\rput(3,1.6){$\jD_k(\de^*,\rho^*)$}
\rput(2.5,.7){$\Sigma_k^-$}
\rput(-.2,-.3){$z_k$}
\rput(.85,.4){$a_k$}
}
\end{pspicture}
\end{center}
\vskip-2cm
\caption{The $(u,r)$-coordinates and the exit section.}\label{fig:section}
\end{figure}
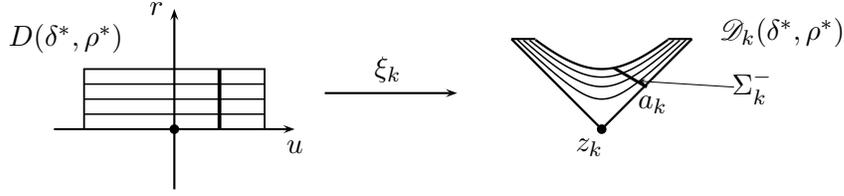

\vskip3mm
{\bf 2. Choice of the entrance sections.} We now choose a set of  functions $(\sig_k)_{1\leq k\leq p}:\Sigma_k^-\to\R$, with positive values,  which will be realized as the transition
times between the sections $\Sigma_k^-$ and $\Sigma_{k+1}^+$ for $1\leq k\leq p$ (as usual $p+1=1$). These functions naturally depend on the coordinate $r$ (which is a coordinate on any section),
so $\sigma_k:[0,\rho^*]\to\R^+$, and they have to be subjected to suitable conditions in order for the final torsion and tameness conditions to be fullfilled.

\vskip2mm

First note that there obviously exists a positive time $T$ such that  for each $k$,
$$
\phi_T(a_k)\in \jD_{k+1}(\de^*,\rho^*).
$$
{\bf Condition 1.}
{\em We assume that, for all $k$,  $\sig_k$ satisfies $\sig_k(0)\geq T$ and $\sig_k(r) >\frac{1}{p}-2\al$ for $r\in[0,\rho^*]$.}

\vskip2mm

The following condition will enable us to prove to the torsion property in each $R_k$.

\vskip2mm

\noindent {\bf Condition 2.} {\em We assume that, for each $k$,  $\sig_k$  is a decreasing function of $r$.}

\vskip2mm

The functions $\sig_k$ being so chosen, we can now define the entrance sections. Recall that we denote
by $\phi_t$ the time--$t$ diffeomorphism associated with the Hamiltonian flow.

\begin{lemma} \label{lem:timesect}
For each $k\in\{1,\ldots, p\}$, let  $a_k(r)$ be the point of coordinate  $r$  on $\Sig_k^-$, that is
$a_k(r)=\xi_k(\de^*/2,r)$.
Then there exists $\bar \rho<\rho^*$ such that the curve
$$
\Sig_{k+1}^-=\{\phi_{\sig_k(r)}(a_k(r))\mid r\in[0,\bar \rho]\}
$$
is a transverse section contained in the normal form domain $O_{k+1}$.  Moreover, obviously,
the transition time between $\Sig_k^-$ and $\Sig_{k+1}^+$, expressed as a function of $r\in[0,\bar\rho]$,
coincides with $\sig_k$.
\end{lemma}

\begin{proof}  First, since $a_k(0)=a_k$, notice  that $\phi_{\sig_k(0)}(a_k(0))\in O_{k+1}$ since $\sig_k(0)\geq T$ (condition~1)
and since $\phi_{\sig_k(0)}(a_k)$ on the stable manifold of  $z_{k+1}$. Now, due to the presence of the fixed point $z_{k+1}$,
 the time needed for the point $a_k(r)$, $r>0$, to reach the exit section $\Sigma_{k+1}^-$ tends to $+\infty$ when $r$ tends to $0$.
 So, by continuity of the flow,
there exists $\bar\rho_k$ such that
$$
\{\phi_{\sig_k(r)}(a_k(r))\mid r\in[0,\bar \rho_k]\}\subset O_{k+1}.
$$
It suffices now to choose $\bar\rho=\Min_{1\leq k\leq p}\bar\rho_k$.
\end{proof}

In the following we localize our constructions to the domain $\jC(\bar\rho)$, but we still denote by $\Sig_k^+$ the entrance
sections in this domain (which are the intersections of the previous ones with $\jC(\bar\rho)$).  We will always use this convention
in the following.

\vskip3mm

{\bf 3. Estimate of the transition time between $\Sig_k^+$ and $\Sig_k^-$.}
The following easy lemma provides us with the necessary estimate for the transition time between two {\em arbitrary} sections
inside the conjugacy neighborhood $O_k$.

\begin{lemma}\label{lem:sect} Consider the rectangle $D(\de,\rho)$ endowed with the vector field
$$
X=\ell(r) \sqrt{u^2+\mu(r)}\,\Dron{}{u},
$$
with $\ell$ and $\mu$ smooth, $\ell>0$, $\mu(0)=0$ and $\mu$ increasing.
Let $\De_0$ and $\De_1$ be the local transverse sections of equations $u=u_0$ and $u=u_1$
in $D(\de,\rho)$, with $u_0<0$ and $u_1>0$.
Then the transition time $\tau$ between  $\De_0$ and $\De_1$ reads
\begin{equation}\label{transtime}
\tau(r)=\Frac{1}{\ell(r)}\Big({\rm Argsh\,}\Frac{u_1}{\sqrt{\mu(r)}}-{\rm Argsh\,} \Frac{u_0}{\sqrt{\mu(r)}}\Big).
\end{equation}
\end{lemma}

\begin{proof}
Immediate computation.
\end{proof}

\vskip1mm

\begin{cor}
Let $\Sigma_0$ and $\Sigma_1$ be two smooth transverse sections for the Hamiltonian flow, contained in $\ov{\jD_k(\de^*,\bar\rho)}$,
such that $\Sigma_0$ intersects the stable manifold of $z_k$ and $\Sigma_1$ intersects the unstable manifold of $z_k$.
Then, if $\tau_k$ is the transition time from $\Sigma_0$ to $\Sigma_1$, expressed as a function of $r$, the derivative $\tau_k'(r)$ tends to
$-\infty$ when $r$ tends to $0$.
\end{cor}

\begin{proof} Let $u_0<0$ and $u_1>0$ be the $u$--coordinates of the intersections points of $\Sigma_0$ and $\Sigma_1$ with the invariant manifolds of $z_k$. Let $S_0=\xi_k(\Delta_0)$ and $S_1=\xi_k(\Delta_1)$, with $\Delta_i$ as in the lemma above, and let $\tau$ be the transition time
between $S_0$ and $S_1$.
Equation (\ref{transtime}) immediately shows that $\tau'(r)\to -\infty$ when $r\to 0$. Now the transition time $\tau_k$ is the sum of the
principal term $\tau$ and two complementray terms, transition times between the
sections $\Sigma_0,S_0$ and $S_1,\Sigma_1$. These terms are bounded in the $C^1$ topology,
which proves that  $\tau_k'(r)$ tends to
$-\infty$ when $r$ tends to $0$.
\end{proof}

As a consequence of the previous corollary, there exists $\ha\rho$ such that, for each $k$, the transition time $\tau_k$ between
the sections $\Sigma_k^+$ and $\Sigma_k^-$ is a decreasing function of $r$ on the interval $[0,\ha\rho]$. In the following, we localize
our constructions to the domain $\jC(\ha \rho)$.

\vskip3mm

{\bf 4. Conjugacy in the domain $O_{k}$.}
The domain $O_{k}$ is now well-defined : namely $O_k$ is limited by the sections $\Sigma_{k}^\pm$, the level curve $H\inv(\{\ha\rho\})$ and the stable and
unstable manifolds of the point $z_{k}$.  The conjugacy we are searching for in this domain is given by the following lemma.

\begin{lemma}\label{lem:conjo}
 There exists a continuous increasing function $\chi_k:[0,1]\to\R$, smooth on $]0,1]$ and such as $\chi_k(0)=0$, such that, if
$V_k$ is the vector field defined on $\O_k$ by
$$
V_k(\th,r)=\sqrt{(\th-k/p)^2+\chi_k(r)}\,\Dron{}{\th},
$$
there exists a homeomorphism $F_k: (\O_k,o_k)\to (O_k,z_k)$,  which is smooth on
$\O_k\setminus\{o_k\}$ and which satisfies on this domain:
\begin{equation}\label{finform}
F_k^*(V^H)=V_k.
\end{equation}
As a consequence, $F_k$ conjugates the local flows of $V_k$ and $V^H$ on the domains $\O_k$ and $O_k$.
Moreover, $F_k$ sends the left
boudary of $\O_k$ on the section $\Sigma_k^+$, the right boundary of $\O_k$ on the section $\Sigma_k^-$ and
the horizontal of equation $r=r_0$ in $\O_k$ on the curve $H\inv(\{\ha \rho\,r_0\})\cap O_k$.
\end{lemma}

\begin{proof} Recall that $\O_k=[k/p-\al,k/p+\al]\times [0,1]$, where $\al>0$ is given. We first perform a rescaling of the variable $r$ to let
it vary in $[0,1]$,  by setting $r'=r/\ha\rho$. This does not affect the form of the  vector field
$V=\xi_k^*(V^H)$ in the coordinates $(u,r)$.

We therefore start with a smooth vector field $V$ of the form (\ref{finform}), defined on a domain $D\subset D(\de^*,1)$,
limited by the section $\Sigma^-$ of equation $u=\de^*/2$ on the right, and by a given global transverse section $\Sigma^+$ on the left.
These two sections are obviously deduced from the previous data.
We denote by $\tau$ the transition time between these two sections. Due to the choice of $\ha\rho$,  $\tau$ is a decreasing function of $r\in\,]0,1]$, which tends to $+\infty$ when $r$ tends to $0$.

Forgetting about the innocuous term $k/p$ in the domain of the conjugacy, we now consider the rectangle $\O=[-\al,\al]\times[0,1]$ together with
a vector field of the form
$$
\bar V(\th,r)= \sqrt{\th^2+\chi(r)}\,\Dron{}{u},
$$
and we want to conjugate $\bar V$ with $V$.

Let $S^+$ and $S^-$ be the sections of equations $u=-\al$ and $u=+\al$ respectively.
We first note that one  can choose the  function $\chi$ so that the transition time $\bar \tau$ induced by $\bar V$ between the two sections
$S^+$ and $S^-$ coincides with $\tau$. Indeed, lemma \ref{lem:sect} shows that $\bar \tau$ has the form
$
\bar \tau(r)=2\,{\rm Argsh\,}\frac{\al}{\sqrt{\chi(r)}},
$
so one chooses
$$
\chi(r) = \Frac{\al}{\sh(\tau(r)/2)}
$$
for $r>0$ and $\chi(0)=0$.  The function $\chi$ is smooth on $]0,1]$ and
one could prove, using a more explicit form of $\tau$, that it admits a $C^1$ continuation to $[0,1]$ (but it is not necessary).
Since $\tau$ is decreasing, $\chi$ is increasing.

Now, to construct the conjugacy, let $\Phi$ and $\bar\Phi$ be the local flows of $V$ and $\bar V$. For a point $(\th,r)\in [-\al,\al]\times ]0,1]$,
let $t(r)\leq 0$ be the time such that $\bar\Phi(t(r),(\th,r))\in S^+$. We define a map $\zeta : [-\al,\al]\times ]0,1]\to  D$  by
$$
\ze(u,r)=\Phi (-t(r), (\al,r)).
$$
Then $\zeta$ is clearly a diffeomorphism which sends $[-\al,\al]\times ]0,1]$ on its image in $D$, which exchanges the vector fields
$V$ and $\bar V$. Moreover, it has a well-defined continuation $\bar\ze$ to $[-\al,\al]\times [0,1]$, which is
a homeomorphism between $\O$ and $D$ which conjugates the flows $\bar \Phi$ and $\Phi$. One readily sees that
$\bar\zeta$ sends $S^+$ on $\Sig^+$ and $S^-$ on $\Sig^-$.

The remainder of the proof is trivial (one only has to translate the variable $\th$ by $k/p$ and to compose by $\xi_k$ to obtain $F_k$).
\end{proof}

\vskip 2mm

{\bf 5. Conjugacy in the domain $R_k$.} As depicted in figure \ref{fig:conjug}, the domain
$R_k$ is the flow-box  zone of the surface $\S$ limited by the orbit $\Ga_k$ joining $z_k$ to $z_{k+1}$, the level $H\inv(\{\ha\rho\})$ and the sections $\Sig^-_k$ and  $\Sig^+_{k+1}$. Let us state the conjugacy result relative to this domain  (recall that $\sig_k$ is the transition
time between $\Sig^-_k$ and  $\Sig^+_{k+1}$).

\begin{lemma}\label{lem:conjr}
  Let $W_k$ be the vector field on $\jR_k$ defined by
$$
W_k(\th,r)= w_k(\th,r)\,\Dron{}{\th},
$$
with $w_k>0$ and smooth, satisfying the relation:
\begin{equation}\label{eq:relfond}
\int_{k/p+\al}^{(k+1)/p-\al} \Frac{d\th}{w_k(\th,r)}=\sig_k(r).
\end{equation}
Then there exists a diffeomorphism $G_k:\jR_k\to R_k$ which conjugates the vector fields $W_k$ and $V^H$. Moreover, $G_k$ sends
the left boundary of $\jR_k$ on $\Sigma_k^-$, the right boundary of $\jR_k$ on $\Sigma_{k+1}^+$ and the horizontal
of equation $r=r_0$ in $\jR_k$ on the curve $H\inv{(\ha\rho\,r_0)}\cap R_k$.
\end{lemma}

\begin{proof} One uses exactly the same construction as in the previous lemma, which is here even simpler due to the absence of
fixed point. This is why the resulting map $G_k$ is a (smooth) diffeomorphism on its domain.
\end{proof}

\vskip2mm

{\bf 6. Smoothing, global gluing and the torsion and tameness properties.} We can now construct a $p$--system on $\jA$
by gluing the vector fields $V_k$ and $W_k$ of lemmas \ref{lem:conjo} and \ref{lem:conjr}.
It is easy to prove the torsion and tameness properties for the glued system.
The drawback of this system is that it is discontinuous at the boudary of the zones $\O_k$. We will see that it is  very simple to
modify it a little bit to obtain a suitable smooth vector field on $\jA$.

\vskip2mm

$\bullet$ We begin with the tameness property, for which we will have to choose the functions $w_k$ more precisely. We want
the fundamental domain of definition \ref{def:tameness} to be contained in $R_1$. We  introduce the subinterval $I=[1/2-a/4,1/2+a/4]$ of
length $a/2$ centered at the middle point $\th_0=1/2$ of $R_1$. We will choose the functions $w_k$ in such a way
that $\jK=I\times [0,1]$; more precisely, we want the length of $V(\th,r_0)$
for $(\th,r_0)\in\jK$ to be constant and larger that the maximal length of $V(\th,r_0)$ for $(\th,r_0)\notin\jK$.

\vskip2mm For $k\geq 2$, let us choose
$$
w_k(\th,r)=\Frac{a}{\sig_k(r)}, \qquad a=1/p-2\al
$$
so $w_k$ is independent of $\th$ and satisfies the relation (\ref{eq:relfond}). As for $w_1$, we choose
a smooth function on $R_1=[\al,1-\al]\times[0,1]$, constant and equal to
\begin{equation}\label{eq:maj}
\Max(a/2,\Max_{1\leq k\leq p}\sqrt{(\de^*)^2+\chi_k(\ha\rho)})
\end{equation}
over $I\times[0,1]$, and we choose the values
of $w_1(\th,r)$ for $\th\notin I$ in order
to satisfy the relation (\ref{eq:relfond}) for each fixed $r$, which is possible since $\sig_1(r)>1$. We moreover require
that
$$
w_1(\th,r)\leq w_1(\th,r')
$$
if $r\leq r'$ and $\th\in[\al,1-\al]$. Such a choice is obviously possible since the function $\sig_1$ is decreasing.

Now, since we have assumed that $\sig_k>2$, one sees that, when $\th\in I$
$$
w_1(\th,r)< \Max_{2\leq k\leq p} w_k.
$$
We also have to compare $w_1$ with the length of $V$ in the domains $O_k$, that is with the function
$\sqrt{(\th-k/p)^2+\chi_k(r)}$. Again, equation (\ref{eq:maj}) proves that the value on $w_1$ on $\jK$ is larger than
the maximum of the lengthes on the domains $O_k$.
Finally, one clearly sees that $I\times[0,1]\subset R_1$ is a fundamental domain for the flow of $V$.
\vskip2mm

Now, fix  two nearby points  $(\th,r)$ and $(\th',r)$ in $\jR_1$, on the same orbit, with lifts $(x,r)$ and $(x',r)$ to the universal covering
$\til\jA=\R\times[0,1]$, chosen such that $0<x'-x<1/2$.
We write as usual $\til V$ and $\til \ph$ for the lifted vector field and flow. Let $t_0$ be the unique real number such that
$(x',r)=\til\ph_{t_0}(x,r)$.  Then, setting
$(x(t),r)$ and $(x'(t),r)$ for $\til\ph_t(x,r)$ and $\til\ph_t(x',r)$, the separation function is defined by $E(t)=x'(t)-x(t)$.
Note that
$$
t_0=\int_{x(t)}^{x'(t)}\Frac{d u}{\til V(u)}= \int_{x}^{x'}\Frac{d u}{\til V(u)}.
$$
Now, since the lentgh of $V$ is maximal in the domain $R_1$, the time separation $E(t)$ of the
two points is  minimal when they both belong to $R_1$. This proves the tameness condition for the fundamental domain $\jK$.

\vskip2mm

$\bullet$ The torsion condition is now easy, indeed it suffices to chek that is is satisfied in each domain $O_k$ and $R_k$ by the
vector fields $V_k$ and $W_k$.  For $V_k$, the torsion condition is an immediate consequence
of the fact that $\chi_k$ is an increasing function, so that for $r'\geq r$ in $[0,1]$, the length of $V_k(\th,r')$ is larger than that
of $V_k(\th,r)$. As for $W_k$, it is even easier, since the length of $W_k(\th,r)$ is independent of $\th$ and equal to
$1/\sig_k(\ha \rho\, r)$, which is an increasing function of $r$ since $\sig_k$ has been assumed to be decreasing (condition 2).

\vskip2mm

$\bullet$ The smoothing process is now obvious. One checks that is it possible to modify the vector field $W_k$ in the neighborhood of the
entrance and exit boudaries of $\jR_k$, in such a way that the gluing with $V_k$ and $V_{k+1}$ is smooth and the equality (\ref{eq:relfond}) is satisfied.
One can moreover require that
$$
w_k(\th,r)< w_k(\th, r')
$$
for $0\leq r < r'\leq 1$ and $\th\in\T$. This way, the tameness and torsion properties are still valid for the modified smooth glued vector
field on $\jA$.

\vskip3mm

{\bf 7. Conjugacy.}
One can now construct a surjective continuous map $\Psi: \jA\to \ov{\jC(\ha\rho)}$ by gluing together the
homeomorphisms  $F_k$ and $G_k$ on the boundaries of their domains. It is clear that  $\Psi\circ \ph=\phi\circ\Psi$,
where $\ph$ is the time-one map of the $p$-model on $\jA$ and $\phi$ is the Hamiltonian time-one map. Moreover, $\Psi$ is
uniformly continuous by compactness. Now the restriction of $\Psi$ to  $\T\times\,]0,1]$ is a diffeomeomorphism, and it is
clear by construction that $\Psi\inv$ is uniformly continuous. Proposition \ref{prop:desing} is proved.


\section{The complexity index of some plane gradient models}\label{Sec:gradmod}
To conclude this paper, we briefly describe the computation of the compexity index $\jL$ in a case which is completely
different from the previous one : an example of a gradient system in the plane. We will not enter into the details, the
case of gradient systems  will be extensively studied in a subsequent paper.

\subsection{Gradient models on plane strips}
We first introduce a special strip  in the plane, on which the systems will be defined.
Given an integer $p\geq2$ and $\de>0$, let $\eta_p$ be a smooth function
on $[0,1/p]$ which satisfies the conditions:
$\eta_p(0)=0$, $\eta(7/(8p))=1$, $\eta'(x)>0$ for $x\in\,]0,1/p[$,  $\eta^{(k)}(7/(8p))=0$ for $k\geq 1$
and finally $\eta'$ constant on the interval $[0,1/(8p)]$.
We denote by $\bS_p$ the subset of the plane formed by the points $(x,y)$ such that:
$$
\begin{array}{llll}
&x\in[1/p,1-1/p], \quad &y\in[0,1]\\
\text{or}&x\in[0,1/p], \quad &0\leq y\leq \eta_p(x)\\
\text{or}&   x\in[1-1/p,1], \quad  &0\leq y\leq \eta_p(1-x).
\end{array}
$$
In the following, the integer $p$ will be fixed once and for all and we abbreviate $\bS_p$ in $\bS$.
We set $O_\al=(0,0)$, $O_\om=(1,0)$ and $O_k=(k/p,0)$ for $1\leq k\leq p-1$.

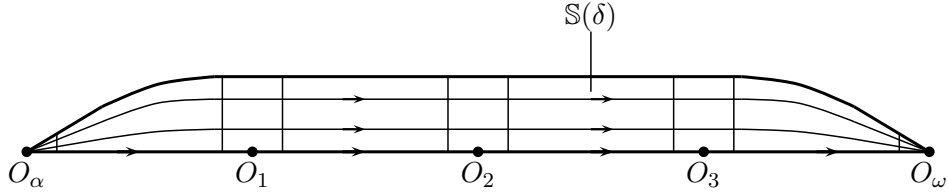
\begin{figure}[h]
\begin{center}
\begin{pspicture}(4cm,2.5cm)
\psset{xunit=1cm,yunit=1cm}
\rput(2,0){
\psline[linewidth=.4mm](-6,0)(6,0)
\psline[linewidth=.4mm](-3.5,1)(3.5,1)
\psline[linewidth=.4mm](-6,0)(-5,.6)
\pscurve[linewidth=.4mm](-5,.6)(-4.25,.9)(-3.5,1)
\psline[linewidth=.4mm](6,0)(5,.6)
\pscurve[linewidth=.4mm](5,.6)(4.25,.9)(3.5,1)
\psline[linewidth=.2mm](-3.5,.7)(3.5,.7)
\pscurve[linewidth=.2mm](5,.4)(4.25,.65)(3.5,.7)
\psline[linewidth=.2mm](5,.4)(6,0)
\pscurve[linewidth=.2mm](-5,.4)(-4.25,.65)(-3.5,.7)
\psline[linewidth=.2mm](-5,.4)(-6,0)
\psline[linewidth=.3mm]{->}(-1.8,.7)(-1.5,.7)
\psline[linewidth=.3mm]{->}(1.5,.7)(1.8,.7)
\psline[linewidth=.2mm](-3.5,.3)(3.5,.3)
\pscurve[linewidth=.2mm](5,.17)(4.25,.27)(3.5,.3)
\psline[linewidth=.2mm](5,.17)(6,0)
\pscurve[linewidth=.2mm](-5,.17)(-4.25,.27)(-3.5,.3)
\psline[linewidth=.2mm](-5,.17)(-6,0)
\psline[linewidth=.3mm]{->}(-1.8,.3)(-1.5,.3)
\psline[linewidth=.3mm]{->}(1.5,.3)(1.8,.3)
\psline[linewidth=.4mm]{->}(-1.8,0)(-1.5,0)
\psline[linewidth=.4mm]{->}(1.5,0)(1.8,0)
\psline[linewidth=.4mm]{->}(-4.8,0)(-4.5,0)
\psline[linewidth=.4mm]{->}(4.5,0)(4.8,0)
\pscircle[fillstyle=solid,fillcolor=black](-6,0){.07}
\pscircle[fillstyle=solid,fillcolor=black](-3,0){.07}
\pscircle[fillstyle=solid,fillcolor=black](0,0){.07}
\pscircle[fillstyle=solid,fillcolor=black](3,0){.07}
\pscircle[fillstyle=solid,fillcolor=black](6,0){.07}
\psline[linewidth=.2mm](-2.6,0)(-2.6,1)
\psline[linewidth=.2mm](-3.4,0)(-3.4,1)
\psline[linewidth=.2mm](-.4,0)(-.4,1)
\psline[linewidth=.2mm](.4,0)(.4,1)
\psline[linewidth=.2mm](2.6,0)(2.6,1)
\psline[linewidth=.2mm](3.4,0)(3.4,1)
\psline[linewidth=.2mm](-5.6,0)(-5.6,.25)
\psline[linewidth=.2mm](5.6,0)(5.6,.25)
\rput(-6,-.3){$O_\al$}
\rput(-3,-.3){$O_1$}
\rput(0,-.3){$O_2$}
\rput(3,-.3){$O_3$}
\rput(6,-.3){$O_\om$}
\rput(1.5,1.8){$\bS(\de)$}
\psline[linewidth=.2mm](1.5,0.8)(1.5,1.6)
}
\end{pspicture}
\end{center}
\vskip0cm
\end{figure}

\begin{Def}
We call {\em asymptotic $p$-model}, or simply a {\em $p$--model}, any $C^\infty$ vector field $V=(X,Y)$ defined on the strip $\bS$ which satisfies
the following conditions.
\vskip 1mm
{\rm (C1)}  $Y(x,y)=0$ if  $x\in[1/p,1-1/p]$ and $X(x,y)>0$ if $y>0$, so the orbits of $V$ are ``horizontal'' between
the segments $x=1/p$ and $x=1-1/p$, and they are oriented from left to right.
\vskip 1mm
{\rm (C2)}  The points $O_\al$, $O_\om$, $O_k$, $k\in\{1,\ldots, p-1\}$, are the only fixed points of $V$ and they are connected by heteroclinc orbits, that is: $X(k/p,0)=0$, $0\leq k\leq p$, and  $X(x,0)>0$ if $x\notin\{0,1/p,\ldots,1\}$.
\vskip 1mm
{\rm (C3)} For $1\leq k\leq p-1$, in the neighborhood $\O_k$ of $O_k$  defined by
$x\in[\frac{k-1/8}{p},\frac{k+1/8}{p}]$,  the $X$ component of the vector field $V$ reads
$$
X(x,y)=\ell_k\sqrt{y+(x-\frac{k}{p})^2},
$$
with $\ell_k>0$.
\vskip 1mm
\vskip 1mm
{\rm (C4)} In the neighborhood $\O_\al$ of $O_\al$  defined by
$x\in[0,1/(8p)]$,  the vector field $V$ reads
$$
X(x,y)=\ell_\al x,\quad Y(x,y)=\ell_{\al} y
$$
with $\ell_\al>0$, while in the neighborhood $\O_\om$ of $O_\om$  defined by
$x\in[1-1/(8p),1]$,  the vector field $V$ reads
$$
X(x,y)=-\ell_\om(x-1),\quad Y(x,y)=-\ell_\om y
$$
with $\ell_\om>0$
\vskip 1mm
\end{Def}

The subset of $\bS$ defined by $7/(8p)\leq x\leq 1-7/(8p)$ will be refer to as the {\em flat zone}. Note that a $p$--model has a well-defined
smooth gobal flow.
As in the Hamiltonian case, we have to add two technical conditions that we now define.

\begin{Def} {\bf (Torsion).} Consider an asymptotic $p$--model $V$ on the strip $\bS$ and denote by $\ph_t$ the time $t$ diffeomorphism associated with $V$, for $t\in\R$. We say that $V$
satifies the torsion condition when
for all $t\geq 0$ and for all $x\in\,]0,1[$, if $y_1$ and $y_2$ are in $]0,\de]$ and satisfy $y_1<y_2$,
and if we set $\ph_t(x,y_1)=(x_1(t),y_1(t))$ and
$\ph_t(x,y_2)=(x_2(t),y_2(t))$, then $x_1(t) \leq x_2(t)$ for all $t\geq 0$, with the stronger condition that $x_1(t) < x_2(t)$ for all $t\geq 0$
when $7/(8p)\leq x\leq 1-7/(8p)$ (and so the initial points are in th flat zone).
\end{Def}

To introduce the second condition we first need to define what we call the separation function for two points on the same orbit of
an asymptotic model.
Consider an orbit $\Ga$ of $V$ and fix two points $a$ and $a'$ on $\Ga$. We set $\ph_t(a)=(x(t),y(t))$ and
$\ph_t(a')=(x'(t),y'(t))$. Then we define the separation of $a$ and $a'$ as the function
$$
E_{a,a'}(t)=x'(t)-x(t),
$$
so $E_{a,a'}$ is $C^\infty$, has constant sign, and $E_{a,a'}\to 0$ when $t\to\pm \infty$. We are interested in the behaviour
of the maxima of $E_{a,a'}$.

A fundamental domain for the system will be a subset $\jK$ of $\bS$ limited by a vertical segment $\Delta$ on the left and by
its image $\ph(\Delta)$ on the right. A fundamental domain on an orbit is the part of the orbit limited by a point and its image
by the time-one flow.

\begin{Def}{\bf (Tameness).}  Consider an asymptotic model $V$ on the strip $\bS$.
We say that $V$ is {\em tame} when there exists a fundamental domain $\K$, contained in the flat zone, such that if $a$ and $a'$ are on the
same orbit of $V$ and are contained in a fundamental domain on this orbit, then for each $t_0$ such that $E_{a,a'}(t_0)$ is maximum,
the points $\ph_{t_0}(a)$ and $\ph_{t_0}(a')$ are located inside the domain $\K$.
\end{Def}

\subsection{The complexity indices of asymptotic models}
We are now in a position to prove the following result.

\begin{thm}\label{thm:indgrad}
Consider a tame $p$--model with torsion on the strip $\bS$,  and let $\ph$ be its time-one map.
Then $\L(\ph)=2$.
\end{thm}

\begin{proof} We will first prove that $\L(\ph)\leq 2$ by exhibiting suitable coverings
of the strip, and then that $\L(\ph)\geq 2$ by finding separated sets. Let us introduce some notation.

\vskip1mm\noindent
-- Given $\eps>0$, we denote by $\Sig_\al(\eps), \Sig_\om(\eps)$ the vertical segments contained in the strip $\bS$ with equations
$x=\eps$ and $x=1-\eps$ respectively.
\vskip1mm\noindent
-- Given $y\in\,]0,\de]$, we denote by $\Ga_y$ the orbit whose ordinate in the flat zone of $\bS$ is $y$.
\vskip1mm\noindent
-- Given $0<eps<1/(8p)$, we denote by $\tau^{(\eps)}(y)$, or $\tau(y)$ for short, the time needed to go from the segment $\Sig_\al(\eps)$
to the segment $\Sig_\om(\eps)$ on the orbit $\Ga_y$, we say that $\tau(y)$ is the transition time on $\Ga_y$.
\vskip1mm\noindent
-- Due to the torsion condition, $\tau$ is a decreasing function
of $y$. So one can also label the orbits by their transition times: we write $C_\tau$ the orbit with transition
time $\tau$, so $C_{\tau(y)}:=\Ga_y$. Of course this labelling depends on the choice of $\eps$, the context will be clear enough
in the following.
\vskip1mm\noindent
-- We write $\Ga_0$ or $C_{\infty}$ for the boundary $y=0$.
\vskip1mm\noindent
-- Given two transition times $\tau$ and $\tau'$ with $\tau<\tau'\leq+\infty$, we denote by $S_{\tau',\tau}$ the strip
bounded  by the curves $C_{\tau'}$ and $C_{\tau}$.

\vskip2mm

{\bf 1. Proof of $\L(\ph)\leq 2$.}
Given $\eps>0$ and an integer $N$, we want to construct a covering of $\bS$ by subset of $d_N$--dimater less than $\eps$.
We will have to discriminate between two different regimes for the system: close enough to the boundary $\Ga_0$ the $N$ first iterates
roughly behave as those of the model on a segment (proposition \ref{prop:indseg}), while on the complement one has to take into account the drift between nearby orbits.
So we will split  $\bS$ into two $N$--depending strips and separately construct coverings for these two domains.

\vskip2mm
\noindent
{\bf 1.a.  Choice of suitable domains.} The following lemma in analogous to lemma \ref{lem:choice} and
will enable us to construct these strips.

\vskip1mm

\begin{lemma}   Fix $\eps>0$ small enough. For $k\in\{0,\ldots,p\}$
we denote by $\De_k^\mp$ the vertical segments contained in the strip $\bS$ of equations
$x=k/p\pm\eps/2$.
Let  $\B_k$ be the ``block'' of $\bS$ limited by the
segments $\De^+_k$ and $\De^-_k$.
Then,  there exists two integers $\ka$ and $N_0$  (both depending on $\eps$) such that if $N\geq N_0$,  for each index $k\in\{0,\ldots,p\}$:
$$
\ph^n(\De^+_k(\ka N))\subset \B_k,\qquad \forall n\in\{1,\ldots,N\},
$$
where we write $\De^+_k(q)$ for the intersection of the left vertical $\De^-_k$ of $\B_k$ with the strip $S_{\infty,q}$.
\end{lemma}

The proof is  essentially the same as that of lemma \ref{lem:choice}.
As in the previous section, the transition time $\ka N$ is the natural cutoff for the system on the timescale $N$. So we introduce the two domains
$$
\bS_{N}=S_{\infty,\ka N},\quad \bS^*_{N}=\bS\setm \bS_{N},
$$
for which we will separately construct adapted coverings.

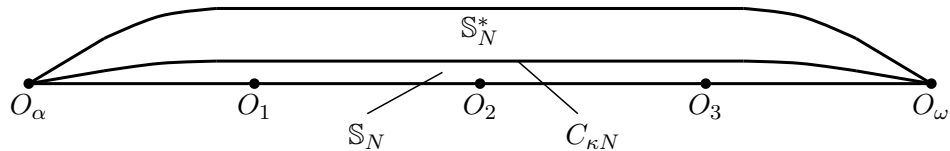
\begin{figure}[h]
\begin{center}
\begin{pspicture}(4cm,1.5cm)
\psset{xunit=1cm,yunit=1cm}
\rput(2,0){
\psline[linewidth=.4mm](-6,0)(6,0)
\psline[linewidth=.4mm](-3.5,1)(3.5,1)
\psline[linewidth=.4mm](-6,0)(-5,.6)
\pscurve[linewidth=.4mm](-5,.6)(-4.25,.9)(-3.5,1)
\psline[linewidth=.4mm](6,0)(5,.6)
\pscurve[linewidth=.4mm](5,.6)(4.25,.9)(3.5,1)
\psline[linewidth=.4mm](-3.5,.3)(3.5,.3)
\pscurve[linewidth=.4mm](5,.17)(4.25,.27)(3.5,.3)
\psline[linewidth=.4mm](5,.17)(6,0)
\pscurve[linewidth=.4mm](-5,.17)(-4.25,.27)(-3.5,.3)
\psline[linewidth=.4mm](-5,.17)(-6,0)
\pscircle[fillstyle=solid,fillcolor=black](-6,0){.07}
\pscircle[fillstyle=solid,fillcolor=black](-3,0){.07}
\pscircle[fillstyle=solid,fillcolor=black](0,0){.07}
\pscircle[fillstyle=solid,fillcolor=black](3,0){.07}
\pscircle[fillstyle=solid,fillcolor=black](6,0){.07}
\rput(-6,-.3){$O_\al$}
\rput(-3,-.3){$O_1$}
\rput(0,-.3){$O_2$}
\rput(3,-.3){$O_3$}
\rput(6,-.3){$O_\om$}
\rput(0,.7){$\bS^*_N$}
\rput(-1.5,-.7){$\bS_N$}
\psline[linewidth=.2mm](-1.5,-.4)(-.5,.15)
\rput(1.5,-.7){$C_{\ka N}$}
\psline[linewidth=.2mm](1.3,-.4)(.5,.3)
}
\end{pspicture}
\end{center}
\vskip1cm
\end{figure}

\noindent
{\bf 1.b. Covering of the strip $\bS_N$.} In this domain, we reproduce with slight modifications the arguments of proposition \ref{prop:indseg}. This easily yields the existence of $c_\eps>0$ such that
$$
G_N(\bS_{N},\eps)\leq c_\eps N.
$$

\vskip2mm
\noindent
{\bf 1.c. Covering of the domain $\bS^*_{N}$.}
The main step consists in estimating from above the numbers $G_N( S_{\tau,\tau'},\eps)$ for
thin enough strips $S_{\tau,\tau'}$.

\begin{lemma} Let $m$ be a fixed positive integer, and let $\eps>0$ be given. There exists positive constants $c_1$  and $c_2$, depending only
on $\eps$,
such that if  the pair $(\tau,\tau')\in [\tau^* ,\infty[^2$ satisfies
$$
0\leq \tau-\tau' \leq c_1\,\eps
$$
then the strip $S_{\tau,\tau'}$ satisfies
$$
G_m(S_{\tau,\tau'},\eps)\leq c_2\, (m+2\tau).
$$
\end{lemma}

\begin{proof}  We will first analyze the dynamics on a single curve $C_\tau$, and then deduce from this study an estimate of the covering
number for a thin enough strip $S_{\tau,\tau'}$. We denote by $\norm{\ }$ the Sup norm on $\R^2$.

\vskip2mm

1.  Fix $\eps>0$ small enough. We let $\K$ be  the fundamental domain of the tameness condition, contained in the flat zone. We also introduce the Lipschitz constant $\la$ of the flow $\Phi$ of $V$ on the set $[-1,1]\times \bS$.
Given a transition time $\tau\in[\tau^*,\infty[$, we denote by $\K_\tau$ the intersection of $\K$ with the orbit $C_\tau$. Then if $J=[a,a']$ is a subinterval of
$\K_\tau$:
$$
\Sup_{t\in[-1,1]}E_{a,a'}(t)\leq \la \norm{a-a'}.
$$
and therefore, by the tameness condition,
$$
\Sup_{t\in\R}E_{a,a'}(t)\leq \la \norm{a-a'}.
$$
Finally, due to the shape of the orbits,  there exists a constant $c>0$ such that for all $t\in\R$, $\diam\,\ph_t(J)\leq c E_{a,a'}(t)$, therefore
for all $n\in\N$, $\diam_n J\leq c\la \norm{a-a'}$. As a consequence
$$
\forall k\in \Z,\quad \diam_n(\ph^k(J))\leq c\la \norm{a-a'}.
$$

Now, thanks to the torsion condition,  there exists a fixed integer $j^*$, independent of $\tau$ such that for all $\tau\in [\tau^*,+\infty[$ it is possible to find a covering of the
interval $\K_\tau$ by subintervals $J_1^{(\tau)},\ldots,J_{j^*}^{(\tau)}$ of diameter $\leq \eps/(2c\la)$. It is then easy to see that the iterates
$$
J_{j,n}^{(\tau)}=\ph^n(J_j^{(\tau)}),\quad n\in \{-m-[\tau]-1,\ldots,[\tau]+1\},\  j\in\{1,\ldots,j^*\},
$$
together with the interval $[O_\al,\ph^{-m-\tau-1}(a_\tau)]$ form a covering of the orbit $C_\tau$, and that all these sets have $m$--diameter
$\leq\eps/2$. So
$$
G_m(C_\tau,\eps/2)\leq j^*(m+2\tau+2).
$$

\vskip2mm

2. We now use the previous covering of a curve and fatten it a little to obtain a covering of a thin strip. Namely, given the initial
transition time $\tau$, we want to find a transition time $\tau'\leq \tau$ such that for any pair of points $a\in C_\tau$ and $a'\in C_{\tau'}$ with the same abscissa, the horizontal separation between any pair of iterates $\ph^n(a)$ and $\ph^n(a')$, $n\in\{0,\ldots,m\}$, is at most $\eps/2c$.
Assume that it is the case.  One then constructs a covering of the strip $S_{\tau,\tau'}$ by considering the previous covering
$$
J_{j,n}^{(\tau)}=[a_{jn},b_{jn}]
$$
of the curve $C_\tau$ and introducing the ``rectangles''  $R_{kl}$ limited by the curves $C_\tau$ and $C_{\tau'}$ and the vertical lines passing through
the points $a_{jn}$ and $b_{jn}$. One easily checks that the family $(R_{jn})$ is a covering of $S_{\tau,\tau'}$ by $j^*(m+2[\tau]+2)$
subsets with $m$--diameter less than $\eps$.

Son one only has to find a suitable time $\tau'<\tau$.
To estimate the mutual drift of the points on $C_\tau$ and $C_{\tau'}$, we will use the torsion condition.
We fix a point $a\in C_\tau$ and a time $t\geq 0$, and we write $a'$ for the point on $C_{\tau'}$ with the same
abscissa as $a$. We define the {\em time separation} as the difference $T(a,t)=t-t'$, where $t'$ is the time needed for $a'$ to pass through
the vertical over $\ph_t(a)$, more precisely, setting $\ph_t(a)=(x(t),y(t))$ and $\ph_t(a')=(x'(t),y'(t))$,  the time $t'$ is defined by the equality
$$
x'(t')=x(t).
$$
By the torsion condition, $T(a,t)\geq0$ for all $t\geq 0$. Moreover, one easily checks that
$$
T(a,t_1+t_2)=T(a,t_1)+T(\ph_{t_1}(a),t_2).
$$
We will have to use an upper bound of $T(a,t)$ for $0\leq t\leq t^*$, where $t^*(a)$ is defined by
$\ph_{t^*}(a)\in \Sig_\om(\eps)$.
Remark first that if $a\in\Sig_\al(\eps)\cap C_\tau$ then, by definition of the transition time, $t^*(a)=\tau$ and
$$
T(a,t)\leq T(a,t^*(a))= \tau-\tau',\qquad  0\leq t\leq t^*(a).
$$
If $b=\ph_s(a)$ with $a$ as above and $\tau\geq s\geq 0$, then $t^*(b)=\tau-s$ and
$$
T(b,t)\leq T(b,t^*(b))=T(a,\tau)-T(a,t)\leq T(a,\tau)=\tau-\tau',\qquad 0\leq t\leq t^*(b).
$$
Finally, note that due to the normal form in $O_\al$, the image by the flow of a vertical segment in this domain is still vertical.
Therefore if $c\in O_\al\cap C_\tau$ and if $t\geq 0$ is such that $\ph_t(c)\in O_\al$
$
T(c,t)=0.
$
As a consequence, if $c=\ph_{-s}(a)$ with $s\geq 0$, then $t^*(c)=\tau+s$ and
$$
T(c,t)\leq T(c,t^*(c))=T(c,s)+T(a,\tau)=T(a,\tau)=\tau-\tau',\qquad 0\leq t\leq t^*(c).
$$

Now, if $\ell$ is the maximal length of the vector field $V$,  the distance between $\ph_t(a)$ and $\ph_{t}(a')$ clearly satisfies
$$
\norm{\ph_t(a)-\ph_{t}(a')}\leq c\, \ell\,T(a,t), \quad \forall t\geq 0.
$$
Therefore, for all $a\in C_\tau$ and for $0\leq t\leq t^*(a)$,
$$
\norm{\ph_t(a)-\ph_{t}(a')}\leq c\, \ell\,(\tau-\tau')\leq \eps/2
$$
as soon as
$$
\tau-\tau'\leq c_1\,\eps,\qquad  c_1=\Frac{1}{2 c\,\ell}.
$$
Finally, when $t\geq t^*(a)$, both points $a$ and $a'$ are on the right of $\Sig_\om(\eps)$, so their distance is less than $\eps/2$.
Gathering these estimates with the construction of the rectangles $R_{jn}$, one sees that our statement is proved for $c_2=2j^*$
(for $m$ large enough). \end{proof}

To conclude the proof of $\L(\ph)\leq2$, we now pick out a family of transition times $(\tau_\ell)_{1\leq \ell\leq\ell^*}$ such that
$0\leq\tau_{\ell+1}- \tau_\ell\leq c_1\eps$ and $\tau_1=\tau^*$, $\tau_{\ell^*}=\ka\,N$. Therefore one can choose:
$$
\ell^* \leq \Frac{\ka\,N-\tau^*}{c_1\,\eps}+1.
$$
We can apply the previous lemma to each strip $S_{\tau{\ell+1},\tau_\ell}$ and get a covering of the strip $\bS^*_N$ by subsets
of $d_N$--diameter less than $\eps$. One sees that this covering has
less than
$$
c_2(N+ 2\tau_{k+1})\leq c_2(1+2\ka)N
$$
 elements.
 The union of these coverings for $1\leq \ell\leq \ell^*$ form a covering
of the strip $\bS^*$ by subsets
of $d_N$--diameter less than $\eps$. This proves the existence of a constant $C_\eps$, depending only on $\eps$, such that
$$
G_N(\bS^*_N,\eps)\leq C_\eps N^2.
$$
Finally
$$
G_N(\bS,\eps)\leq G_N(\bS_N,\eps)+G_N(\bS^*_N,\eps)\leq c_\eps N+C_\eps N^2
$$
which concludes the proof.

\vskip3mm

{\bf 2. Proof of $\L(\ph)\geq 2$.} Given $\eps$ small enough and $N$ large enough, we will now construct an explicit $(N,\eps)$--separated subset of $\bS$. We fix a vertical segment $\Sig^-$ in the flat zone, with abscissa $\leq 1/p$, and a vertical segment $\Sig^+$ in the flat
zone, with abscissa $\geq 1-1/p$. We denote by $\ha \tau(y)$ the transition time between these two sections on the orbit $\Ga_y$, so
$\ha\tau$ is a decreasing function of $y$ and $\ha\tau(y)\to+\infty$ when $y\to 0$. We denote by $a_\tau$ the intersection point
of $\Sig^-$ with the orbit $C_\tau$. We denote by $a^\pm$ the intersection point of $\Sig^\pm$ with the boundary $C_\infty$.

\begin{lemma} Assume that $\eps<\Min(\norm{a-\ph(a^+)},\norm{\ph(a^-)-\ph^2(a^-)})$. Then the subset
$$
A=\{\ph^{-\ell}(a_{2n})\mid (\ell,n)\in\N^2,\ 0\leq \ell\leq N/2,\ N/3\leq n\leq N/2\}
$$
is $(N,\eps)$ separated.
\end{lemma}

\proof We begin by proving that the subset of all points of $A$ located on the same curve
$C_{2n}$ is $(N,\eps)$--separated. Indeed, this is true for each subset of the form
$$
D_\tau=\{\ph^{-\ell}(a_\tau)\mid 0\leq \ell\leq N-1\}.
$$
To see this, remark first that due to the torsion condition:
$$
\norm{a_\tau-\ph(a_\tau)} >\eps.
$$
Let $0\leq\ell<\ell'\leq N-1$. Then
$$
\ph^{\ell'}(\ph^{-\ell'}(a_\tau))=a_\tau,\quad \ph^{\ell'}(\ph^{-\ell}(a_\tau))=\ph^{\ell'-\ell}(a_\tau)
$$
and since $\ell'-\ell\geq1$, the point $\ph^{\ell'-\ell}(a_\tau)$ is on the right of $\ph(a_\tau)$ on the orbit $C_\tau$, so
$\norm{\ph^{-\ell'}(a_\tau)-\ph^{-\ell}(a_\tau)}>\eps$. This
proves that
$$
d_N^\ph(\ph^{-\ell'}(a_\tau),\ph^{-\ell'}(a_\tau))>\eps
$$
and therefore $D_\tau$ is $(N,\eps)$ separated. Now $A_{2n}:=C_{2n}\cap A\subset D_{2n}$, so $A_{2n}$ is $(N,\eps)$
separated.

Now we prove that if $x\in A_{2n}$ and $x'\in A_{2n'}$ with $n\neq n'$, then $d_N^\ph(x,x')>\eps$. So let
$x=\ph^{-\ell}(a_{2n})$ and $x'=\ph^{-\ell'}(a_{2n'})$, with $0\leq\ell\leq\ell'\leq N/2$.
\vskip1mm\noindent
-- Assume first that $\ell\neq\ell'$. Then by the same argument as above one easily sees that
$
d_N^\ph(\ph^{-\ell'}(a_\tau),\ph^{-\ell'}(a_\tau))>\eps.
$
\vskip1mm\noindent
-- Assume that $\ell\neq\ell'$. Then $\ph^\ell(x)=a_{2n}$ and $\ph^\ell(x')=a_{2n'}$. Let $\nu,\nu'$ be the (only) integers such that
$\ph^\nu(x)$ and $\ph^{\nu'}(x')$ are in the fundamental domain $\K^-$. Obviously $\nu\geq 2n$, and since $2n-2n'\geq 2$,
one sees that $\nu'\leq \nu-2$. So
$$
\ph^\nu(a_{2n'})\in\ph^2(\K^-)
$$
which due to the assumption on $\eps$ yields $\norm{\ph^\nu(a_{2n'})-\ph^\nu(a_{2n})}>\eps$.
Notice finally that $\ell+\nu\leq N$, therefore $d_N^\ph(x,x')>\eps$.

\vskip2mm

So any pair of points in $A$ is $(N,\eps)$ separated, wich proves our statement.
\end{proof}

Now $\# A\geq c N^2$, so $\L(\ph)\geq 2$. Therefore $\L(\ph)=2$.



\end{document}